\documentclass[letter,
fontsize=11pt,%
oneside,%
numbers=enddot]{scrartcl}
\KOMAoptions{DIV=14}    
\usepackage[utf8]{inputenc}
\usepackage[T1]{fontenc}
\usepackage{graphicx}
\usepackage{textcomp}
\usepackage[fleqn]{amsmath}
\usepackage{amsthm}
\usepackage{amssymb}
\usepackage{amsfonts}
\usepackage{accents}
\usepackage{soul}
\usepackage{pifont}        
\usepackage{libertine}
\usepackage[libertine,
slantedGreek,
nosymbolsc,
nonewtxmathopt,
subscriptcorrection]{newtxmath}
\usepackage[scaled=0.95,varqu,varl]{inconsolata}
\usepackage[scr=boondox]{mathalpha}   
\usepackage{euscript}   
\usepackage{leftindex}
\allowdisplaybreaks[1]  
\numberwithin{equation}{section}    
\usepackage[svgnames,hyperref]{xcolor}
\definecolor{orng}{HTML}{F35400}
\definecolor{bleu}{HTML}{BCE6F2}
\definecolor{dblue}{HTML}{0455BF}
\definecolor{dgreen}{HTML}{02724A}
\definecolor{dgreen2}{HTML}{025951}
\definecolor{dred}{HTML}{D90404}
\definecolor{dviolet}{HTML}{42208C}
\definecolor{labelkey}{HTML}{025951}
\definecolor{refkey}{HTML}{025951}
\definecolor{refkey}{rgb}{0,0.6,0.0}
\definecolor{Brown}{rgb}{0.45,0.0,0.05}
\definecolor{dgreen}{rgb}{0.00,0.49,0.00}
\definecolor{dblue}{rgb}{0,0.18,0.75}
\definecolor{lblue}{rgb}{0,0.7,0.75}
\definecolor{dviolet}{HTML}{9400D3}
\definecolor{pblue}{rgb}{0.1176,0.5647,1}
\definecolor{nblue}{rgb}{0.2,0.3,1}
\definecolor{pgreen}{rgb}{0.1961,0.8039,0.1961}
\definecolor{ngreen}{rgb}{0.0,0.6,0.3}
\definecolor{pred}{rgb}{1.0,0.2706,0.0}
\definecolor{magenta}{HTML}{ff00ff}
\definecolor{hotmagenta}{rgb}{1.0, 0.11, 0.81}
\definecolor{dorng}{rgb}{0.91,0.41,0.17}
\definecolor{dgray}{rgb}{0.41,0.41,0.41}
\definecolor{azure}{rgb}{0.0, 0.5, 1.0}
\usepackage{tikz,tkz-euclide,pgfplots}
\usetikzlibrary{arrows,decorations}
\addtokomafont{section}{\centering}

\usepackage{enumitem}
\setlist{itemsep=-2.0pt}

\usepackage{upref}  
\usepackage{hyperref}
\hypersetup{colorlinks=true,
linktocpage=true,
linkcolor=dblue,
citecolor=dgreen,
urlcolor=dred,
pdfencoding=auto,
hypertexnames=false}
\makeatletter
\g@addto@macro\th@plain{
\thm@headfont{\bfseries\sffamily}
\thm@notefont{}}
\g@addto@macro\th@definition{
\thm@headfont{\bfseries\sffamily}
\thm@notefont{}}
\g@addto@macro\th@remark{
\thm@headfont{\bfseries\sffamily}
\thm@notefont{}}
\makeatother
\theoremstyle{plain}
\newtheorem{theorem}{Theorem}[section]
\newtheorem{proposition}[theorem]{Proposition}
\newtheorem{corollary}[theorem]{Corollary}
\newtheorem{lemma}[theorem]{Lemma}

\theoremstyle{definition}
\newtheorem{definition}[theorem]{Definition}
\newtheorem{example}[theorem]{Example}
\newtheorem{problem}[theorem]{Problem}
\newtheorem{assumption}[theorem]{Assumption}
\theoremstyle{remark}
\newtheorem{remark}[theorem]{Remark}

\newtheorem{algorithm}[theorem]{Algorithm}

\usepackage{epsfig}
\usepackage[usenames,dvipsnames]{pstricks}
\usepackage{pstricks-add}
\usepackage{pst-grad}
\usepackage{pst-plot} 
\usepackage{pst-node} 
\usepackage{pst-eucl} 
\usepackage{pst-coil} 
\usepackage[extdef=true]{delimset}
\DeclareMathDelimiterSet{\scal}[2]{
\selectdelim[l]<{#1}
\mathpunct{}\selectdelim[p]|
{#2}\selectdelim[r]>}
\DeclareMathDelimiterSet{\EC}[2]{
\mathsf{E}\selectdelim[l]({#1}
\mathpunct{}\selectdelim[p]|
{#2}\selectdelim[r])}

\newcommand{\menge}[2]{\bigl\{{#1}\mid{#2}\bigr\}} 
\DeclareMathDelimiterSet{\Menge}[2]{\selectdelim[l]\{
{#1}\selectdelim[m]|{#2}\selectdelim[r]\}}

\makeatletter
\def\upintkern@{\mkern-7mu\mathchoice{\mkern-3.5mu}{}{}{}}
\def\upintdots@{\mathchoice{\mkern-4mu\@cdots\mkern-4mu}%
{{\cdotp}\mkern1.5mu{\cdotp}\mkern1.5mu{\cdotp}}%
{{\cdotp}\mkern1mu{\cdotp}\mkern1mu{\cdotp}}%
{{\cdotp}\mkern1mu{\cdotp}\mkern1mu{\cdotp}}}
\makeatother
\DeclareFontFamily{OMX}{mdbch}{}
\DeclareFontShape{OMX}{mdbch}{m}{n}{ <->s * [0.8]  mdbchr7v }{}
\DeclareFontShape{OMX}{mdbch}{b}{n}{ <->s * [0.8]  mdbchb7v }{}
\DeclareFontShape{OMX}{mdbch}{bx}{n}{<->ssub * mdbch/b/n}{}
\DeclareSymbolFont{uplargesymbols}{OMX}{mdbch}{m}{n}
\SetSymbolFont{uplargesymbols}{bold}{OMX}{mdbch}{b}{n}
\DeclareMathSymbol{\upintop}{\mathop}{uplargesymbols}{82}
\DeclareMathSymbol{\upointop}{\mathop}{uplargesymbols}{"48}
\makeatletter
\renewcommand{\int}{\DOTSI\upintop\ilimits@}
\renewcommand{\oint}{\DOTSI\upointop\ilimits@}
\makeatother
\newcommand{\MM}{\mathscr{m}}
\newcommand{\LL}{\mathscr{l}}
\newcommand{\CC}{\mathscr{c}}

\newcommand{\bunder}[2][4]{\mathrlap{\mkern\the\numexpr#1/2mu\relax\underline{\phantom{\mathrm{#2}\mkern-#1mu}}}#2}
\newcommand{\RR}{\mathbb{R}}
\newcommand{\NN}{\mathbb{N}}
\newcommand{\XX}{\EuScript{X}}

\newcommand{\WS}{\mathsf{W}}
\newcommand{\MS}{\mathsf{M}}
\newcommand{\AS}{\mathsf{A}}
\newcommand{\BS}{\mathsf{B}}
\newcommand{\CS}{\mathsf{C}}
\newcommand{\DS}{\mathsf{D}}
\newcommand{\RS}{\mathsf{R}}
\newcommand{\QS}{\mathsf{Q}}

\newcommand{\HHS}{\ensuremath{\boldsymbol{\mathsf H}}}

\newcommand{\GGS}{\ensuremath{{\boldsymbol{\mathsf G}}}}

\newcommand{\HS}{\mathsf{H}}
\newcommand{\GS}{\mathsf{G}}
\newcommand{\ZS}{\mathsf{Z}}
\newcommand{\zS}{\mathsf{z}}
\newcommand{\nS}{{\mathsf{n}}}
\newcommand{\mS}{{\mathsf{m}}}
\newcommand{\pS}{{\mathsf{p}}}
\newcommand{\nnn}{\mathsf{n}\in\mathbb{N}}
\newcommand{\jjj}{\mathsf{j}\in\mathbb{N}}
\newcommand{\II}{\ensuremath{\mbox{\large\ttfamily{I}}}}
\newcommand{\KK}{\ensuremath{\mbox{\large\ttfamily{K}}}}
\newcommand{\iii}{\ensuremath{\mathsf{i}\in%
\mbox{\scriptsize\ttfamily{I}}}}
\newcommand{\jJj}{\ensuremath{\mathsf{j}\in%
\mbox{\scriptsize\ttfamily{I}}}}
\newcommand{\jjJ}{\ensuremath{\mathsf{j}\in%
\mbox{\scriptsize\ttfamily{K}}}}
\newcommand{\kkk}{\ensuremath{\mathsf{k}\in%
\mbox{\scriptsize\ttfamily{K}}}}
\newcommand{\iS}{\mathsf{i}}
\newcommand{\dS}{\mathsf{d}}
\newcommand{\jS}{\mathsf{j}}
\newcommand{\kS}{\mathsf{k}}
\newcommand{\xS}{\mathsf{x}}
\newcommand{\yS}{\mathsf{y}}
\newcommand{\sS}{\mathsf{s}}
\newcommand{\vS}{\mathsf{v}}
\newcommand{\rS}{\mathsf{r}}
\newcommand{\lS}{\mathsf{l}}
\newcommand{\LS}{\mathsf{L}}
\newcommand{\LLS}{\boldsymbol{\mathsf{L}}}
\newcommand{\BE}{\EuScript{B}}
\newcommand{\FE}{\EuScript{F}}

\newcommand{\sad}{\bunder{\boldsymbol{\EuScript{S}}}}

\newcommand{\XXX}{\ensuremath{\bunder{\boldsymbol{\mathsf{X}}}}}

\newcommand{\pinf}{{+}\infty}
\newcommand{\minf}{{-}\infty}
\newcommand{\zeroun}{\intv[o]{0}{1}}

\newcommand{\RXX}{\intv{\minf}{\pinf}}
\newcommand{\RX}{\intv[l]0{\minf}{\pinf}}
\newcommand{\RP}{\intv[r]0{0}{\pinf}}
\newcommand{\RPP}{\intv[o]0{0}{\pinf}}

\newcommand{\emp}{\varnothing}

\newcommand{\WC}{\ensuremath{{\mathfrak W}}}
\newcommand{\SC}{\ensuremath{{\mathfrak S}}}
\newcommand{\Sum}{\displaystyle\sum}

\newcommand{\minimize}[2]{\underset{\substack{{#1}}}
{\operatorname{minimize}}\;\;#2}
\newcommand{\infconv}{\mathbin{\mbox{\small$\square$}}}

\newcommand{\pushfwd}%
{\ensuremath{\mbox{\Large$\,\triangleright\,$}}}

\DeclareMathOperator{\argmin}{argmin}
\newcommand{\Id}{\mathsf{Id}}

\DeclareMathOperator{\card}{card}

\DeclareMathOperator{\dom}{dom}

\DeclareMathOperator{\Fix}{Fix}
\DeclareMathOperator{\gra}{gra}
\DeclareMathOperator{\zer}{zer}

\DeclareMathOperator{\reli}{ri}

\DeclareMathOperator{\prox}{prox}
\DeclareMathOperator{\proj}{proj}

\newcommand{\EE}{\mathsf{E}}
\newcommand{\PP}{\mathsf{P}}

\renewcommand{\leq}{\leqslant}
\renewcommand{\geq}{\geqslant}
\newcommand{\exi}{\exists\,}
\newcommand{\weakly}{\rightharpoonup}

\newcommand{\Pas}{\text{\normalfont$\PP$-a.s.}}
\newcommand{\Pto}{\ensuremath{\overset{\mathsf{P}}{\to}}}

\renewenvironment{abstract}{%
\vspace*{-0.50cm}
\small
\quotation%
\noindent%
{\normalfont\bfseries\sffamily
\nobreak\abstractname\ }%
}{%
\endquotation%
\medskip
}
\renewcommand{\abstractname}{Abstract.}
\newcommand\keywordsname{Keywords.}
\newenvironment{keywords}
{\renewcommand\abstractname{\keywordsname}\begin{abstract}}
{\end{abstract}}

\usepackage[auth-sc]{authblk}
\newcommand{\email}[1]{\href{mailto:#1}{\nolinkurl{#1}}}
\renewcommand*\Affilfont{\normalfont\normalsize}
\newcommand\affilcr{\protect\\ \protect\Affilfont}
\makeatletter
\renewcommand\AB@affilsepx{\protect\\[0.5em]}
\makeatother

\author[1]{Patrick L. Combettes}
\affil[1]{North Carolina State University
\affilcr
Department of Mathematics
\affilcr
Raleigh, NC 27695, USA
\affilcr
\email{plc@math.ncsu.edu}
}
\author[2]{Javier I. Madariaga}
\affil[2]{North Carolina State University
\affilcr
Department of Mathematics
\affilcr
Raleigh, NC 27695, USA
\affilcr
\email{jimadari@ncsu.edu}
}

\begin{document}

\title{ 
Asymptotic Analysis of an Abstract Stochastic Scheme for
Solving Monotone Inclusions\thanks{Contact author: 
P. L. Combettes. Email: \email{plc@math.ncsu.edu}.
Phone: +1 919 515 2671.
This work was supported by the National
Science Foundation under grant DMS-2513409.
}}

\date{~}

\maketitle

\thispagestyle{empty}
\begin{abstract} 
We propose an abstract stochastic scheme for solving a broad class
of monotone operator inclusion problems in Hilbert spaces. It is
based on a reduced inclusion model with two operators. At each
iteration of the abstract algorithm, stochastic approximations of
points in the graphs of these operators are used to construct the
update. We establish the almost sure and $L^2$ convergence of the
iterates to a solution of the underlying inclusion problem. 
A key feature of the proposed framework is its flexibility
as it can accommodate stochasticity at several levels: the
approximation of operators, the selection of coordinates and
operators in block-iterative implementations, and the choice of
relaxation parameters. This makes it well suited for the asymptotic
analysis of a wide range of stochastic monotone operator splitting
methods. The broad scope of the proposed framework is illustrated
by applications to various splitting methods. Thus, we derive new
results on the almost sure and $L^2$ convergence of stochastic
versions of the proximal point algorithm, as well as of randomized
block-iterative Kuhn--Tucker and saddle projective splitting
methods for solving systems of coupled inclusions involving a mix
of set-valued, cocoercive, and Lipschitzian monotone operators
combined via various monotonicity-preserving operations. Specific
applications are discussed.
\end{abstract}



\setcounter{page}{0}

\newpage

\section{Introduction}
\label{sec:1}
The object of the present paper is to study the asymptotic behavior
of an abstract stochastic scheme for solving a broad class of 
monotone inclusion problems in Hilbert spaces. As in the
deterministic methods unified in \cite{Acnu24}, our analysis is
articulated around the following two-operator abstract model.

\begin{problem}
\label{prob:19}
Let $\HS$ be a separable real Hilbert space, let 
$\WS\colon\HS\to 2^{\HS}$ be maximally monotone, let 
$\upalpha\in\RPP$, and let $\mathsf{C}\colon\HS\to\HS$ be 
$\upalpha$-cocoercive and such that 
$\ZS=\zer(\WS+\mathsf{C})\neq\emp$. The task is to 
\begin{equation}
\text{find}\;\xS\in\HS\;\;\text{such that}\;\;{0}\in
\WS\xS+\CS\xS.
\end{equation}
\end{problem}

If the resolvent of $\WS$ were numerically tractable,
Problem~\ref{prob:19} could be solved via the classical
forward-backward algorithm \cite{Jmaa15,Merc79,Tsen91}.
However, in the general inclusion models to be considered, $\WS$ 
is typically a composite operator defined on a product space, 
which makes such an assumption unrealistic. Instead, we merely 
assume the ability to pick points in the graph of $\WS$.
This leads us to the following deterministic algorithmic 
template from \cite[Section~4.4]{Acnu24}, which was first 
considered in \cite[Proposition~3]{Sadd22} in the context 
of saddle projective splitting methods.

\begin{algorithm}
\label{algo:19}
In the setting of Problem~\ref{prob:19}, let $\xS_0\in\HS$ and 
iterate
\begin{equation}
\label{e:a19}
\begin{array}{l}
\text{for}\;\nS=0,1,\ldots\\
\left\lfloor
\begin{array}{l}
(\mathsf{w}_{\nS},\mathsf{w}_{\nS}^*)\in\gra\WS\;\text{and}\;
\mathsf{q}_{\nS}\in\HS\\
\mathsf{t}_{\nS}^*=\mathsf{w}_{\nS}^*+\CS \mathsf{q}_{\nS}\\
\upDelta_{\nS}=\scal{\xS_{\nS}-
\mathsf{w}_{\nS}}{\mathsf{t}_{\nS}^*}-
(4\upalpha)^{-1}\|\mathsf{w}_{\nS}-\mathsf{q}_{\nS}\|^2\\[2mm]
\uptheta_{\nS}=
\begin{cases}
\dfrac{\upDelta_{\nS}}
{\|\mathsf{t}_{\nS}^*\|^2},&\text{if}\:\:\upDelta_{\nS}>0;\\
0,&\text{otherwise}\\
\end{cases}\\
\dS_{\nS}=\uptheta_{\nS}\mathsf{t}_{\nS}^*\\
\uplambda_{\nS}\in\left]0,2\right[\\
\xS_{\nS+1}=\xS_{\nS}-\uplambda_{\nS}\dS_{\nS}.
\end{array} 
\right. 
\end{array} 
\end{equation}
\end{algorithm}

As shown in \cite{Acnu24}, Algorithm~\ref{algo:19} is at the core
of a broad range of classical and block-iterative deterministic
splitting methods, in particular those of
\cite{Bric18,Jmaa20,Sadd22,Cham11,Siop13,MaPr18,Opti14,Cond13,%
Davi17,Gise21,Ragu19,Spin83,Tsen91,Tsen00,Bang13}. Stochasticity
can be introduced in various components of these deterministic
algorithms: stochastic approximation of operators, random selection
of coordinates and operators in block-iterative implementations,
and random relaxation parameters. To design and analyze such
stochastic variants of existing models, we propose to transform
Algorithm~\ref{algo:19} into the following abstract stochastic
scheme.

\begin{algorithm}
\label{algo:20}
In the setting of Problem~\ref{prob:19}, let
$\uprho\in\left[2,\pinf\right[$, 
let $x_{0}\in L^2(\upOmega,\FE,\PP;\HS)$, and iterate
\begin{equation}
\label{e:a20}
\begin{array}{l}
\text{for}\;\nS=0,1,\ldots\\
\left\lfloor
\begin{array}{l}
\XX_{\nS}=\upsigma(x_{0},\dots,x_{\nS})\\
\text{let}\;
\{w_{\nS},w_{\nS}^*,e_{\nS},e_{\nS}^*\}\subset 
L^2(\upOmega,\FE,\PP;\HS)\;\text{be such that}\;
\brk1{w_{\nS}+e_{\nS},w_{\nS}^*+e_{\nS}^*}\in\gra\WS\;\Pas\\
\text{let}\;
\{q_{\nS},c_{\nS}^*,f_{\nS}^*\}\subset L^2(\upOmega,\FE,\PP;\HS)\;
\text{be such that}\;
c_{\nS}^*+f_{\nS}^*=\CS q_{\nS}\;\Pas\\
t_{\nS}^*=w_{\nS}^*+c_{\nS}^*\\
\Delta_{\nS}=\scal{x_{\nS}-w_{\nS}}{t_{\nS}^*}
-(4\upalpha)^{-1}\norm{w_{\nS}-q_{\nS}}^2\\[2mm]
\theta_{\nS}=\dfrac{\mathsf{1}_{[t_{\nS}^*\neq0]}
\mathsf{1}_{\left[\Delta_{\nS}>0\right]}\Delta_{\nS}}
{\norm{t_{\nS}^*}^2+\mathsf{1}_{[t_{\nS}^*=0]}}\\
d_{\nS}=\theta_{\nS}t_{\nS}^*\\
\lambda_{\nS}\in 
L^\infty(\upOmega,\FE,\PP;\left]0,\uprho\right])\\
x_{\nS+1}=x_{\nS}-\lambda_{\nS}d_{\nS}.
\end{array}
\right.\\
\end{array}
\end{equation}
\end{algorithm}

At iteration $\nS$ of Algorithm~\ref{algo:20}, the variables
$e_{\nS}$, $e_{\nS}^*$, and $f_{\nS}^*$ model stochastic errors
allowed in the activation of the operators $\WS$ and $\CS$. Thus,
the algorithm does not require an exact point in the graph of $\WS$
but merely a stochastic approximation $(w_{\nS},w_{\nS}^*)$ of such
a point. Likewise, it does not require the exact evaluation of $\CS
q_{\nS}$ but merely a stochastic approximation $c_{\nS}^*$ of it.
The broad reach of this algorithmic template stems from the
flexibility it offers in choosing the triple
$(w_{\nS},w_{\nS}^*,q_{\nS})$. Another notable new feature of
\eqref{e:a20} is the use of a random relaxation parameter
$\lambda_{\nS}$ which, furthermore, is not restricted to the usual
interval $\left]0,2\right[$. 

Notation and preliminary results are presented in
Section~\ref{sec:2}. The asymptotic behavior of
Algorithm~\ref{algo:20} is analyzed in Section~\ref{sec:3}, where
we prove in particular weak almost sure convergence to a solution
to Problem~\ref{prob:19} under suitable assumptions. Just as the
analysis of Algorithm~\ref{algo:19} provided a unifying framework
to establish the convergence of a broad class of classical and
block-iterative deterministic methods in \cite{Acnu24}, that of
Section~\ref{sec:3} can be used to derive stochastic versions of
these methods and obtain new results on their asymptotic behavior.
Thus, in Section~\ref{sec:4}, we establish new results on the
almost sure and $L^2$ weak convergence of the proximal point
algorithm with stochastic approximations of the resolvents and
random relaxations. To further illustrate the versatility of
Algorithm~\ref{algo:20}, we consider in Section~\ref{sec:5} a
drastically different model, namely, a multivariate composite
monotone inclusion problem. We propose a stochastic version of the
Kuhn--Tucker projective splitting method of \cite{MaPr18} which
combines three novelties: random relaxations, random selection of
blocks of variables and blocks of operators, and stochastic
approximations of the resolvent operators. To establish
convergence, our strategy is to show that it can be framed as an
instance of Algorithm~\ref{algo:20} in a suitable product space.
Finally, in Section~\ref{sec:6}, we consider a highly structured
multivariate monotone inclusion model involving a mix of
set-valued, cocoercive, and Lipschitzian monotone operators, as
well as linear operators, and various monotonicity-preserving
operations among them. We design a stochastic version of the
deterministic saddle projective splitting algorithm of
\cite{Sadd22} in which the blocks of variables and operators are
now selected randomly over the course of the iterations, and 
the relaxations are random. We define a suitable saddle operator in
a product space in which we transpose Algorithm~\ref{algo:20} to
establish, for the first time, the almost sure convergence of such 
a random block-iterative algorithm.

\section{Notation and preliminary results}
\label{sec:2}

\subsection{General notation}
\label{sec:21}
We use sans-serif letters to denote deterministic variables and
italicized serif letters to denote random variables. $\HS$ is a
separable real Hilbert space, with 
identity operator $\Id$, power set $2^{\HS}$,
scalar product $\scal{\cdot}{\cdot}$, and associated norm 
$\|\cdot\|$. The strong and weak convergence in $\HS$ are denoted
by the symbols $\to$ and $\weakly$, respectively. The sets of
strong and weak sequential cluster points of a sequence
$(\xS_{\nS})_{\nnn}$ in $\HS$ are denoted by
$\SC(\xS_{\nS})_{\nnn}$ and $\WC(\xS_{\nS})_{\nnn}$, respectively.
The reader is referred to \cite{Livre1} for background on convex
analysis and fixed point theory, and to \cite{Ledo91} for
background on probability theory.

\subsection{Operators}
\label{sec:22}
Let $\MS\colon\HS\to2^{\HS}$. The graph of $\MS$ is
$\gra\MS=\menge{(\xS,\xS^*)\in\HS\times\HS}{\xS^*\in\MS\xS}$ and
the set of zeros of $\MS$ is
$\zer\MS=\menge{\xS\in\HS}{{0}\in\MS\xS}$. The inverse of
$\MS$ is the operator $\MS^{-1}\colon\HS\to2^{\HS}$ with graph
$\gra\MS^{-1}=\menge{(\xS^*,\xS)\in\HS\times\HS}{\xS^*\in\MS\xS}$
and the resolvent of $\MS$ is $\mathsf{J}_{\MS}=(\Id+\MS)^{-1}$.
We say that $\MS$ is monotone if
\begin{equation}
\brk1{\forall(\xS,\xS^*)\in\gra\MS}
\brk1{\forall(\yS,\yS^*)\in\gra\MS}\quad
\scal{\xS-\yS}{\xS^*-\yS^*}\geq 0,
\end{equation}
and that it is maximally monotone if
\begin{equation}
\brk1{\forall(\xS,\xS^*)\in\HS\times\HS}\quad
\brk[s]1{(\xS,\xS^*)\in\gra\MS\Leftrightarrow
\brk1{\forall(\yS,\yS^*)\in\gra\MS}\;\;
\scal{\xS-\yS}{\xS^*-\yS^*}\geq 0}.
\end{equation}
If $\MS$ is maximally monotone, then $\mathsf{J}_{\MS}$ is a
single-valued operator defined on $\HS$ and which satisfies
\begin{equation}
\label{e:s6}
\Fix\mathsf{J}_{\MS}=\zer\MS\quad\text{and}\quad
(\forall\xS\in\HS)(\forall\yS\in\HS)\quad
\|\mathsf{J}_{\MS}\xS-\mathsf{J}_{\MS}\yS\|^2+
\|(\Id-\mathsf{J}_{\MS})\xS-(\Id-\mathsf{J}_{\MS})\yS\|^2\leq
\|\xS-\yS\|^2.
\end{equation}
Let $\upbeta\in\RPP$. Then $\MS$ is $\upbeta$-strongly monotone if
$\MS-\upbeta\Id$ is monotone, i.e.,
\begin{equation}
\brk1{\forall(\xS,\xS^*)\in\gra\MS}
\brk1{\forall(\yS,\yS^*)\in\gra\MS}\quad
\scal{\xS-\yS}{\xS^*-\yS^*}\geq\upbeta\norm{\xS-\yS}^2.
\end{equation}
The parallel sum of
$\BS\colon\HS\to2^{\HS}$ and $\DS\colon\HS\to2^{\HS}$ is
$\BS\infconv\DS=(\BS^{-1}+\DS^{-1})^{-1}$. An operator
$\CS\colon\HS\to\HS$ is cocoercive with constant $\upalpha\in\RPP$
if 
\begin{equation}
(\forall\xS\in\HS)(\forall\yS\in\HS)\quad
\scal{\xS-\yS}{\CS\xS-\CS\yS}\geq\upalpha\norm{\CS\xS-\CS\yS}^2.
\end{equation}
We denote by $\upGamma_0(\HS)$ the class of lower semicontinuous
convex functions $\mathsf{f}\colon\HS\to\RX$ such that 
$\dom\mathsf{f}=\menge{\xS\in\HS}{\mathsf{f}(\xS)<\pinf}\neq\emp$.
The subdifferential of
$\mathsf{f}\in\upGamma_0(\HS)$ is the maximally monotone operator
$\uppartial\mathsf{f}\colon\HS\to 2^{\HS}\colon
\xS\mapsto\menge{\xS^*\in\HS}{(\forall\yS\in\HS)\;
\scal{\yS-\xS}{\xS^*}+\mathsf{f}(\xS)\leq\mathsf{f}(\yS)}$ and the
proximity operator of $\mathsf{f}$ is 
\begin{equation}
\prox_{\mathsf{f}}=\mathsf{J}_{\uppartial\mathsf{f}}\colon\HS\to\HS
\colon\xS\mapsto\argmin_{\zS\in\HS}\brk3{\mathsf{f}(\zS)
+\dfrac{1}{2}\norm{\xS-\zS}^2}. 
\end{equation}
The infimal convolution of $\mathsf{f}$ and
$\mathsf{h}\in\upGamma_0(\HS)$ is
$\mathsf{f}\infconv\mathsf{h}\colon\HS\to\RXX\colon 
\xS\mapsto\inf_{\yS\in\HS}(\mathsf{f}(\yS)+\mathsf{h}(\xS-\yS))$.

\subsection{Probabilistic setting}
\label{sec:23}
The underlying probability space $(\upOmega,\FE,\PP)$ is complete. 
Let $(\upXi,\EuScript{G})$ be a measurable space. A $\upXi$-valued
random variable (random variable for short) is a measurable mapping
$x\colon(\upOmega,\FE,\PP)\to(\upXi,\EuScript{G})$. In particular, 
an $\HS$-valued random variable is a measurable 
mapping $x\colon(\upOmega,\FE,\PP)\to(\HS,\BE_{\HS})$, where 
$\BE_{\HS}$ denotes the Borel $\upsigma$-algebra of $\HS$.
Given $x\colon\upOmega\to\upXi$ and $\mathsf{S}\in\EuScript{G}$, we
set $[x\in\mathsf{S}]=\menge{\upomega\in\upOmega}
{x(\upomega)\in\mathsf{S}}$.
Let $\mathsf{p}\in\left[1,\pinf\right[$ and let $\XX$ be a
sub $\upsigma$-algebra of $\FE$. Then
$L^\mathsf{p}(\upOmega,\XX,\PP;\HS)$ denotes the space of 
equivalence classes of $\Pas$ equal $\HS$-valued random variables
$x\colon(\upOmega,\XX,\PP)\to(\HS,\BE_{\HS})$ such that 
$\EE\|x\|^\mathsf{p}<\pinf$. Endowed with the norm 
\begin{equation}
\|\cdot\|_{L^\mathsf{p}(\upOmega,\XX,\PP;\HS)}\colon
x\mapsto\EE^{1/\mathsf{p}}\|x\|^\mathsf{p}
=\brk3{\int_{\upOmega}\|x(\upomega)\|^{\mathsf{p}}
\PP(d\upomega)}^{1/\mathsf{p}},
\end{equation}
$L^\mathsf{p}(\upOmega,\XX,\PP;\HS)$ is a
real Banach space. Further,
\begin{equation}
(\forall\mathsf{S}\in\BE_{\HS})\quad
L^\mathsf{p}(\upOmega,\XX,\PP;\mathsf{S})=\Menge1{x\in
L^\mathsf{p}(\upOmega,\XX,\PP;\HS)}{x\in\mathsf{S}\:\;\Pas}.
\end{equation}
The $\upsigma$-algebra generated by a family $\upPhi$ of random 
variables is denoted by $\upsigma(\upPhi)$. Let $(x_{\nS})_{\nnn}$ 
and $x$ be $\HS$-valued random variables. 
We say that $(x_{\nS})_{\nnn}$ converges in probability to
$x$, denoted by $x_{\nS}\Pto x$, 
if $\norm{x_{\nS}-x}$ converges in probability to $0$, i.e.,
$(\forall\upvarepsilon\in\RPP)$
$\PP\brk{\brk[s]{\norm{x_{\nS}-x}>\upvarepsilon}}\to 0$.
We say that $\varphi\colon\upOmega\times\HS\to\RR$ is an 
$\XX$-Carath\'eodory integrand if
\begin{equation}
\begin{cases}
\text{for}\;\PP\text{-almost every}\;\upomega\in\upOmega,
\;\varphi(\upomega,\cdot)\;\text{is continuous};\\
\text{for every}\;\xS\in\HS,\;\varphi(\cdot,\xS)\;
\text{is}\;\XX\text{-measurable.}
\end{cases}
\end{equation}
We denote by $\mathfrak{C}(\upOmega,\XX,\PP;\HS)$ the class of 
$\XX$-Carath\'eodory integrands
$\varphi\colon\upOmega\times\HS\to\RP$.

\subsection{Preliminary results}
\label{sec:24}

Our main results rest on several technical facts, which are
presented below. The first two lemmas are direct consequences of
the corresponding statements for $\RR$-valued random variables; see
\cite[Section~2.10]{Shir16}.

\begin{lemma}
\label{l:1}
Let $(x_{\nS})_{\nnn}$ and $x$ be $\HS$-valued random variables and
let $\mathsf{p}\in\left[1,\pinf\right[$ be such that 
$(x_{\nS})_{\nnn}$ converges strongly in 
$L^{\mathsf{p}}(\upOmega,\FE,\PP;\HS)$ to $x$.
Then $x_{\nS}\Pto x$.
\end{lemma}
\newpage
\begin{lemma}
\label{l:2}
Let $(x_{\nS})_{\nnn}$ and $x$ be $\HS$-valued random variables
such that $x_{\nS}\Pto x$. Then there exists a strictly 
increasing sequence $(\jS_{\nS})_{\nnn}$ in $\NN$ such that 
$(x_{\jS_{\nS}})_{\nnn}$ converges strongly $\Pas$ to $x$. 
\end{lemma}

\begin{lemma}
\label{l:3}
Let $(\xi_{\nS})_{\nnn}$, $(\Delta_{\nS})_{\nnn}$, and
$(\chi_{\nS})_{\nnn}$ be sequences of\, $\RR$-valued random
variables such that
\begin{equation}
\label{e:l3}
\begin{cases}
\varlimsup\Delta_{\nS}\leq0\;\Pas;\\
\chi_{\nS}\Pto 0;\\
(\forall\nnn)\;\;\xi_{\nS}\geq 0\;\Pas\;\;\text{and}\;\;
\xi_{\nS}+\chi_{\nS}\leq\Delta_{\nS}\;\Pas
\end{cases}
\end{equation}
Then $\xi_{\nS}\Pto 0$.
\end{lemma}
\begin{proof}
Let $\upOmega'\in\FE$ be such that $\PP(\upOmega')=1$ and, 
for every $\upomega\in\upOmega'$,
$\xi_{\nS}(\upomega)+\chi_{\nS}(\upomega)
\leq\Delta_{\nS}(\upomega)$.
Let $\upvarepsilon\in\RPP$ and $\nnn$. Let $\upomega\in\upOmega'$
and suppose that $\xi_{\nS}(\upomega)>\upvarepsilon$. Then there
are two cases:
\begin{itemize}
\item
$\chi_{\nS}(\upomega)<-\upvarepsilon/2$.
\item
$\chi_{\nS}(\upomega)\geq-\upvarepsilon/2$, in which case 
${\upvarepsilon}/{2}=\upvarepsilon-{\upvarepsilon}/{2}
<\xi_{\nS}(\upomega)+\chi_{\nS}(\upomega)
\leq\Delta_{\nS}(\upomega)$. Therefore, 
\begin{equation}
[\xi_{\nS}>\upvarepsilon]\subset[\chi_{\nS}<-\upvarepsilon/2]\cup
[\Delta_{\nS}>\upvarepsilon/2].
\end{equation}
\end{itemize}
Note that $\PP([\chi_{\nS}<-\upvarepsilon/2])\to 0$ since 
$\chi_{\nS}\Pto 0$. On the other hand, since
$\varlimsup\Delta_{\nS}\leq0\;\Pas$, we have
\begin{align}
\varlimsup\PP\brk{[\Delta_{\nS}>\upvarepsilon/2]}
&\leq\PP\brk2{\varlimsup\,[\Delta_{\nS}>\upvarepsilon/2]}\nonumber\\
&=\PP\brk2{\menge{\upomega\in\upOmega}
{(\forall\nnn)(\exi\kS\in\{\nS,\nS+1,\ldots\})\;
\Delta_{\kS}(\upomega)>\upvarepsilon/2}}\nonumber\\
&=0.
\end{align}
Altogether,
$\PP([\abs{\xi_{\nS}}>\upvarepsilon])
=\PP([\xi_{\nS}>\upvarepsilon])
\leq\PP([\chi_{\nS}<-\upvarepsilon/2])
+\PP([\Delta_{\nS}>\upvarepsilon/2])\to0$ and we conclude that
$\xi_{\nS}\Pto 0$.
\end{proof}

\begin{lemma}
\label{l:100}
Let $x\in L^2(\upOmega,\FE,\PP;\HS)$ and let
$\mathsf{T}\colon\HS\to\HS$ be Lipschitzian. Then
$\mathsf{T}x\in L^2(\upOmega,\FE,\PP;\HS)$.  
\end{lemma}
\begin{proof}
Let $\upbeta\in\RPP$ be the Lipschitz constant of $\mathsf{T}$.
Since $\mathsf{T}$ is continuous, the mapping
$\upomega\mapsto(\mathsf{T}\circ x)(\upomega)
=\mathsf{T}x(\upomega)$ is measurable. Furthermore,
\begin{equation}
\label{e:l100}
\frac{1}{2}\EE\norm{\mathsf{T}x}^2
\leq\EE\norm{\mathsf{T}x-\mathsf{T}{0}}^2
+\EE\norm{\mathsf{T}{0}}^2
\leq\upbeta^2\EE\norm{x-{0}}^2
+\EE\norm{\mathsf{T}{0}}^2
=\upbeta^2\EE\norm{x}^2
+\norm{\mathsf{T}{0}}^2<\pinf,
\end{equation}
which confirms that $\mathsf{T}x\in L^2(\upOmega,\FE,\PP;\HS)$.
\end{proof}

\begin{lemma}
\label{l:101}
Let $(x_{\nS})_{\nnn}$ be a sequence in
$L^2(\upOmega,\FE,\PP;\HS)$, let
$\mathsf{m}\in\NN$, and let $\vartheta(\mathsf{m})$ be a
$\{0,\ldots,\mathsf{m}\}$-valued random variable. Then the
function $x_{\vartheta(\mathsf{m})}^{}\colon\upomega\mapsto 
x_{\vartheta(\mathsf{m})(\upomega)}^{}(\upomega)$ is in
$L^2(\upOmega,\FE,\PP;\HS)$.
\end{lemma}
\begin{proof}
We note that 
$x_{\vartheta(\mathsf{m})}^{}
=\sum_{\jS=0}^{\mathsf{m}}
\mathsf{1}_{[\vartheta(\mathsf{m})=\jS]}x_{\jS}\;\;\Pas$,
which shows that $x_{\vartheta(\mathsf{m})}^{}$ is measurable, as
$(\upOmega,\FE,\PP)$ is complete, and that
$\EE\norm{x_{\vartheta(\mathsf{m})}^{}}^2
\leq\brk{\mathsf{m}+1}
\max_{0\leq\jS\leq\mathsf{m}}\EE\norm{x_{\jS}}^2<\pinf$.
Thus, $x_{\vartheta(\mathsf{m})}^{}\in L^2(\upOmega,\FE,\PP;\HS)$. 
\end{proof}

The following theorem is a straightforward consequence of 
\cite[Theorems~3.2 and 3.6]{Moco25}.

\begin{theorem}
\label{t:moc}
Let $\ZS$ be a nonempty closed convex subset of $\HS$,
let $x_{0}\in L^2(\upOmega,\FE,\PP;\HS)$, and let
$\uprho\in\left[2,\pinf\right[$. 
Iterate
\begin{equation}
\label{e:a2}
\begin{array}{l}
\textup{for}\;\nS=0,1,\ldots\\
\left\lfloor
\begin{array}{l}
\XX_{\nS}=\upsigma(x_{0},\dots,x_{\nS})\\
t^*_{\nS}\in L^2(\upOmega,\FE,\PP;\HS)\;\textup{and}\; 
\eta_{\nS}\in{L^1(\upOmega,\FE,\PP;\RR)}\;\textup{satisfy}\\[1mm]
\qquad\begin{cases}
\dfrac{\mathsf{1}_{[t_{\nS}^*\neq0]}\mathsf{1}_{
\left[\scal{x_{\nS}}{t_{\nS}^*}>\eta_{\nS}\right]}\eta_{\nS}}
{\|t_{\nS}^*\|+\mathsf{1}_{[t_{\nS}^*=0]}}\in
L^2(\upOmega,\FE,\PP;\RR);\\[3mm]
\theta_{\nS}=
\dfrac{\mathsf{1}_{[t_{\nS}^*\neq0]}\mathsf{1}_{
\left[\scal{x_{\nS}}{t_{\nS}^*}>\eta_{\nS}\right]}
\bigl(\scal{x_{\nS}}{t_{\nS}^*}-\eta_{\nS}\bigr)}
{\|t_{\nS}^*\|^2+\mathsf{1}_{[t_{\nS}^*=0]}};\\
(\forall\mathsf{z}\in\ZS)\;\;
\scal{\mathsf{z}}{\EC{\theta_{\nS}t^*_{\nS}}{\XX_{\nS}}}\leq
\EC{\theta_{\nS}\eta_{\nS}}{\XX_{\nS}}+
\varepsilon_{\nS}(\cdot,\mathsf{z})\;\Pas,\;
\textup{where}\;\varepsilon_{\nS}\in 
\mathfrak{C}(\upOmega,\XX_{\nS},\PP;\HS)
\end{cases}\\
d_{\nS}=\theta_{\nS}t_{\nS}^*\\
\lambda_{\nS}\in L^\infty(\upOmega,\FE,\PP;\left]0,\uprho\right])\\
x_{\nS+1}=x_{\nS}-\lambda_{\nS} d_{\nS}.
\end{array}
\right.\\
\end{array}
\end{equation}
Suppose that, for every $\nnn$, $\lambda_{\nS}$ is independent of 
$\upsigma(\{x_{0},\ldots,x_{\nS},d_{\nS}\})$, and
$\EE(\lambda_{\nS}(2-\lambda_{\nS}))\geq 0$.
Then the following hold:
\begin{enumerate}
\item
\label{t:2i-}
$(x_{\nS})_{\nnn}$ is a well-defined sequence in
$L^2(\upOmega,\FE,\PP;\HS)$.
\item
\label{t:2i+}
Let $\nnn$ and $\zS\in\ZS$. Then 
$\EE\norm{x_{\nS+1}-\zS}^2\leq
\EE\norm{x_{\nS}-\zS}^2
+2\varepsilon_{\nS}(\cdot,\zS)\EE\lambda_{\nS}\;\Pas$
\item
Suppose that, for every $\zS\in\ZS$, $\sum_{\nnn}
\EE\varepsilon_{\nS}(\cdot,\zS)\EE\lambda_{\nS}<\pinf$.
Then the following are satisfied:
\begin{enumerate}
\item
\label{t:2ii}
$(\norm{x_{\nS}})_{\nnn}$ is bounded $\Pas$ and
$(\EE\norm{x_{\nS}}^2)_{\nnn}$ is bounded.
\item
\label{t:2viid}
$\sum_{\nnn}\EE(\lambda_{\nS}(2-\lambda_{\nS}))
\EE{\norm{d_{\nS}}^2}<\pinf$.
\item
\label{t:2viie}
Suppose that $\inf_{\nnn}\EE(\lambda_{\nS}(2-\lambda_{\nS}))>0$.
Then $\sum_{\nnn}\EE{\|x_{\nS+1}-x_{\nS}\|^2}<\pinf$.
\item
\label{t:2viif}
Suppose that $\mathfrak{W}(x_{\nS})_{\nnn}\subset\ZS\;\Pas$ 
Then $(x_{\nS})_{\nnn}$ converges weakly $\Pas$ and weakly in
$L^2(\upOmega,\FE,\PP;\HS)$ to a random variable
$x\in L^2(\upOmega,\FE,\PP;\ZS)$.
\item
\label{t:2viig}
Suppose that $\mathfrak{S}(x_{\nS})_{\nnn}\cap\ZS\neq\emp\;\Pas$ 
Then $(x_{\nS})_{\nnn}$ converges strongly $\Pas$ and strongly in
$L^1(\upOmega,\FE,\PP;\HS)$ to a random variable 
$x\in L^2(\upOmega,\FE,\PP;\ZS)$. Additionally, 
$(x_{\nS})_{\nnn}$ converges weakly in $L^2(\upOmega,\FE,\PP;\HS)$
to $x$.
\end{enumerate}
\end{enumerate}
\end{theorem}

\begin{lemma}[\protect{\cite[Lemma~A.2]{Sadd22}}]
\label{l:A2}
Let $\upalpha\in\RP$, let $\AS\colon\HS\to\HS$ be
$\upalpha$-Lipschitzian, let $\upphi\in\RPP$, and let 
$\upgamma\in\left]0,1/(\upalpha+\upphi)\right]$.
Then $\upgamma^{-1}\Id-\AS$ is $\upphi$-strongly monotone.
\end{lemma}

\section{Convergence analysis}
\label{sec:3}

This section is dedicated to establishing the weak convergence to
solutions to Problem~\ref{prob:19}, in the almost sure and
$L^2(\upOmega,\FE,\PP;\HS)$ modes, of the sequence
$(x_{\nS})_{\nnn}$ generated by
the stochastic Algorithm~\ref{algo:20}. 

\begin{theorem}
\label{t:5}
In the context of Problem~\ref{prob:19}, let $(x_{\nS})_{\nnn}$ be
the sequence generated by Algorithm~\ref{algo:20}. For every 
$\nnn$ and every $\zS\in\ZS$, set
\begin{equation}
\label{e:GBeps}
\varepsilon_{\nS}(\cdot,\zS)
=\max\brk[c]3{0,
\EC2{\theta_{\nS}\brk2{
\scal1{w_{\nS}-\zS}{e_{\nS}^*+f_{\nS}^*}
+\scal1{e_{\nS}}{w_{\nS}^*+\CS\zS}
+\scal1{e_{\nS}}{e_{\nS}^*}}}{\XX_{\nS}}},
\end{equation}
and suppose that $\lambda_{\nS}$ is independent
of $\upsigma(\{x_{0},\ldots,x_{\nS},d_{\nS}\})$ and
that $\EE(\lambda_{\nS}(2-\lambda_{\nS}))\geq 0$. 
Then the following hold:
\begin{enumerate}
\item
\label{t:50}
Let $\nnn$ and $\zS\in\ZS$. Then 
\begin{equation}
\label{e:110}
\scal1{{\zS}}{\EC{\theta_{\nS}t_{\nS}^*}{\XX_{\nS}}}
\leq\EC2{\theta_{\nS}\scal{w_{\nS}}{t_{\nS}^*}}{\XX_{\nS}} 
+\dfrac{1}{4\upalpha}
\EC1{\theta_{\nS}\norm{w_{\nS}-q_{\nS}}^2}{\XX_{\nS}}
+\varepsilon_{\nS}(\cdot,\zS)\;\;\Pas
\end{equation}
\item
\label{t:5i-}
$(x_{\nS})_{\nnn}$ lies in
$L^2(\upOmega,\FE,\PP;\HS)$.
\item
\label{t:5i+}
Let $\nnn$ and $\zS\in\ZS$. Then 
$\EE\norm{x_{\nS+1}-\zS}^2\leq
\EE\norm{x_{\nS}-\zS}^2
+2\varepsilon_{\nS}(\cdot,\zS)\EE\lambda_{\nS}\;\Pas$
\item
Suppose that, for every $\zS\in\ZS$,
$\sum_{\nnn}\EE\varepsilon_{\nS}(\cdot,\zS)\EE\lambda_{\nS}<\pinf$.
Then the following are satisfied:
\begin{enumerate}
\item
\label{t:5i}
$(\norm{x_{\nS}})_{\nnn}$ is bounded 
$\Pas$ and $(\EE\norm{x_{\nS}}^2)_{\nnn}$ is bounded.
\item
\label{t:5ii-}
$\sum_{\nnn}\EE(\lambda_{\nS}(2-\lambda_{\nS}))\EE\norm{d_{\nS}}^2
<\pinf$.
\item
\label{t:5ii}
Suppose that $\inf_{\nnn}\EE(\lambda_{\nS}(2-\lambda_{\nS}))>0$.
Then $\sum_{\nnn}\EE{\norm{x_{\nS+1}-x_{\nS}}^2}<\pinf$.
\item
\label{t:5iii}
Suppose that $\inf_{\nnn}\lambda_{\nS}>0\;\Pas$ and that 
$(t_{\nS}^*)_{\nnn}$ is bounded $\Pas$ Then 
$\varlimsup\Delta_{\nS}\leq0\;\Pas$
\item
\label{t:5iv-}
Suppose that 
$x_{\nS}-w_{\nS}-e_{\nS}\weakly{0}\;\Pas$,
$w_{\nS}+e_{\nS}-q_{\nS}\to{0}\;\Pas$, and
$w_{\nS}^*+e_{\nS}^*+\CS q_{\nS}\to{0}\;\Pas$
Then $(x_{\nS})_{\nnn}$ converges weakly $\Pas$ and weakly
in $L^2(\upOmega,\FE,\PP;\HS)$ to a $\ZS$-valued 
random variable.
\item
\label{t:5iv}
Suppose that $\dim\HS<\pinf$,
$x_{\nS}-w_{\nS}-e_{\nS}\Pto{0}$,
$w_{\nS}+e_{\nS}-q_{\nS}\Pto{0}$, and
$w_{\nS}^*+e_{\nS}^*+\CS q_{\nS}\Pto{0}$.
Then $(x_{\nS})_{\nnn}$ converges $\Pas$ and 
in $L^1(\upOmega,\FE,\PP;\HS)$ to a $\ZS$-valued random variable.
\end{enumerate}
\end{enumerate}
\end{theorem}
\begin{proof}
\ref{t:50}:
Note that $(\zS,-{\CS}{\zS})\in\gra\WS$. Hence, \eqref{e:a20}
and the monotonicity of $\WS$ yield
\begin{align}
&\scal1{\zS-w_{\nS}-e_{\nS}}{w_{\nS}^*+e_{\nS}^*+c_{\nS}^*}
\nonumber\\
&\quad=\scal{\zS-w_{\nS}-e_{\nS}}{w_{\nS}^*+e_{\nS}^*+\CS q_{\nS}}
-\scal{\zS-w_{\nS}-e_{\nS}}{f_{\nS}^*}
\nonumber\\
&\quad=\scal{\zS-w_{\nS}-e_{\nS}}{w_{\nS}^*+e_{\nS}^*+\CS\zS}
+\scal{\zS-w_{\nS}-e_{\nS}}{\CS q_{\nS}-\CS\zS}
-\scal{\zS-w_{\nS}-e_{\nS}}{f_{\nS}^*}
\nonumber\\
&\quad\leq\scal{\zS-w_{\nS}-e_{\nS}}{\CS q_{\nS}-\CS\zS}
-\scal{\zS-w_{\nS}-e_{\nS}}{f_{\nS}^*}
\nonumber\\
&\quad=-\scal{\zS-q_{\nS}}{\CS\zS-\CS q_{\nS}}
+\scal{w_{\nS}-q_{\nS}}{\CS\zS-\CS q_{\nS}}
+\scal{e_{\nS}}{\CS\zS-\CS q_{\nS}}
-\scal{\zS-w_{\nS}-e_{\nS}}{f_{\nS}^*}
\nonumber\\
&\quad\leq-\upalpha\norm{\CS\zS-\CS q_{\nS}}^2
+\norm{w_{\nS}-q_{\nS}}\norm{\CS\zS-\CS q_{\nS}}
+\scal{e_{\nS}}{\CS\zS-\CS q_{\nS}}
-\scal{\zS-w_{\nS}-e_{\nS}}{f_{\nS}^*}
\nonumber\\
&\quad=\dfrac{\norm{w_{\nS}-q_{\nS}}^2}{4\upalpha}
-\abs2{\brk1{2\sqrt{\upalpha}}^{-1}\norm{w_{\nS}-q_{\nS}}
-\sqrt{\upalpha}\norm{\CS\zS-\CS q_{\nS}}}^2
\nonumber\\
&\hspace{1cm}
+\scal{e_{\nS}}{\CS\zS-\CS q_{\nS}}
-\scal{\zS-w_{\nS}-e_{\nS}}{f_{\nS}^*}
\nonumber\\
&\quad\leq\dfrac{\norm{w_{\nS}-q_{\nS}}^2}{4\upalpha}
+\scal{w_{\nS}-\zS}{f_{\nS}^*}
+\scal{e_{\nS}}{\CS\zS-\CS q_{\nS}}
+\scal{e_{\nS}}{f_{\nS}^*}\;\;\Pas
\end{align}
Therefore, since $t_{\nS}^*=w_{\nS}^*+c_{\nS}^*$, 
\begin{equation}
\label{e:111}
\scal1{\zS}{t_{\nS}^*}
\leq\scal1{w_{\nS}}{t_{\nS}^*}
+\dfrac{\norm{w_{\nS}-q_{\nS}}^2}{4\upalpha}
+\scal1{w_{\nS}-\zS}{e_{\nS}^*+f_{\nS}^*}
+\scal1{e_{\nS}}{w_{\nS}^*+\CS\zS}
+\scal1{e_{\nS}}{e_{\nS}^*}\;\;\Pas
\end{equation}
On the other hand, because $\theta_{\nS}\geq0\;\Pas$, it 
follows from scaling by $\theta_{\nS}$ and taking the conditional
expectation with respect to $\XX_{\nS}$ in \eqref{e:111} that
\eqref{e:110} holds.

\ref{t:5i-}--\ref{t:5i+}: Let $\nnn$ and set
$\eta_{\nS}=\scal{w_{\nS}}{t_{\nS}^*}
+(4\upalpha)^{-1}\norm{w_{\nS}-q_{\nS}}^2$.
Then $\eta_{\nS}\in L^1(\upOmega,\FE,\PP;\RR)$ and
$\varepsilon_{\nS}\in\mathfrak{C}(\upOmega,\XX_{\nS},\PP;\HS)$.
Furthermore, by the Cauchy--Schwarz inequality,
\begin{align}
\frac{1}{2}\EE
\Bigg|\dfrac{\mathsf{1}_{[t_{\nS}^*\neq0]}\mathsf{1}_{
\left[\scal{x_{\nS}}{t_{\nS}^*}>\eta_{\nS}\right]}\eta_{\nS}}
{\|t_{\nS}^*\|+\mathsf{1}_{[t_{\nS}^*=0]}}\Bigg|^2
&\leq\EE\Bigg|\dfrac{\mathsf{1}_{[t_{\nS}^*\neq0]}\mathsf{1}_{
\left[\scal{x_{\nS}}{t_{\nS}^*}>\eta_{\nS}\right]}
\big(\scal{x_{\nS}}{t_{\nS}^*}-\eta_{\nS}\big)}
{\|t_{\nS}^*\|+\mathsf{1}_{[t_{\nS}^*=0]}}\Bigg|^2
+\EE\Bigg|\dfrac{\mathsf{1}_{[t_{\nS}^*\neq0]}\mathsf{1}_{
\left[\scal{x_{\nS}}{t_{\nS}^*}>\eta_{\nS}\right]}
\scal{x_{\nS}}{t_{\nS}^*}}
{\|t_{\nS}^*\|+\mathsf{1}_{[t_{\nS}^*=0]}}\Bigg|^2
\nonumber\\
&\leq\EE\Bigg|\dfrac{\mathsf{1}_{[t_{\nS}^*\neq0]}\mathsf{1}_{
\left[\scal{x_{\nS}}{t_{\nS}^*}>\eta_{\nS}\right]}
\big(\scal{x_{\nS}}{t_{\nS}^*}-\eta_{\nS}\big)}
{\|t_{\nS}^*\|+\mathsf{1}_{[t_{\nS}^*=0]}}\Bigg|^2
+\EE\norm{x_{\nS}}^2\nonumber\\
&=\EE\Bigg|\dfrac{\mathsf{1}_{[t_{\nS}^*\neq0]}\mathsf{1}_{
\left[\scal{x_{\nS}}{t_{\nS}^*}>\eta_{\nS}\right]}
\big(\scal{x_{\nS}-w_{\nS}}{t_{\nS}^*}
-(4\upalpha)^{-1}\norm{w_{\nS}-q_{\nS}}^2\big)}
{\|t_{\nS}^*\|+\mathsf{1}_{[t_{\nS}^*=0]}}\Bigg|^2
+\EE\norm{x_{\nS}}^2\nonumber\\
&\leq\EE\Bigg|\dfrac{\mathsf{1}_{[t_{\nS}^*\neq0]}\mathsf{1}_{
\left[\scal{x_{\nS}}{t_{\nS}^*}>\eta_{\nS}\right]}
\scal{x_{\nS}-w_{\nS}}{t_{\nS}^*}}
{\|t_{\nS}^*\|+\mathsf{1}_{[t_{\nS}^*=0]}}\Bigg|^2
+\EE\norm{x_{\nS}}^2\nonumber\\
&\leq\EE\norm{x_{\nS}-w_{\nS}}^2+\EE\norm{x_{\nS}}^2
\nonumber\\
&<\pinf.
\label{e:35+}
\end{align}
Altogether, in view of \ref{t:50}, we deduce that \eqref{e:a20} is
a realization of \eqref{e:a2}. Hence, the claims follow from
Theorem~\ref{t:moc}\ref{t:2i-}--\ref{t:2i+}.

\ref{t:5i}--\ref{t:5ii}: These follow from
Theorem~\ref{t:moc}\ref{t:2ii}--\ref{t:2viie}.

\ref{t:5iii}: Since $\inf_{\nnn}\lambda_{\nS}>0\;\Pas$, we proceed,
for $\PP$-almost every $\upomega\in\upOmega$, as in the proof of
\cite[Proposition~3(iii)]{Sadd22} to get the result using
\ref{t:5ii}.

\ref{t:5iv-}: In view of \ref{t:5i}, we fix $\upOmega'\in\FE$ 
such that
\begin{equation}
\label{e:l123-}
\PP\brk{\upOmega'}=1\;\;\text{and}\;\;
(\forall\upomega\in\upOmega')\quad\begin{cases}
x_{\nS}(\upomega)-w_{\nS}(\upomega)-e_{\nS}(\upomega)
\weakly{0};\\
w_{\nS}(\upomega)+e_{\nS}(\upomega)-q_{\nS}(\upomega)
\to{0};\\
w^*_{\nS}(\upomega)+e_{\nS}^*(\upomega)+\CS q_{\nS}(\upomega)
\to{0};\\
(\norm{x_{\nS}(\upomega)})_{\nnn}\;
\text{is bounded.}
\end{cases}
\end{equation}
Now let $\upomega\in\upOmega'$ and 
$\xS\in\WC(x_{\nS}(\upomega))_{\nnn}$. 
Then there exists a strictly increasing sequence in $\NN$, say 
$(\kS_{\nS})_{\nnn}$, such that 
$x_{\kS_{\nS}}(\upomega)\weakly\xS$. Furthermore,
\begin{equation}
\label{e:cuz1}
w_{\kS_{\nS}}(\upomega)+e_{\kS_{\nS}}(\upomega)
=x_{\kS_{\nS}}(\upomega)-\brk1{x_{\kS_{\nS}}(\upomega)-
w_{\kS_{\nS}}(\upomega)-e_{\kS_{\nS}}(\upomega)}\weakly\xS
\end{equation}
and, since $\CS$ is $\alpha^{-1}$-Lipschitzian, 
\begin{align}
\label{e:cuz2}
&\norm1{w_{\kS_{\nS}}^*(\upomega)+e_{\kS_{\nS}}^*(\upomega)
+\CS \brk1{w_{\kS_{\nS}}(\upomega)+e_{\kS_{\nS}}(\upomega)}}
\nonumber\\
&\qquad\leq
\norm1{w_{\kS_{\nS}}^*(\upomega)+e_{\kS_{\nS}}^*(\upomega)
+\CS q_{\kS_{\nS}}(\upomega)}
+\norm1{\CS\brk1{w_{\kS_{\nS}}(\upomega)+e_{\kS_{\nS}}(\upomega)}
-\CS q_{\kS_{\nS}}(\upomega)}
\nonumber\\
&\qquad\leq
\norm1{w_{\kS_{\nS}}^*(\upomega)+e_{\kS_{\nS}}^*(\upomega)
+\CS q_{\kS_{\nS}}(\upomega)}
+\dfrac{\norm1{w_{\kS_{\nS}}(\upomega)+e_{\kS_{\nS}}(\upomega)
-q_{\kS_{\nS}}(\upomega)}}{\upalpha}
\nonumber\\
&\qquad\to 0.
\end{align}
On the other hand, \eqref{e:a20} yields
\begin{equation}
\label{e:cuz3}
(\forall\nnn)\quad
\brk2{w_{\kS_{\nS}}(\upomega)+e_{\kS_{\nS}}(\upomega),
w_{\kS_{\nS}}^*(\upomega)+e_{\kS_{\nS}}^*(\upomega)
+\CS \brk1{w_{\kS_{\nS}}(\upomega)+e_{\kS_{\nS}}(\upomega)}}
\in\gra(\WS+\CS).
\end{equation}
Since, by \cite[Corollary~25.5(i)]{Livre1}, $\WS+\CS$ is maximally
monotone, \eqref{e:cuz1}, \eqref{e:cuz2}, 
\eqref{e:cuz3}, and \cite[Proposition~20.38(ii)]{Livre1} imply
that $\xS\in\ZS$. Since $\xS$ is arbitrarily chosen in
$\WC(x_{\nS}(\upomega))_{\nnn}$, we deduce that
$\WC(x_{\nS}(\upomega))_{\nnn}\subset\ZS$ and, since
$\PP(\upOmega')=1$, that $\WC(x_{\nS})_{\nnn}\subset\ZS\;\Pas$
Therefore, it follows from Theorem~\ref{t:moc}\ref{t:2viif} that
$(x_{\nS})_{\nnn}$ converges weakly $\Pas$ and weakly in
$L^2(\upOmega,\FE,\PP;\HS)$ to a $\ZS$-valued random variable.

\ref{t:5iv}: Lemma~\ref{l:2} guarantees the existence of 
a strictly increasing sequence in $\NN$, say
$(\lS_{\nS})_{\nnn}$, such that 
$x_{\lS_{\nS}}-w_{\lS_{\nS}}-e_{\lS_{\nS}}\to{0}\;\Pas$,
$w_{\lS_{\nS}}+e_{\lS_{\nS}}-q_{\lS_{\nS}}\to{0}\;\Pas$, and
$w^*_{\lS_{\nS}}+e_{\lS_{\nS}}^*+\CS q_{\lS_{\nS}}\to{0}\;
\Pas$
Additionally, it follows from \ref{t:5i} that
$(\norm{x_{\lS_{\nS}}})_{\nnn}$ is bounded
$\Pas$ Let $\upOmega'\in\FE$ be such that
\begin{equation}
\label{e:l123}
\PP\brk{\upOmega'}=1\;\;\text{and}\;\;
(\forall\upomega\in\upOmega')\quad\begin{cases}
x_{\lS_{\nS}}(\upomega)-w_{\lS_{\nS}}(\upomega)
-e_{\lS_{\nS}}(\upomega)\to{0};\\
w_{\lS_{\nS}}(\upomega)+e_{\lS_{\nS}}(\upomega)
-q_{\lS_{\nS}}(\upomega)\to{0};\\
w^*_{\lS_{\nS}}(\upomega)+e_{\lS_{\nS}}^*(\upomega)
+\CS q_{\lS_{\nS}}(\upomega)\to{0};\\
(\norm{x_{\lS_{\nS}}(\upomega)})_{\nnn}\;
\text{is bounded.}
\end{cases}
\end{equation}
Let $\upomega\in\upOmega'$.
We derive from \eqref{e:l123} and the fact that $\HS$ is
finite-dimensional that there exists 
$\xS\in\HS$ and a further subsequence
$(\kS_{\lS_{\nS}})_{\nnn}$ such that 
$x_{\kS_{\lS_{\nS}}}(\upomega)\to\xS$, 
\begin{equation}
w_{\kS_{\lS_{\nS}}}(\upomega)+e_{\kS_{\lS_{\nS}}}(\upomega)
=x_{\kS_{\lS_{\nS}}}(\upomega)-
\brk1{x_{\kS_{\lS_{\nS}}}(\upomega)-w_{\kS_{\lS_{\nS}}}(\upomega)
-e_{\kS_{\lS_{\nS}}}(\upomega)}\to\xS,
\end{equation}
and, as in \eqref{e:cuz2},
\begin{equation}
w_{\kS_{\lS_{\nS}}}^*(\upomega)
+e_{\kS_{\lS_{\nS}}}^*(\upomega)
+\CS \brk1{w_{\kS_{\lS_{\nS}}}(\upomega)
+e_{\kS_{\lS_{\nS}}}(\upomega)}\to 0.
\end{equation}
However, as in \eqref{e:cuz3}, 
\begin{equation}
(\forall\nnn)\quad
\brk2{w_{\kS_{\lS_{\nS}}}(\upomega)+e_{\kS_{\lS_{\nS}}}(\upomega),
w_{\kS_{\lS_{\nS}}}^*(\upomega)+e_{\kS_{\lS_{\nS}}}^*(\upomega)
+\CS \brk1{w_{\kS_{\lS_{\nS}}}(\upomega)
+e_{\kS_{\lS_{\nS}}}(\upomega)}}
\in\gra(\WS+\CS),
\end{equation}
and the maximal monotonicity of $\WS+\CS$ yields
$\xS\in\ZS$. Thus, $(x_{\nS}(\upomega))_{\nnn}$ has 
a cluster point in $\ZS$ and we conclude that 
$\SC(x_{\nS})_{\nnn}\cap\ZS\neq\emp\;\Pas$ 
Therefore, it follows from Theorem~\ref{t:moc}\ref{t:2viig} that 
$(x_{\nS})_{\nnn}$ converges $\Pas$ and in
$L^1(\upOmega,\FE,\PP;\HS)$ to a $\ZS$-valued random variable.
\end{proof}

\begin{remark}
\label{r:8}
Let us comment on the role of some parameters in Theorem~\ref{t:5}.
\begin{enumerate}
\item
\label{r:8i}
The variable $\varepsilon_{\nS}$ regroups the
stochastic errors $(e_{\nS},e_{\nS}^*,f_{\nS}^*)$ affecting the
selection of points in the graphs of $\WS$ and $\CS$ in 
Algorithm~\ref{algo:20}.
\item
\label{r:8ii}
The random relaxation parameters 
$(\lambda_{\nS})_{\nnn}$ satisfy
$\inf_{\nnn}\EE(\lambda_{\nS}(2-\lambda_{\nS}))>0$.
When the relaxation parameters are deterministic,
this condition imposes that, for every $\nnn$,  
$\uplambda_{\nS}\in\left]0,2\right[$,
which is the standard range found in deterministic
methods in the literature \cite{Sadd22,Acnu24,Ecks92,Gols79}.
However, Theorem~\ref{t:5} allows for the use of so-called
super relaxation parameters \cite{Moco25} which may exceed 2
by satisfying
\begin{equation}
\inf_{\nnn}\EE\brk1{\lambda_{\nS}(2-\lambda_{\nS})}>0
\;\;\text{and}\;\;\inf_{\nnn}\PP([\lambda_{\nS}>2])>0.
\end{equation}
Note that the use of super relaxation parameters leads to novel
results and faster convergence; see 
\cite[Section~6]{Moco25} for examples of super relaxation
strategies.
\end{enumerate}
\end{remark}

\begin{remark}
\label{r:17}
As shown in \cite{Acnu24}, through suitable choices of the ambient 
space $\HS$ and the operators $\WS$ and $\CS$, 
Algorithm~\ref{algo:19} captures a broad range of deterministic 
splitting methods. Theorem~\ref{t:5} offers the possibility of
deriving stochastic versions of these algorithms. In the next
sections, we illustrate this process in the cases of the proximal
point algorithm, Kuhn--Tucker projective splitting, and saddle
projective splitting.
\end{remark}

\section{Stochastic proximal point algorithm}
\label{sec:4}

The proximal point algorithm is a classical method for finding a
zero of a maximal monotone operator $\AS\colon\HS\to 2^{\HS}$
\cite{Brez78,Lieu69,Mart70,Roc76a}. In this section, we propose a
stochastic version of it which involves stochastic approximations 
of the resolvents together with random relaxations.

\begin{theorem}
\label{t:3}
Let $\AS\colon\HS\to2^{\HS}$ be a maximally monotone operator such
that $\zer\AS\neq\emp$ and let $x_{0}\in
L^2(\upOmega,\FE,\PP;\HS)$. Iterate
\begin{equation}
\label{e:p25}
\begin{array}{l}
\textup{for}\;\nS=0,1,\ldots\\
\left\lfloor
\begin{array}{l}
\upgamma_{\nS}\in\RPP,\;
e_{\nS}\in L^2(\upOmega,\FE,\PP;\HS)\;
\textup{and}\;\lambda_{\nS}\in 
L^\infty(\upOmega,\FE,\PP;\left]0,2\right[)\\
x_{\nS+1}=x_{\nS}+\lambda_{\nS}
\brk1{\mathsf{J}_{\upgamma_{\nS}\AS}^{}x_{\nS}-e_{\nS}-x_{\nS}}.
\end{array}
\right.\\
\end{array}
\end{equation}
Suppose that, for every $\nnn$, $\lambda_{\nS}$ is independent of
$\upsigma(x_{0},\ldots,x_{\nS},e_{\nS})$, and that one 
of the following holds:
\begin{enumerate}
\item
\label{t:3i}
$\sum_{\nnn}\EE(\lambda_{\nS}(2-\lambda_{\nS}))=\pinf$,
$\sum_{\nnn}\sqrt{\EE|\lambda_{\nS}|^2\EE\norm{e_{\nS}}^2}<\pinf$,
$(\EE\norm{e_{\nS}}^2)_{\nnn}$ is bounded, and 
$(\forall\nnn)$ $\upgamma_{\nS}=1$.
\item
\label{t:3iii}
$\inf_{\nnn}\EE(\lambda_{\nS}(2-\lambda_{\nS}))>0$,
$\inf_{\nnn}\upgamma_{\nS}>0$, and
$\sum_{\nnn}\sqrt{\EE\norm{e_{\nS}}^2}<\pinf$.
\item
\label{t:3ii}
$\sum_{\nnn}\upgamma_{\nS}^2=\pinf$,
$\sum_{\nnn}\sqrt{\EE\norm{e_{\nS}}^2}<\pinf$,
and $(\forall\nnn)$ $\lambda_{\nS}=1\;\Pas$
\end{enumerate}
Then $(x_{\nS})_{\nnn}$ converges weakly $\Pas$ and weakly in 
$L^2(\upOmega,\FE,\PP;\HS)$ to a $(\zer\AS)$-valued random 
variable.
\end{theorem}
\begin{proof}
We apply Theorem~\ref{t:5} with $\WS=\AS$, $\CS=0$ 
(hence $\ZS=\zer\AS$) and
\begin{equation}
\label{e:s1}
(\forall\nnn)\quad
\begin{cases}
w_{\nS}=\mathsf{J}_{\upgamma_{\nS}\AS}x_{\nS}-e_{\nS};\\
w_{\nS}^*=\upgamma_{\nS}^{-1}\brk{x_{\nS}-w_{\nS}};\\
q_{\nS}=w_{\nS};\\
c_{\nS}^*=f_{\nS}^*=0;\\
e_{\nS}^*=-\upgamma_{\nS}^{-1}e_{\nS}.
\end{cases}
\end{equation}
In this setting, it follows from \cite[Proposition~23.22]{Livre1}
that
\begin{equation}
(\forall\nnn)\quad\brk1{w_{\nS}+e_{\nS},w_{\nS}^*+e_{\nS}^*}
=\brk1{\mathsf{J}_{\upgamma_{\nS}\AS}^{}x_{\nS},
\upgamma_{\nS}^{-1}
(x_{\nS}-\mathsf{J}_{\upgamma_{\nS}\AS}^{}x_{\nS})}\in\gra\AS\;\;
\Pas
\end{equation}
and that algorithm~\eqref{e:p25} is an instantiation of
Algorithm~\ref{algo:20} with 
\begin{equation}
\label{e:s7}
(\forall\nnn)\quad
t_{\nS}^*=\upgamma_{\nS}^{-1}\brk{x_{\nS}-w_{\nS}}
\quad\text{and}\quad
\theta_{\nS}=\upgamma_{\nS}.
\end{equation}
We therefore deduce from
Theorem~\ref{t:5}\ref{t:5i-} that the sequence $(x_{\nS})_{\nnn}$
lies in $L^2(\upOmega,\FE,\PP;\HS)$. Next, let us define a family
of auxiliary sequences as follows. For every $\kS\in\NN$, set
\begin{equation}
\label{e:s5}
y_{{0},\kS}=x_{\kS}\quad\text{and}\quad
(\forall\nnn)\quad y_{\nS+1,\kS}
=y_{\nS,\kS}+\lambda_{\nS+\kS}
\brk1{\mathsf{J}_{\upgamma_{\nS+\kS}\AS}^{}y_{\nS,\kS}
-y_{\nS,\kS}}.
\end{equation}
Let $\kS\in\NN$.
Then, as above, $(y_{\nS,\kS})_{\nnn}$ is a sequence generated by 
an instantiation of Algorithm~\ref{algo:20} now initialized at 
$x_{\kS}$ with, for every $\nnn$, $e_{\nS}=0$ and 
$q_{\nS}=\mathsf{J}_{\upgamma_{\nS+\kS}\AS}^{}y_{\nS,\kS}$.
Consequently, Theorem~\ref{t:5}\ref{t:5i} 
asserts that $(\norm{y_{\nS,\kS}})_{\nnn}$ is bounded $\Pas$ and 
that $(\EE\norm{y_{\nS,\kS}}^2)_{\nnn}$ is bounded. 
Additionally, we deduce from Theorem~\ref{t:5}\ref{t:5ii-} that
\begin{equation}
\label{e:s2}
\sum_{\nnn}\EE\brk1{\lambda_{\nS+\kS}(2-\lambda_{\nS+\kS})}
\EE\norm1{y_{\nS,\kS}
-\mathsf{J}_{\upgamma_{\nS+\kS}\AS}^{}y_{\nS,\kS}}^2<\pinf.
\end{equation}
Next, let us show that, under any of scenarios
\ref{t:3i}--\ref{t:3ii}, 
\begin{equation}
\label{e:s87}
\norm{y_{\nS,\kS}-\mathsf{J}_{\upgamma_{\nS+\kS}\AS}
y_{\nS,\kS}}\to 0\;\:\Pas\;\;\text{as}\;\;\nS\to\pinf.
\end{equation}
\begin{itemize}
\item
Suppose that \ref{t:3i} holds. Then we deduce from \eqref{e:s2}
that $\varliminf\EE\norm{y_{\nS,\kS}
-\mathsf{J}_{\AS}^{}y_{\nS,\kS}}^2=0$. In turn, 
Fatou's lemma yields $\varliminf\norm{y_{\nS,\kS}
-\mathsf{J}_{\AS}^{}y_{\nS,\kS}}=0\;\Pas$ Now set
$\mathsf{T}=2\mathsf{J}_{\AS}-\Id$ and recall that it is
nonexpansive \cite[Corollary~23.11(ii)]{Livre1}. Therefore,
\begin{align}
(\forall\nnn)\quad
2\norm1{y_{\nS+1,\kS}
-\mathsf{J}_{\AS}^{}y_{\nS+1,\kS}}
&=\norm1{\mathsf{T}y_{\nS+1,\kS}-y_{\nS+1,\kS}}
\nonumber\\
&=\norm1{\mathsf{T}y_{\nS+1,\kS}-\mathsf{T}y_{\nS,\kS}
+(1-\lambda_{\nS+\kS}/2)(\mathsf{T}y_{\nS,\kS}-y_{\nS,\kS})}
\nonumber\\
&\leq\norm{y_{\nS+1,\kS}-y_{\nS,\kS}}
+(1-\lambda_{\nS+\kS}/2)\norm1{\mathsf{T}y_{\nS,\kS}-y_{\nS,\kS}}
\nonumber\\
&=(\lambda_{\nS+\kS}/2)\norm1{\mathsf{T}y_{\nS,\kS}-y_{\nS,\kS}}
+(1-\lambda_{\nS+\kS}/2)\norm1{\mathsf{T}y_{\nS,\kS}-y_{\nS,\kS}}
\nonumber\\
&=2\norm1{y_{\nS,\kS}-\mathsf{J}_{\AS}^{}y_{\nS,\kS}}\;\;\Pas,
\label{e:s2+}
\end{align}
which shows that 
$(\norm{y_{\nS,\kS}-\mathsf{J}_{\AS}^{}y_{\nS,\kS}})_{\nnn}$ 
decreases $\Pas$ Hence,
$\norm{y_{\nS,\kS}-\mathsf{J}_{\AS}^{}y_{\nS,\kS}}\to 0\;\Pas$
as $\nS\to\pinf$.
\item
Suppose that \ref{t:3iii} or \ref{t:3ii} holds. Then it follows 
from \eqref{e:s2} that 
\begin{equation}
\EE\sum_{\nnn}
\norm1{y_{\nS,\kS}
-\mathsf{J}_{\upgamma_{\nS+\kS}\AS}^{}y_{\nS,\kS}}^2
=\sum_{\nnn}
\EE\norm1{y_{\nS,\kS}
-\mathsf{J}_{\upgamma_{\nS+\kS}\AS}^{}y_{\nS,\kS}}^2
<\pinf.
\end{equation}
Thus $\sum_{\nnn}\norm{y_{\nS,\kS}
-\mathsf{J}_{\upgamma_{\nS+\kS}\AS}y_{\nS,\kS}}^2<\pinf\;\Pas$ and
hence $\norm{y_{\nS,\kS}-\mathsf{J}_{\upgamma_{\nS+\kS}\AS}
y_{\nS,\kS}}\to 0\;\Pas$ as $\nS\to\pinf$.
\end{itemize}
This establishes \eqref{e:s87}. On the other hand, let us note
that, under any of scenarios \ref{t:3i}--\ref{t:3ii},
\begin{equation}
\EE\sum_{\nnn}\abs{\lambda_{\nS}}\,\norm{e_{\nS}}
=\sum_{\nnn}\EE\brk1{\abs{\lambda_{\nS}}\,\norm{e_{\nS}}}
=\sum_{\nnn}\EE\abs{\lambda_{\nS}}\,\EE\norm{e_{\nS}}
\leq\sum_{\nnn}\sqrt{\EE\abs{\lambda_{\nS}}^2\EE\norm{e_{\nS}}^2}
<\pinf.
\end{equation}
Hence $\sum_{\nnn}\abs{\lambda_{\nS}}\,\norm{e_{\nS}}<\pinf\;\Pas$
Consequently, taking into account \eqref{e:p25}, \eqref{e:s5}, and
\eqref{e:s6}, we infer that, for every $\nnn\smallsetminus\{0\}$,
\begin{align}
\label{e:s4}
\norm1{x_{\nS+\kS}-y_{\nS,\kS}}
\leq\sum_{\jS=\kS}^{\nS+\kS-1}
\norm1{\lambda_{\jS}e_{\jS}}
\leq\sum_{\jS=\kS}^{\pinf}
\abs{\lambda_{\jS}}\,\norm{e_{\jS}}
&<\pinf\;\;\Pas
\end{align}
and
\begin{align}
\norm1{x_{\nS+\kS}-y_{\nS,\kS}}_{L^2(\upOmega,\FE,\PP;\HS)}
&\leq\sum_{\jS=\kS}^{\nS+\kS-1}
\norm1{\lambda_{\jS}e_{\jS}}_{L^2(\upOmega,\FE,\PP;\HS)}
\nonumber\\
&=\sum_{\jS=\kS}^{\nS+\kS-1}
\sqrt{\EE\bigl|\lambda_{\jS}\norm{e_{\jS}}\bigr|^2}
\nonumber\\
&=\sum_{\jS=\kS}^{\nS+\kS-1}
\sqrt{\EE|\lambda_{\jS}|^2\EE\norm{e_{\jS}}^2}
\nonumber\\
&\leq\sum_{\jS=\kS}^{\pinf}
\sqrt{\EE|\lambda_{\jS}|^2\EE\norm{e_{\jS}}^2}
\nonumber\\
&<\pinf
\label{e:s4'}
\end{align}
In turn, since $(\EE\norm{y_{\nS,\kS}}^2)_{\nnn}$ is bounded, so is
$(\EE\norm{x_{\nS}}^2)_{\nnn}$. Next, fix $\zS\in\zer\AS$. 
We derive from \eqref{e:GBeps}, \eqref{e:s1},
\eqref{e:s7}, the Cauchy--Schwarz inequality, and \eqref{e:s6} that
\begin{align}
\label{e:s9}
\hspace{-3mm}
\sum_{\nnn}\EE\varepsilon_{\nS}(\cdot,\zS)\EE\lambda_{\nS}
&=\sum_{\nnn}\EE\max
\Big\{0,\EC2{\scal{\zS-\mathsf{J}_{\upgamma_{\nS}\AS}x_{\nS}}
{e_{\nS}}
+\scal{e_{\nS}}{x_{\nS}-\mathsf{J}_{\upgamma_{\nS}\AS}x_{\nS}}
+\norm{e_{\nS}}^2}{\XX_{\nS}}\Big\}\EE\lambda_{\nS}
\nonumber\\
&\leq\sum_{\nnn}\brk2{
\sqrt{\EE\norm{\zS-\mathsf{J}_{\upgamma_{\nS}\AS}x_{\nS}}^2}
+\sqrt{\EE\norm{x_{\nS}-\mathsf{J}_{\upgamma_{\nS}\AS}x_{\nS}}^2}
+\sqrt{\EE\norm{e_{\nS}}^2}}\sqrt{\EE\norm{e_{\nS}}^2}\;
\EE\lambda_{\nS}
\nonumber\\
&\leq\sum_{\nnn}\brk2{
\sqrt{\EE\norm{\zS-x_{\nS}}^2}
+\sqrt{\EE\norm{(\Id-\mathsf{J}_{\upgamma_{\nS}\AS})x_{\nS}
-(\Id-\mathsf{J}_{\upgamma_{\nS}\AS})\zS}^2}
+\sqrt{\EE\norm{e_{\nS}}^2}}
\sqrt{\EE|\lambda_{\nS}|^2\EE\norm{e_{\nS}}^2}
\nonumber\\
&\leq\sum_{\nnn}\brk2{
2\sqrt{\EE\norm{x_{\nS}-\zS}^2}
+\sqrt{\EE\norm{e_{\nS}}^2}}
\sqrt{\EE|\lambda_{\nS}|^2\EE\norm{e_{\nS}}^2}
\nonumber\\
&<\pinf.
\end{align}
We conclude the proof using Theorem~\ref{t:5}\ref{t:5iv-}. 
\begin{itemize}
\item
Convergence under assumption \ref{t:3i} or \ref{t:3iii}:
In view of \eqref{e:s1}, let us show that
\begin{equation}
\label{e:s90}
\begin{cases}
x_{\nS}-w_{\nS}-e_{\nS}=
x_{\nS}-\mathsf{J}_{\upgamma_{\nS}\AS}^{}x_{\nS}
\weakly{0}\;\:\Pas;\\
w_{\nS}+e_{\nS}-q_{\nS}=
\mathsf{J}_{\upgamma_{\nS}\AS}^{}x_{\nS}-x_{\nS}
\to{0}\;\:\Pas;\\
w_{\nS}^*+e_{\nS}^*+\CS q_{\nS}=\upgamma_{\nS}^{-1}
\brk1{x_{\nS}-\mathsf{J}_{\upgamma_{\nS}\AS}^{}x_{\nS}}
\to{0}\;\:\Pas 
\end{cases}
\end{equation}
By invoking \eqref{e:s6}, \eqref{e:s4}, and \eqref{e:s87}, we
obtain 
\begin{align}
\varlimsup_{\mS\to\pinf}\norm1{x_{\mS}
-\mathsf{J}_{\upgamma_{\mS}\AS}x_{\mS}}
&=\varlimsup_{\nS\to\pinf}\norm1{x_{\nS+\kS}
-\mathsf{J}_{\upgamma_{\nS+\kS}\AS}x_{\nS+\kS}}
\nonumber\\
&\leq\varlimsup_{\nS\to\pinf}\Big(
\norm{x_{\nS+\kS}-y_{\nS,\kS}}
+\norm1{\mathsf{J}_{\upgamma_{\nS+\kS}\AS}x_{\nS+\kS}
-\mathsf{J}_{\upgamma_{\nS+\kS}\AS}y_{\nS,\kS}}
+\norm1{y_{\nS,\kS}-\mathsf{J}_{\upgamma_{\nS+\kS}\AS}y_{\nS,\kS}}
\Big)
\nonumber\\
&\leq\varlimsup_{\nS\to\pinf}\Big(
2\norm{x_{\nS+\kS}-y_{\nS,\kS}}
+\norm1{y_{\nS,\kS}-\mathsf{J}_{\upgamma_{\nS+\kS}\AS}y_{\nS,\kS}}
\Big)
\nonumber\\
&\leq\varlimsup_{\nS\to\pinf}\Big(
2\sum_{\jS=\kS}^{\pinf}\norm1{\lambda_{\jS}e_{\jS}}
+\norm1{y_{\nS,\kS}-\mathsf{J}_{\upgamma_{\nS+\kS}\AS}y_{\nS,\kS}}
\Big)
\nonumber\\
&=2\sum_{\jS=\kS}^{\pinf}
|\lambda_{\jS}|\,\norm{e_{\jS}}+\lim_{\nS\to\pinf}
\norm1{y_{\nS,\kS}-\mathsf{J}_{\upgamma_{\nS+\kS}\AS}y_{\nS,\kS}}
\nonumber\\
&=2\sum_{\jS=\kS}^{\pinf}
|\lambda_{\jS}|\,\norm{e_{\jS}}\;\;\Pas
\label{e:s89}
\end{align}
Thus, upon taking the limit as $\kS\to\pinf$ in \eqref{e:s89}, we
obtain $\lim_{\mS\to\pinf}\norm{x_{\mS}
-\mathsf{J}_{\upgamma_{\mS}\AS}x_{\mS}}=0\;\Pas$
Hence, since $(\upgamma_{\nS})_{\nnn}$ is bounded away from $0$, 
\eqref{e:s90} holds. 
\item
Convergence under assumption \ref{t:3ii}:
Note that \eqref{e:s2} yields
$\sum_{\nnn}\upgamma_{\nS+\kS}^2\EE\norm{\upgamma_{\nS+\kS}^{-1}
(y_{\nS,\kS}-\mathsf{J}_{\upgamma_{\nS+\kS}\AS}^{}y_{\nS,\kS})}^2
<\pinf$, which forces
$\varliminf\norm{\upgamma_{\nS+\kS}^{-1}(y_{\nS,\kS}
-\mathsf{J}_{\upgamma_{\nS+\kS}\AS}^{}y_{\nS,\kS})}=0\;\Pas$
Upon invoking \cite[Proposition~23.22]{Livre1} and the
Cauchy--Schwarz inequality, we obtain, for every $\nnn$, 
\begin{align}
0&\leq\dfrac{1}{\upgamma_{\nS+1+\kS}}
\scal3{\mathsf{J}_{\upgamma_{\nS+\kS}\AS}y_{\nS,\kS}
-\mathsf{J}_{\upgamma_{\nS+1+\kS}\AS}y_{\nS+1,\kS}}
{\frac{y_{\nS,\kS}
-\mathsf{J}_{\upgamma_{\nS+\kS}\AS}^{}y_{\nS,\kS}}
{\upgamma_{\nS+\kS}}-\frac{y_{\nS+1,\kS}
-\mathsf{J}_{\upgamma_{\nS+\kS+1}\AS}^{}y_{\nS+1,\kS}}
{\upgamma_{\nS+1+\kS}}}
\nonumber\\[2mm]
&=\scal3{\dfrac{\mathsf{J}_{\upgamma_{\nS+\kS}\AS}y_{\nS,\kS}
-\mathsf{J}_{\upgamma_{\nS+1+\kS}\AS}y_{\nS+1,\kS}}
{\upgamma_{\nS+1+\kS}}}
{\frac{y_{\nS,\kS}
-\mathsf{J}_{\upgamma_{\nS+\kS}\AS}^{}y_{\nS,\kS}}
{\upgamma_{\nS+\kS}}
}
-\norm3{\frac{y_{\nS+1,\kS}
-\mathsf{J}_{\upgamma_{\nS+\kS+1}\AS}^{}y_{\nS+1,\kS}}
{\upgamma_{\nS+1+\kS}}}^2
\nonumber\\[2mm]
&\leq\norm3{\frac{y_{\nS+1,\kS}
-\mathsf{J}_{\upgamma_{\nS+\kS}\AS}^{}y_{\nS+1,\kS}}
{\upgamma_{\nS+1+\kS}}}
\brk4{\norm3{\frac{y_{\nS,\kS}
-\mathsf{J}_{\upgamma_{\nS+\kS}\AS}^{}y_{\nS,\kS}}
{\upgamma_{\nS+\kS}}}-\norm3{\frac{y_{\nS+1,\kS}
-\mathsf{J}_{\upgamma_{\nS+\kS+1}\AS}^{}y_{\nS+1,\kS}}
{\upgamma_{\nS+1+\kS}}}}\;\;\Pas
\end{align}
Hence, $(\norm{\upgamma_{\nS+\kS}^{-1}(y_{\nS,\kS}
-\mathsf{J}_{\upgamma_{\nS+\kS}\AS}^{}y_{\nS,\kS})})_{\nnn}$
decreases $\Pas$, which implies that 
${\upgamma_{\nS+\kS}^{-1}(y_{\nS,\kS}
-\mathsf{J}_{\upgamma_{\nS+\kS}\AS}^{}y_{\nS,\kS})}\to 0\;\Pas$
as $\nS\to\pinf$. Consequently, 
we deduce then from Theorem~\ref{t:5}\ref{t:5iv-} that, for every
$\kS\in\NN$, $(y_{\nS,\kS})_{\nnn}$ converges weakly $\Pas$ and
weakly in $L^2(\upOmega,\FE,\PP;\HS)$ to
some ($\zer\AS$)-valued random variable which we denote by 
$y_{\kS}$. In addition, we deduce from \eqref{e:p25}, \eqref{e:s5},
and \eqref{e:s6} that
\begin{equation}
(\forall\kS\in\NN)(\forall\nnn)\quad
\begin{cases}
\norm{y_{\nS,\kS+1}-y_{\nS+1,\kS}}
\leq\norm{e_{\kS}}\;\Pas;\\[1mm]
\norm{y_{\nS,\kS+1}-y_{\nS+1,\kS}}_{L^2(\upOmega,\FE,\PP;\HS)}
\leq\norm{e_{\kS}}_{L^2(\upOmega,\FE,\PP;\HS)}.
\end{cases}
\end{equation}
In turn, the weak lower semicontinuity of the norm and Fatou's
lemma imply that
\begin{equation}
\label{e:wls}
(\forall\kS\in\NN)\quad
\begin{cases}
\norm{y_{\kS+1}-y_{\kS}}
\leq\varliminf\:\norm1{y_{\nS,\kS+1}-y_{\nS+1,\kS}}
\leq\norm{e_{\kS}}\;\Pas;\\[1mm]
\EE\norm{y_{\kS+1}-y_{\kS}}^2
\leq\varliminf\EE\norm1{y_{\nS,\kS+1}-y_{\nS+1,\kS}}^2
\leq\EE\norm{e_{\kS}}^2.
\end{cases}
\end{equation}
Since $\sum_{\nnn}\norm{e_{\nS}}<\pinf\;\Pas$ and
$\sum_{\nnn}\sqrt{\EE\norm{e_{\nS}}^2}<\pinf$, \eqref{e:wls} shows
that $(y_{\kS})_{\kS\in\NN}$ is a Cauchy sequence both $\Pas$ and
in $L^2(\upOmega,\FE,\PP;\HS)$. Hence, we deduce from \eqref{e:s4},
\eqref{e:s4'}, and \eqref{e:wls} that there exists a 
($\zer\AS$)-valued random variable $y$ such that
\begin{equation}
\begin{cases}
x_{\nS+\kS}-y_{\nS,\kS}\to
0\;\:\Pas\;\text{and in}\;L^2(\upOmega,\FE,\PP;\HS)\;\text{as}\;
\nS\to\pinf\;\text{and}\;\kS\to\pinf;\\
\text{for every}\;\kS\in\NN,\; y_{\nS,\kS}-y_{\kS}\weakly
0\;\:\Pas\;\text{and in}\;L^2(\upOmega,\FE,\PP;\HS)\;\text{as}\;
\nS\to\pinf;\\
y_{\kS}-y\to 0\;\:\Pas\;\text{and in}\;L^2(\upOmega,\FE,\PP;\HS)\;
\;\text{as}\;\kS\to\pinf.
\end{cases}
\end{equation}
Thus, $x_{\nS+\kS}-y=x_{\nS+\kS}-y_{\nS,\kS}+y_{\nS,\kS}
-y_{\kS}+y_{\kS}-y\weakly
0\;\Pas\;\text{and in}\;L^2(\upOmega,\FE,\PP;\HS)\;\;
\text{as}\;\nS\to\pinf\;\text{and}\;\kS\to\pinf$. This
confirms that $(x_{\nS})_{\nnn}$ converges weakly $\Pas$ and weakly
in $L^2(\upOmega,\FE,\PP;\HS)$ to $y$.
\end{itemize}
\end{proof}

\begin{remark}
\label{r:71}
Here are a few commentaries on Theorem~\ref{t:3}.
\begin{enumerate}
\item
\label{r:71i}
In the deterministic setting with $(\uplambda_{\nS})_{\nnn}$ in
$\left]0,2\right[$, Theorem~\ref{t:3}\ref{t:3i} follows from 
\cite[Theorem~2.1(i)(a)]{Joca09}, 
Theorem~\ref{t:3}\ref{t:3iii} was established in 
\cite[Theorem~3]{Ecks92},
and Theorem~\ref{t:3}\ref{t:3ii} was established 
in \cite[Remarque~14(a)]{Brez78}. In this setting, the condition
does not even require that the errors converge weakly to $0$; see
\cite[Remark~2.2(iii)]{Joca09}.
\item
In the case of deterministic relaxations 
$(\uplambda_{\nS})_{\nnn}$ in $\left]0,2\right[$ and constant
proximal parameters $(\upgamma_{\nS})_{\nnn}$, 
the almost sure weak convergence result in 
Theorem~\ref{t:3}\ref{t:3iii} follows from 
\cite[Proposition~5.1]{Siop15}.
\item
As discussed in \cite[Section~5]{Acnu24}, the deterministic
proximal point algorithm can be employed to solve equilibrium
problems beyond the simple inclusion $0\in\AS\xS$. It captures in
particular the method of partial inverses to split multi-operator
inclusions, problems involving resolvent compositions, and
the Chambolle--Pock algorithm. Stochasticity can be introduced in
these methods via Theorem~\ref{t:3}.
\end{enumerate}
\end{remark}

\section{Randomized block-iterative Kuhn--Tucker projective
splitting}
\label{sec:5}

\subsection{Problem setting}
We revisit a multivariate primal-dual inclusion problem studied 
in \cite{MaPr18}. We introduce stochasticity at three levels in the
proposed algorithm by incorporating random relaxations, random
block selection, and stochastic approximations in the
implementation of the operators. See also \cite[Section~9]{Acnu24}
and \cite{Ecks17} for further discussions on the deterministic
setting.

\begin{problem}
\label{prob:4}
Let $(\HS_{\iS})_{\iii}$ and $(\GS_{\kS})_{\kkk}$ be finite
families of Euclidean spaces with respective direct sums
$\HHS=\bigoplus_{\iii}\HS_{\iS}$ and
$\GGS=\bigoplus_{\kkk}\GS_{\kS}$. Denote by
$\boldsymbol{\xS}=(\xS_{\iS})_{\iii}$ a generic element in $\HHS$.
For every $\iS\in\II$ and every $\kS\in\KK$,
$\AS_{\iS}\colon\HS_{\iS}\to 2^{\HS_{\iS}}$ is maximally monotone,
$\BS_{\kS}\colon\GS_{\kS}\to 2^{\GS_{\kS}}$ is maximally monotone, 
and $\LS_{\kS\iS}\colon\HS_{\iS}\to\GS_{\kS}$ is linear. The
objective is to solve the primal problem
\begin{equation}
\label{e:4p}
\text{find}\;\:\overline{\boldsymbol{\xS}}\in\HHS
\;\:\text{such that}\;\:(\forall\iS\in\II)\;\;
0\in\AS_{\iS}\overline{\xS}_{\iS}
+\Sum_{\kkk}\LS_{\kS\iS}^*\Bigg(\BS_{\kS}
\Bigg(\Sum_{\jS\in\mbox{\scriptsize\ttfamily{I}}}
\LS_{\kS\jS}\overline{\xS}_{\jS}\Bigg)\Bigg)
\end{equation}
and the associated dual problem
\begin{equation}
\label{e:4d}
\text{find}\;\:\overline{\boldsymbol{\vS}}^*\in\GGS
\;\:\text{such that}\;\:
(\exi\boldsymbol{\xS}\in\HHS)
\quad
\begin{cases}
(\forall\iS\in\II)\quad\xS_{\iS}\in\AS_{\iS}^{-1}\brk3{
-\Sum_{\kkk}
\LS_{\kS\iS}^*\overline{\vS}_{\kS}^*};\\
(\forall\kS\in\KK)\quad
\Sum_{\iii}\LS_{\kS\iS}\xS_{\iS}
\in\BS_{\kS}^{-1}\overline{\vS}_{\kS}^*.
\end{cases}
\end{equation}
Finally, $\mathscr{P}$ denotes the set of solutions to
\eqref{e:4p} and $\mathscr{D}$ the set of solutions to
\eqref{e:4d}.
\end{problem}

The Kuhn--Tucker operator associated with Problem~\ref{prob:4}
is \cite[Equation~(9.18)]{Acnu24}
\begin{equation}
\label{e:179}
{\bunder{\boldsymbol{\WS}}}\colon\HHS\oplus\GGS\to 
2^{\HHS\oplus\GGS}\colon
(\boldsymbol{\xS},\boldsymbol{\vS}^*)
\mapsto\brk4{
\bigtimes_{\iii}\bigg(\AS_{\iS}\xS_{\iS}
+\sum_{\kkk}\LS^*_{\kS\iS}\vS^*_{\kS}\bigg),
\bigtimes_{\kkk}\brk3{\BS_{\kS}^{-1}\vS^*_{\kS}-
\sum_{\iii}\LS_{\kS\iS}\xS_{\iS}}~}
\end{equation}
As shown in \cite[Lemma~9.7(ii)]{Acnu24},
$\zer{\bunder{\boldsymbol{\WS}}}
\subset\mathscr{P}\times\mathscr{D}$. We can therefore approach
Problem~\ref{prob:4} as an instance of Problem~\ref{prob:19} with
$\bunder{\boldsymbol{\mathsf{C}}}={\boldsymbol{0}}$ and  
$\upalpha$ can be selected arbitrarily in $\RPP$. By applying
Theorem~\ref{t:5} in this context, we obtain a stochastic version
of the deterministic algorithm of \cite{MaPr18}, which relied on
Algorithm~\ref{algo:19}. 

\subsection{Algorithm and convergence}
The following assumptions regulate the way in which the coordinates
and the sets are randomly activated over the course of the
iterations. 

\begin{assumption}
\label{a:2}
$\II$ and $\KK$ are nonempty finite sets, $(\uppi_{\iS})_{\iii}$
and $(\upzeta_{\kS})_{\kkk}$ are in $\left]0,1\right]$, and
$\mathsf{N}\in\NN\smallsetminus\{0\}$. $(I_{\nS})_{\nnn}$ are
nonempty sets composed of elements randomly taken in $\II$ and
$(K_{\nS})_{\nnn}$ are nonempty sets composed of elements randomly
taken in $\KK$. Further, $(I_{\nS},K_{\nS})_{\nnn}$ is an
independent family; i.e., the $\upsigma$-algebras
$(\upsigma(I_{\nS},K_{\nS}))_{\nnn}$ are mutually
independent. Moreover, $I_0=\II$, $K_0=\KK$, and
\begin{equation}
\label{e:2393+}
\quad(\forall\nnn)\;\;
\brk3{\displaystyle(\forall\iS\in\II)\;\;
\PP\brk3{\brk[s]3{\iS\in
\bigcup_{\jS=\nS}^{\nS+\mathsf{N}-1}I_{\jS}}}
\geq\uppi_{\iS}\quad\text{and}\quad
\displaystyle(\forall\kS\in\KK)\;\;
\PP\brk3{\brk[s]3{\kS\in
\bigcup_{\jS=\nS}^{\nS+\mathsf{N}-1}K_{\jS}}}
\geq\upzeta_{\kS}}.
\end{equation}
\end{assumption}

\begin{example}\
\label{ex:u}
\begin{enumerate}
\item
\label{ex:ui}
The (deterministic) rule of \cite[Assumption~2]{Sadd22} satisfies
Assumption~\ref{a:2} by setting, for every $\iS\in\II$ and every 
$\kS\in\KK$, $\uppi_{\iS}=1$ and $\upzeta_{\kS}=1$.
\item
Set, for every $\nnn$, $I_{\nS}=\{i_{\nS}\}$ and
$K_{\nS}=\{k_{\nS}\}$, where $(i_{\nS},k_{\nS})_{\nnn}$ are
i.i.d. random variables uniformly distributed on $\II\times\KK$.
This rule satisfies Assumption~\ref{a:2} for $\mathsf{N}=1$,
$\uppi_{\iS}\equiv 1/\card\II$, and 
$\upzeta_{\kS}\equiv 1/\card\KK$.
\end{enumerate}
\end{example}

\begin{proposition}
\label{p:4}
Let $\II$ be a nonempty finite set and let $(I_{\nS})_{\nnn}$ be
nonempty sets composed of elements randomly taken in $\II$. Suppose
that $I_0=\II$, and that $(\upsigma\brk{I_{\nS}})_{\nnn}$ is an
independent family that satisfies
\begin{equation}
\label{e:53681}
(\exi\mathsf{N}\in\NN\smallsetminus\{0\})
(\exi\uppi_{\iS}\in\left]0,1\right])(\forall\nnn)\quad
\PP\brk3{
\brk[s]3{\iS\in\bigcup_{\jS=\nS}^{\nS+\mathsf{N}-1}I_{\jS}}}
\geq\uppi_{\iS}.
\end{equation}
Set, for every $\nnn$, $\vartheta_{\iS}(\nS)=
\max\menge{\jS\in\NN}{\jS\leq\nS\;\text{and}\;\iS\in I_{\jS}}$.
Further, let $(x_{\nS})_{\nnn}$ be a sequence
in $L^\mathsf{2}(\upOmega,\FE,\PP;\HS)$ such that
$\sum_{\nnn}\EE\|x_{\nS+1}-x_{\nS}\|^2<\pinf\;\Pas$
Then $x_{\vartheta_{\iS}(\nS)}^{}-x_{\nS}^{}\to 0$ in 
$L^1(\upOmega,\FE,\PP;\HS)$.
\end{proposition}
\begin{proof}
Note that $(\forall\nnn)$
$\vartheta_{\iS}(\nS)\in\{0,\ldots,\nS\}\;\Pas$ Hence,
Lemma~\ref{l:101} ensures that, for every $\nnn$,
$x_{\vartheta_{\iS}(\nS)}^{}\in L^2(\upOmega,\FE,\PP;\HS)$.
On the other hand, for every $\jS\in\NN$, 
$\brk[s]{\iS\in I_{\jS}}\in\upsigma(I_{\jS})$. 
Thus, it follows from the independence condition and
\eqref{e:53681} that
\begin{align}
(\forall\nnn)\quad
\PP\brk3{\brk[s]3{\iS\not\in\bigcup_{\jS=\nS}^{\pinf}I_{\jS}}}
&=\PP\brk3{\bigcap_{\jS=\nS}^{\pinf}
\brk[s]3{\iS\in\complement I_{\jS}}}
\nonumber\\
&=\PP\brk3{\lim_{0<\mathsf{m}\to\pinf}
\bigcap_{\jS=\nS}^{\nS+\mathsf{m}\mathsf{N}-1}
\brk[s]3{\iS\in\complement I_{\jS}}}
\nonumber\\
&=\lim_{0<\mathsf{m}\to\pinf}\PP\brk3{
\bigcap_{\jS=\nS}^{\nS+\mathsf{m}\mathsf{N}-1}
\brk[s]3{\iS\in\complement I_{\jS}}}
\nonumber\\
&=\lim_{0<\mathsf{m}\to\pinf}
\prod_{\kS=0}^{\mathsf{m}-1}\PP\brk3{
\bigcap_{\jS=\nS+\kS\mathsf{N}}^{\nS+(\kS+1)\mathsf{N}-1}
\brk[s]3{\iS\in\complement I_{\jS}}}
\nonumber\\
&=\lim_{0<\mathsf{m}\to\pinf}
\prod_{\kS=0}^{\mathsf{m}-1}\PP\brk3{\complement
\brk[s]3{\iS\in\bigcup_{\jS=\nS+\kS
\mathsf{N}}^{(\nS+\kS\mathsf{N})+\mathsf{N}-1}I_{\jS}}}
\nonumber\\
&\leq\lim_{0<\mathsf{m}\to\pinf}(1-\uppi_{\iS})^{\mathsf{m}}
\nonumber\\
&=0.
\end{align}
Therefore $\vartheta_{\iS}(\nS)\to\pinf\;\Pas$ as $\nS\to\pinf$
and, since 
$\sum_{\nnn}\norm{x_{\nS+1}-x_{\nS}}^2<\pinf\;\Pas$, we have
$\sum_{\jS\geq\vartheta_{\iS}(\nS)}\norm{x_{\jS+1}-x_{\jS}}^2
\downarrow 0\;\Pas$ as $\nS\to\pinf$. Thus, 
\begin{equation}
(\forall\nnn)\quad 0
\leq\sum_{\jS\geq\vartheta_{\iS}(\nS)}\norm{x_{\jS+1}-x_{\jS}}^2
\leq\sum_{\jjj}\norm{x_{\jS+1}-x_{\jS}}^2\in
L^1(\upOmega,\FE,\PP;\RR),
\end{equation}
from which we deduce via \cite[Theorem~2.6.1(b)]{Shir16} that
$\EE\sum_{\jS\geq\vartheta_{\iS}(\nS)}
\norm{x_{\jS+1}-x_{\jS}}^2\to 0$ as $\nS\to\pinf$. 
On the other hand, let $\nnn$ and $\mathsf{m}\in\NN$ be such
that $\mathsf{m}\mathsf{N}\leq\nS<(\mathsf{m}+1)\mathsf{N}$. Then
\begin{align}
&\EE\brk1{\nS-\vartheta_{\iS}(\nS)}\nonumber\\
&\quad=\sum_{\lS=0}^{\nS-1}(\nS-\lS)
\PP\brk1{\brk[s]1{\iS\in I_{\lS}}}
\PP\brk3{\brk[s]3{\iS\not\in\bigcup_{\jS=\lS+1}^{\nS}I_{\jS}}}
\nonumber\\
&\quad\leq\sum_{\lS=0}^{\nS-1}(\nS-\lS)
\PP\brk3{\brk[s]3{\iS\not\in\bigcup_{\jS=\lS+1}^{\nS}I_{\jS}}}
\nonumber\\
&\quad\leq\mathsf{N}(\nS-\mathsf{m}\mathsf{N})
+\sum_{\kS=0}^{\mathsf{m}-1}
\sum_{\lS=\kS\mathsf{N}}^{(\kS+1)\mathsf{N}-1}(\nS-\lS)
\PP\brk3{\brk[s]3{\iS\not\in\bigcup_{\jS=\lS+1}^{\nS}I_{\jS}}}
\nonumber\\
&\quad\leq\mathsf{N}^2+\sum_{\kS=0}^{\mathsf{m}-1}
\sum_{\lS=\kS\mathsf{N}}^{(\kS+1)\mathsf{N}-1}(\nS-\lS)
\PP\brk3{\brk[s]3{\iS\not\in
\bigcup_{\jS=(\kS+1)\mathsf{N}}^{\mathsf{m}\mathsf{N}-1}I_{\jS}}}
\nonumber\\
&\quad\leq\mathsf{N}^2
+\sum_{\kS=0}^{\mathsf{m}-1}(\nS-\kS\mathsf{N})
\sum_{\lS=\kS\mathsf{N}}^{(\kS+1)\mathsf{N}-1}
(1-\uppi_{\iS})^{\mathsf{m}-\kS-1}\nonumber\\
&\quad=\mathsf{N}^2+
\sum_{\kS=0}^{\mathsf{m}-1}(\nS-\kS\mathsf{N})
\mathsf{N}(1-\uppi_{\iS})^{\mathsf{m}-\kS-1}\nonumber\\
&\quad\leq\mathsf{N}^2
+\sum_{\kS=0}^{\mathsf{m}-1}\brk1{(\mathsf{m}+1)\mathsf{N}
-\kS\mathsf{N}}\mathsf{N}(1-\uppi_{\iS})^{\mathsf{m}-\kS-1}
\nonumber\\
&\quad=\mathsf{N}^2
\brk3{1+\sum_{\kS=0}^{\mathsf{m}-1}\brk1{\mathsf{m}+1
-\kS}(1-\uppi_{\iS})^{\mathsf{m}-\kS-1}}
\nonumber\\
&\quad=\mathsf{N}^2\brk3{1+\sum_{\lS=0}^{\mathsf{m}-1}
\brk1{\lS+2}(1-\uppi_{\iS})^{\lS}}
\nonumber\\
&\quad=\mathsf{N}^2\brk3{1+\sum_{\lS=0}^{\mathsf{m}-1}
\lS(1-\uppi_{\iS})^{\lS}+\sum_{\lS=0}^{\mathsf{m}-1}
2(1-\uppi_{\iS})^{\lS}}
\nonumber\\
&\quad=\mathsf{N}^2\brk3{1+(1-\uppi_{\iS})
\dfrac{1-\mathsf{m}(1-\uppi_{\iS})^{\mathsf{m}-1}
+(\mathsf{m}-1)(1-\uppi_{\iS})^\mathsf{m}}{\uppi_{\iS}^2}
+2\dfrac{1-(1-\uppi_{\iS})^\mathsf{m}}
{\uppi_{\iS}}}
\nonumber\\
&\quad=\mathsf{N}^2\brk3{1+
\dfrac{1-\uppi_{\iS}-\mathsf{m}(1-\uppi_{\iS})^{\mathsf{m}}
+(\mathsf{m}-1)(1-\uppi_{\iS})^{\mathsf{m}+1}}{\uppi_{\iS}^2}
+\dfrac{2\uppi_{\iS}
+2(1-\uppi_{\iS})^{\mathsf{m}+1}-2(1-\uppi_{\iS})^\mathsf{m}}
{\uppi_{\iS}^2}}
\nonumber\\
&\quad=\mathsf{N}^2\brk3{1
+\dfrac{(\mathsf{m}+1)(1-\uppi_{\iS})^{\mathsf{m}+1}-(\mathsf{m}+2)
(1-\uppi_{\iS})^{\mathsf{m}}+1+\uppi_{\iS}}{\uppi_{\iS}^2}},
\end{align}
which shows that $\varlimsup\EE(\nS-\vartheta_{\iS}(\nS))
\leq\mathsf{N}^2(1+(1+\uppi_{\iS})/\uppi_{\iS}^2)<\pinf$.
Thus,
$\EE\norm{x_{\nS}^{}-x_{\vartheta_{\iS}(\nS)}^{}}
\leq\EE\sum_{\jS=\vartheta_{\iS}(\nS)}^{\nS}
\norm{x_{\jS+1}-x_{\jS}}
\leq\EE\brk{\sqrt{\nS+1-\vartheta_{\iS}(\nS)}
\;\sqrt{\sum_{\jS=\vartheta_{\iS}(\nS)}^{\nS}
\norm{x_{\jS+1}-x_{\jS}}^2}\,}
\leq\sqrt{1+\EE\brk1{\nS-\vartheta_{\iS}(\nS)}}
\;\sqrt{\EE\sum_{\jS=\vartheta_{\iS}(\nS)}^{\pinf}
\norm{x_{\jS+1}-x_{\jS}}^2}
\to 0$.
This confirms that $x_{\vartheta_{\iS}(\nS)}^{}-x_{\nS}^{}\to 0$ in
$L^1(\upOmega,\FE,\PP;\HS)$.
\end{proof}

\begin{algorithm}
\label{algo:7}
Consider the setting of Problem~\ref{prob:4} and suppose that
Assumption~\ref{a:2} is in force. Let $\upvarepsilon\in\zeroun$,
let $\uprho\in[2,\pinf[$ and, for every
$(\iS,\kS)\in\II\times\KK$, let $x_{\iS,0}\in
L^2(\upOmega,\FE,\PP;\HS_{\iS})$ and $v_{\kS,0}^*\in
L^2(\upOmega,\FE,\PP;\GS_{\kS})$. Iterate
\begin{equation}
\label{e:long7}
\begin{array}{l}
\text{for}\;\nS=0,1,\ldots\\
\left\lfloor
\begin{array}{l}
\text{for every}\;\iS\in I_{\nS}\\
\left\lfloor
\begin{array}{l}
\upgamma_{\iS,\nS}\in[\upvarepsilon,1/\upvarepsilon];\;
e_{\iS,\nS}\in L^2(\upOmega,\FE,\PP;\HS_{\iS});\;
l_{\iS,\nS}^*=\sum_{\kkk}\LS_{\kS\iS}^*v_{\kS,\nS}^*;\\
y_{\iS,\nS}=x_{\iS,\nS};\;
a_{\iS,\nS}=\mathsf{J}_{\upgamma_{\iS,\nS}\AS_{\iS}}\big(
x_{\iS,\nS}-\upgamma_{\iS,\nS}l_{\iS,\nS}^*\big)+e_{\iS,\nS};\;
a_{\iS,\nS}^*=\upgamma_{\iS,\nS}^{-1}\big(x_{\iS,\nS}
-a_{\iS,\nS}\big)-l_{\iS,\nS}^*;
\end{array}
\right.\\
\text{for every}\;\iS\in\II\smallsetminus I_{\nS}\\
\left\lfloor
\begin{array}{l}
y_{\iS,\nS}=y_{\iS,\nS-1};\;
a_{\iS,\nS}=a_{\iS,\nS-1};\;
a_{\iS,\nS}^*=a_{\iS,\nS-1}^*;\;
e_{\iS,\nS}=e_{\iS,\nS-1};
\end{array}
\right.\\
\text{for every}\;\kS\in K_{\nS}\\
\left\lfloor
\begin{array}{l}
\upmu_{\kS,\nS}\in[\upvarepsilon,1/\upvarepsilon];\;
f_{\kS,\nS}\in L^2(\upOmega,\FE,\PP;\GS_{\kS});\;
l_{\kS,\nS}=\sum_{\iii}\LS_{\kS\iS}x_{\iS,\nS};\\
y_{\kS,\nS}^*=v_{\kS,\nS}^*;\;
b_{\kS,\nS}=\mathsf{J}_{\upmu_{\kS,\nS}\BS_{\kS}}
\big(l_{\kS,\nS}+\upmu_{\kS,\nS}v_{\kS,\nS}^*\big)+f_{\kS,\nS};\;
b_{\kS,\nS}^*=v_{\kS,\nS}^*+\upmu_{\kS,\nS}^{-1}\big(l_{\kS,\nS}
-b_{\kS,\nS}\big);
\end{array}
\right.\\
\text{for every}\;\kS\in\KK\smallsetminus K_{\nS}\\
\left\lfloor
\begin{array}{l}
y_{\kS,\nS}^*=y_{\kS,\nS-1}^*;\;
b_{\kS,\nS}=b_{\kS,\nS-1};\;
b_{\kS,\nS}^*=b_{\kS,\nS-1}^*;\;
f_{\kS,\nS}=f_{\kS,\nS-1};
\end{array}
\right.\\
\text{for every}\;\iS\in\II\\
\left\lfloor
\begin{array}{l}
t_{\iS,\nS}^*=a_{\iS,\nS}^*
+\sum_{\kkk}\LS_{\kS\iS}^*b_{\kS,\nS}^*;
\end{array}
\right.\\[1mm]
\text{for every}\;\kS\in\KK\\
\left\lfloor
\begin{array}{l}
t_{\kS,\nS}=b_{\kS,\nS}-\sum_{\iii}\LS_{\kS\iS}a_{\iS,\nS};
\end{array}
\right.\\[2mm]
\phi_{\nS}=\sum_{\iii}\scal{x_{\iS,\nS}
-y_{\iS,\nS}}{t_{\iS,\nS}^*}
+\sum_{\kkk}\scal{v_{\kS,\nS}^*-y_{\kS,\nS}^*}{t_{\kS,\nS}};\\
\psi_{\nS}=\sum_{\iii}\brk1{
\scal{y_{\iS,\nS}-a_{\iS,\nS}}{t_{\iS,\nS}^*}
-\upgamma_{\iS,\nS}^{-1}\norm{y_{\iS,\nS}-a_{\iS,\nS}}^2}
+\sum_{\kkk}\brk1{
\scal{y_{\kS,\nS}^*-b_{\kS,\nS}^*}{t_{\kS,\nS}}
-\upmu_{\kS,\nS}\norm{y_{\kS,\nS}^*-b_{\kS,\nS}^*}^2};\\[1mm]
\nu_{\nS}=\max\brk[c]{0,\phi_{\nS}-\psi_{\nS}};\\
\Delta_{\nS}=
\sum_{\iii}\scal{x_{\iS,\nS}-a_{\iS,\nS}}{t_{\iS,\nS}^*}
+\sum_{\kkk}\scal{v_{\kS,\nS}^*-b_{\kS,\nS}^*}{t_{\kS,\nS}}
-\nu_{\nS};\\[2mm]
\theta_{\nS}=
\dfrac{\mathsf{1}_{[\Delta_{\nS}>0]}\Delta_{\nS}}
{\sum_{\iii}\norm{t_{\iS,\nS}^*}^2
+\sum_{\kkk}\norm{t_{\kS,\nS}}^2
+\mathsf{1}_{[\Delta_{\nS}\leq 0]}};\\
\lambda_{\nS}\in
L^\infty(\upOmega,\FE,\PP;[\upvarepsilon,\uprho]);\\
\text{for every}\;\iS\in\II\\
\left\lfloor
\begin{array}{l}
x_{\iS,\nS+1}=x_{\iS,\nS}-\lambda_{\nS}\theta_{\nS}t_{\iS,\nS}^*;
\end{array}
\right.\\[1mm]
\text{for every}\;\kS\in\KK\\
\left\lfloor
\begin{array}{l}
v_{\kS,\nS+1}^*=v_{\kS,\nS}^*-\lambda_{\nS}\theta_{\nS}t_{\kS,\nS}.
\end{array}
\right.\\[1mm]
\end{array}
\right.
\end{array}
\end{equation}
\end{algorithm}

Let us derive the convergence properties of Algorithm~\ref{algo:7}.

\begin{theorem}
\label{t:7}
Consider the setting of Algorithm~\ref{algo:7} and suppose that
$\mathscr{D}\neq\emp$,
$\inf_{\nnn}\EE(\lambda_{\nS}(2-\lambda_{\nS}))>0$, and
\begin{equation}
\label{t:7as}
\begin{cases}
\upsigma\brk1{\{I_{\nS},K_{\nS}\}_{\nnn}}
\;\text{is independent of}\;
\upsigma\brk1{\{x_{0},e_{\iS,\nS},f_{\kS,\nS},\lambda_{\nS}
\}_{\nnn,\iii,\kkk}},
\\
(\forall\nnn)\;
\lambda_{\nS}\;\text{is independent of}\;
\upsigma\brk2{
\brk[c]1{x_{\iS,\mS}}_{\iii,0\leq\mS\leq\nS}^{}
\!\cup\!\brk[c]1{v_{\kS,\mS}^*}_{\kkk,0\leq\mS\leq\nS}^{}
\!\cup\!
\brk[c]{e_{\iS,\nS}}_{\iii}
\!\cup\!
\brk[c]{f_{\kS,\nS}}_{\kkk}
\!\cup\!
\brk[c]{I_{\nS},K_{\nS}}},\\[5mm]
(\forall(\iS,\kS)\in\II\times\KK)\quad
\sum_{\nnn}\sqrt{
\EE\norm{e_{\iS,\nS}}^2}<\pinf\;
\text{and}\;
\sum_{\nnn}\sqrt{
\EE\norm{f_{\kS,\nS}}^2}<\pinf.
\end{cases}
\end{equation}
Then the following hold:
\begin{enumerate}
\item
\label{t:7i}
Let $\iS\in\II$. Then $(x_{\iS,\nS})_{\nnn}$ lies in
$L^2(\upOmega,\FE,\PP;\HS_{\iS})$ and
$\sum_{\nnn}\EE\|x_{\iS,\nS+1}-x_{\iS,\nS}\|^2<\pinf$.
\item
\label{t:7ii}
Let $\kS\in\KK$. Then $(v_{\kS,\nS}^*)_{\nnn}$ lies in
$L^2(\upOmega,\FE,\PP;\GS_{\kS})$ and
$\sum_{\nnn}\EE\|v_{\kS,\nS+1}^*-v_{\kS,\nS}^*\|^2<\pinf$.
\item
\label{t:7iii-}
Let $\iS\in\II$ and $\kS\in\KK$. Then
$x_{\iS,\nS}-a_{\iS,\nS}\Pto 0$ and
$v_{\kS,\nS}^*-b_{\kS,\nS}^*\Pto 0$.
\item
\label{t:7iii}
There exist a $\mathscr{P}$-valued random variable
$\overline{\boldsymbol{x}}$ and a $\mathscr{D}$-valued random
variable $\overline{\boldsymbol{v}}^*$ such that, for every
$\iS\in\II$ and every $\kS\in\KK$,
$x_{\iS,\nS}\to\overline{x}_{\iS}\;\Pas$ 
and $v_{\kS,\nS}^*\to\overline{v}_{\kS}^*\;\Pas$
\end{enumerate}
\end{theorem}
\begin{proof}
Define
\begin{equation}
\label{e:activ}
(\forall\nnn)\quad
\vartheta_{\iS}(\nS)=\max\menge{\jjj}{\jS\leq\nS
\;\:\text{and}\;\:\iS\in I_{\jS}}
\quad\text{and}\quad
\varrho_{\kS}(\nS)=\max\menge{\jjj}{\jS\leq\nS
\;\:\text{and}\;\:\kS\in K_{\jS}}.
\end{equation}
It follows from \eqref{e:long7} that, for every $\nnn$, every
$\iS\in\II$, and every $\kS\in\KK$,
\begin{equation}
\label{e:fr}
\begin{cases}
a_{\iS,\nS}=a_{\iS,\vartheta_{\iS}(\nS)};\;
a_{\iS,\nS}^*=a_{\iS,\vartheta_{\iS}(\nS)}^*;\;
e_{\iS,\nS}=e_{\iS,\vartheta_{\iS}(\nS)};\;
l_{\iS,\nS}^*=\Sum_{\kkk}\LS_{\kS\iS}^*
v_{\kS,\vartheta_{\iS}(\nS)}^*;\;
y_{\iS,\nS}=x_{\iS,\vartheta_{\iS}(\nS)};\\
b_{\kS,\nS}=b_{\kS,\varrho_{\kS}(\nS)};\;
b_{\kS,\nS}^*=b_{\kS,\varrho_{\kS}(\nS)}^*;\;
f_{\kS,\nS}=f_{\kS,\varrho_{\kS}(\nS)};\;
l_{\kS,\nS}=\Sum_{\iii}\LS_{\kS\iS}x_{\iS,\varrho_{\kS}(\nS)};\;
y_{\kS,\nS}^*=v_{\kS,\varrho_{\kS}(\nS)}^*.
\end{cases}
\end{equation}
Let $\upalpha\in\RPP$ and suppose that $\bunder{\boldsymbol{\zS}}'$
is a unit vector in $\HHS\oplus\GGS$. To match the notation of
Theorem~\ref{t:5}, set
\begin{equation}
\label{e:6var}
(\forall\nnn)\quad
\begin{cases}
\bunder{\boldsymbol{x}}_{\nS}
=\brk1{(x_{\iS,\nS})_{\iii},(v_{\kS,\nS}^*)_{\kkk}};\\
\bunder{\boldsymbol{w}}_{\nS}
=\brk1{(a_{\iS,\nS})_{\iii},(b_{\kS,\nS}^*)_{\kkk}};\\
\bunder{\boldsymbol{w}}_{\nS}^*=\bunder{\boldsymbol{t}}_{\nS}^*
=\brk1{(t_{\iS,\nS}^*)_{\iii},(t_{\kS,\nS})_{\kkk}};\\
\bunder{\boldsymbol{c}}_{\nS}^*
=\bunder{\boldsymbol{f}}_{\nS}^*=\bunder{\boldsymbol{0}};\\
\bunder{\boldsymbol{q}}_{\nS}=\bunder{\boldsymbol{w}}_{\nS}
+2\sqrt{\upalpha\nu_{\nS}}\,
\bunder{\boldsymbol{\zS}}';\\
\bunder{\boldsymbol{e}}_{\nS}
=\brk1{(-e_{\iS,\nS})_{\iii},
(\upmu_{\kS,\nS}^{-1}f_{\kS,\nS})_{\kkk}};\\[1mm]
\bunder{\boldsymbol{e}}_{\nS}^*
=\brk3{\brk3{\upgamma_{\iS,\nS}^{-1}e_{\iS,\nS}
+\sum_{\kkk}\upmu_{\kS,\nS}^{-1}
\LS_{\kS\iS}^*f_{\kS,\nS}}_{\iii},
\brk3{-f_{\kS,\nS}+\sum_{\iii}\LS_{\kS\iS}e_{\iS,\nS}}_{\kkk}}.
\end{cases}
\end{equation}
It is clear from Lemmas~\ref{l:100} and \ref{l:101}, and
an inductive argument that the variables defined in \eqref{e:6var}
belong to $L^2(\upOmega,\FE,\PP;\HHS\oplus\GGS)$.
Now, let us show that \eqref{e:long7} is a realization of
\eqref{e:a20} in $\HHS\oplus\GGS$, with $\bunder{\boldsymbol{\WS}}$
as in \eqref{e:179} and $\bunder{\boldsymbol{\mathsf{C}}}
=\boldsymbol{0}$. Let $\nnn$, let $\iS\in\II$, and let $\kS\in\KK$.
Then we deduce from \eqref{e:long7} and
\cite[Proposition~23.2(ii)]{Livre1} that
\begin{equation}
\label{e:6gr}
a_{\iS,\nS}^*+\upgamma_{\iS,\nS}^{-1}e_{\iS,\nS}
\in\AS_{\iS}\brk1{a_{\iS,\nS}-e_{\iS,\nS}}
\quad\text{and}\quad
b_{\kS,\nS}-f_{\kS,\nS}\in\BS_{\kS}^{-1}\brk1{b_{\kS,\nS}^*
+\upmu_{\kS,\nS}^{-1}f_{\kS,\nS}}\;\;\Pas
\end{equation}
Thus, in view of \eqref{e:179}, \eqref{e:long7}, and \eqref{e:6gr},
$\brk{\bunder{\boldsymbol{w}}_{\nS}+\bunder{\boldsymbol{e}}_{\nS},
\bunder{\boldsymbol{w}}_{\nS}^*+\bunder{\boldsymbol{e}}_{\nS}^*}
\in\gra\bunder{\boldsymbol{\WS}}\;\Pas$
Moreover, $\bunder{\boldsymbol{\mathsf{C}}}=\boldsymbol{0}$ is
$\upalpha$-cocoercive, and \eqref{e:6var} yields
$\bunder{\boldsymbol{c}}_{\nS}^*
+\bunder{\boldsymbol{f}}_{\nS}^*
=\bunder{\boldsymbol{\mathsf{C}}}\bunder{\boldsymbol{q}}_{\nS}
=\bunder{\boldsymbol{0}}$,
$\bunder{\boldsymbol{t}}_{\nS}^*
=\bunder{\boldsymbol{w}}_{\nS}^*
+\bunder{\boldsymbol{c}}_{\nS}^*$, and 
$\norm{\bunder{\boldsymbol{w}}_{\nS}
-\bunder{\boldsymbol{q}}_{\nS}}^2/(4\upalpha)=\nu_{\nS}$.
Hence, \eqref{e:long7} yields
\begin{equation}
\Delta_{\nS}=\scal1{\bunder{\boldsymbol{x}}_{\nS}
-\bunder{\boldsymbol{w}}_{\nS}}{\bunder{\boldsymbol{t}}_{\nS}^*}
-\dfrac{1}{4\upalpha}\norm1{\bunder{\boldsymbol{w}}_{\nS}
-\bunder{\boldsymbol{q}}_{\nS}}^2.
\end{equation}
Altogether, \eqref{e:long7} is a realization of
\eqref{e:a20} and, therefore, we apply Theorem~\ref{t:5} to
study its asymptotic behavior. Thus, let us show that, for every
$\bunder{\boldsymbol{\zS}}\in\zer\bunder{\boldsymbol{\WS}}$,
$\sum_{\nnn}\EE\varepsilon_{\nS}(\cdot,\bunder{\boldsymbol{\zS}})
\EE\lambda_{\nS}<\pinf$. To this end, let $\nnn$ and note that,
for every $\iS\in\II$, $\upgamma_{\iS,\nS}\norm{t_{\iS,\nS}^*
-\upgamma_{\iS,\nS}^{-1}(y_{\iS,\nS}-a_{\iS,\nS})}^2\geq 0$ and,
hence, 
\begin{equation}
\label{e:6sqp}
2\brk2{\scal1{y_{\iS,\nS}-a_{\iS,\nS}}{t_{\iS,\nS}^*}
-\upgamma_{\iS,\nS}^{-1}\norm1{y_{\iS,\nS}-a_{\iS,\nS}}^2}
+\upgamma_{\iS,\nS}^{-1}\norm1{y_{\iS,\nS}-a_{\iS,\nS}}^2
\leq\upgamma_{\iS,\nS}\norm1{t_{\iS,\nS}^*}^2\;\;\Pas
\end{equation}
Likewise, for every $\kS\in\KK$,
\begin{equation}
\label{e:6sqd}
2\brk2{\scal1{y_{\kS,\nS}^*-b_{\kS,\nS}^*}{t_{\kS,\nS}}
-\upmu_{\kS,\nS}\norm1{y_{\kS,\nS}^*-b_{\kS,\nS}^*}^2}
+\upmu_{\kS,\nS}\norm1{y_{\kS,\nS}^*-b_{\kS,\nS}^*}^2
\leq\upmu_{\kS,\nS}^{-1}\norm1{t_{\kS,\nS}}^2\;\;\Pas
\end{equation}
Hence, upon summing over $\iS\in\II$ and $\kS\in\KK$, we obtain
\begin{equation}
\label{e:eptn}
2\psi_{\nS}
+\Sum_{\iii}\upgamma_{\iS,\nS}^{-1}
\norm1{y_{\iS,\nS}-a_{\iS,\nS}}^2
+\Sum_{\kkk}\upmu_{\kS,\nS}
\norm1{y_{\kS,\nS}^*-b_{\kS,\nS}^*}^2
\leq\Sum_{\iii}\upgamma_{\iS,\nS}\norm1{t_{\iS,\nS}^*}^2
+\Sum_{\kkk}\upmu_{\kS,\nS}^{-1}\norm1{t_{\kS,\nS}}^2
\leq\dfrac{1}{\upvarepsilon}
\norm1{\bunder{\boldsymbol{t}}_{\nS}^*}^2\;\;\Pas
\end{equation}
Furthermore, it follows from \eqref{e:long7} and \eqref{e:eptn} 
that
\begin{align}
\label{e:Del1}
\Delta_{\nS}
&=\Sum_{\iii}\scal1{y_{\iS,\nS}-a_{\iS,\nS}}{t_{\iS,\nS}^*}
+\Sum_{\kkk}\scal1{y_{\kS,\nS}^*-b_{\kS,\nS}^*}{t_{\kS,\nS}}
+\phi_{\nS}-\nu_{\nS}
\nonumber\\
&=\Sum_{\iii}\upgamma_{\iS,\nS}^{-1}
\norm1{y_{\iS,\nS}-a_{\iS,\nS}}^2
+\Sum_{\kkk}\upmu_{\kS,\nS}\norm1{y_{\kS,\nS}^*-b_{\kS,\nS}^*}^2
+\psi_{\nS}+\phi_{\nS}-\nu_{\nS}
\nonumber\\
&\leq\Sum_{\iii}\upgamma_{\iS,\nS}^{-1}
\norm1{y_{\iS,\nS}-a_{\iS,\nS}}^2
+\Sum_{\kkk}\upmu_{\kS,\nS}\norm1{y_{\kS,\nS}^*-b_{\kS,\nS}^*}^2
+2\psi_{\nS}
\nonumber\\
&\leq\dfrac{1}{\upvarepsilon}
\norm1{\bunder{\boldsymbol{t}}_{\nS}^*}^2\;\;\Pas
\end{align}
Consequently, $\theta_{\nS}\leq1/\upvarepsilon\;\Pas$
Next, fix
$\bunder{\boldsymbol{\zS}}
=(\overline{\boldsymbol{\xS}},\overline{\boldsymbol{\vS}}^*)
\in\zer\bunder{\boldsymbol{\WS}}$.
Then \eqref{e:179} and 
\cite[Proposition~23.2(ii)]{Livre1} yield, for every
$\upgamma\in\RPP$ and every $\upmu\in\RPP$,
\begin{equation}
\label{e:6zfix}
\begin{cases}
(\forall\iS\in\II)\quad
\overline{\xS}_{\iS}=\mathsf{J}_{\upgamma\AS_{\iS}}
\brk3{\overline{\xS}_{\iS}-\upgamma\Sum_{\kkk}
\LS_{\kS\iS}^*\overline{\vS}_{\kS}^*}\\[1mm]
(\forall\kS\in\KK)\quad
\Sum_{\iii}\LS_{\kS\iS}
\overline{\xS}_{\iS}=\mathsf{J}_{\upmu\BS_{\kS}}
\brk3{\Sum_{\iii}\LS_{\kS\iS}
\overline{\xS}_{\iS}+\upmu\overline{\vS}_{\kS}^*}.
\end{cases}
\end{equation}
Now define
\begin{equation}
\label{e:L2}
\boldsymbol{\LS}\colon\HHS\to\GGS\colon\boldsymbol{\xS}\mapsto
\brk3{\sum_{\iii}\LS_{\kS\iS}\xS_{\iS}}_{\kkk},\;\;
\text{with adjoint}\;\;
\LLS^*\colon\GGS\to\HHS\colon\boldsymbol{\vS}^*
\mapsto\brk3{\sum_{\kkk}\LS_{\kS\iS}^*{\vS}^*_{\kS}}_{\iii},
\end{equation}
and $\updelta=(\sum_{\iii}\sum_{\kkk}\norm{\LS_{\kS\iS}}^2)^{1/2}$.
Then $\norm{\boldsymbol{\LS}}\leq\updelta$ and
$\norm{\boldsymbol{\LS^*}}\leq\updelta$. Moreover, by invoking
the nonexpansiveness of the resolvents, we deduce from
\eqref{e:long7} and \eqref{e:6zfix} that
\begin{equation}
\label{e:6ba}
\begin{cases}
(\forall\iS\in\II)\quad
\norm1{a_{\iS,\nS}-\overline{\xS}_{\iS}}
\leq\norm{e_{\iS,\nS}}
+\dfrac{1+\updelta}{\upvarepsilon}
\norm1{\bunder{\boldsymbol{x}}_{\vartheta_{\iS}(\nS)}
-\bunder{\boldsymbol{\zS}}}\;\;\Pas\\
(\forall\kS\in\KK)\quad
\norm{b_{\kS,\nS}-\sum_{\iii}\LS_{\kS\iS}
\overline{\xS}_{\iS}}
\leq\norm{f_{\kS,\nS}}
+\dfrac{1+\updelta}{\upvarepsilon}
\norm1{\bunder{\boldsymbol{x}}_{\varrho_{\kS}(\nS)}
-\bunder{\boldsymbol{\zS}}}\;\;\Pas
\end{cases}
\end{equation}
Furthermore, \eqref{e:long7} and \eqref{e:6zfix} yield
\begin{equation}
\begin{cases}
(\forall\iS\in\II)\quad
t_{\iS,\nS}^*=\upgamma_{\iS,\nS}^{-1}
\brk{y_{\iS,\nS}-a_{\iS,\nS}}
-\sum_{\kkk}\LS_{\kS\iS}^*
\brk{v_{\kS,\vartheta_{\iS}(\nS)}^*-b_{\kS,\nS}^*}\;\;\Pas\\
(\forall\kS\in\KK)\quad
t_{\kS,\nS}=\brk{b_{\kS,\nS}-\sum_{\iii}\LS_{\kS\iS}
\overline{\xS}_{\iS}}
-\sum_{\iii}\LS_{\kS\iS}
\brk{a_{\iS,\nS}-\overline{\xS}_{\iS}}\;\;\Pas
\end{cases}
\end{equation}
Hence, by defining 
$\xi_{\nS}
=\norm{\bunder{\boldsymbol{x}}_{\nS}-\bunder{\boldsymbol{\zS}}}
+\sum_{\iii}\norm{\bunder{\boldsymbol{x}}_{\vartheta_{\iS}(\nS)}
-\bunder{\boldsymbol{\zS}}}
+\sum_{\kkk}\norm{\bunder{\boldsymbol{x}}_{\varrho_{\kS}(\nS)}
-\bunder{\boldsymbol{\zS}}}$
and $\tau_{\nS}=\sum_{\iii}\norm{e_{\iS,\nS}}
+\sum_{\kkk}\norm{f_{\kS,\nS}}$,
we deduce from \eqref{e:6ba} and \eqref{e:6var} that
\begin{equation}
\label{e:6bdd}
\begin{cases}
\norm1{\bunder{\boldsymbol{e}}_{\nS}}\leq\dfrac{1}{\upvarepsilon}
\tau_{\nS}\;\;\Pas\\[3mm]
\norm1{\bunder{\boldsymbol{e}}_{\nS}^*}
\leq\dfrac{1+\updelta}{\upvarepsilon}\tau_{\nS}\;\;\Pas\\[3mm]
\norm1{\bunder{\boldsymbol{w}}_{\nS}-\bunder{\boldsymbol{\zS}}}
\leq\dfrac{4(1+\updelta)}{\upvarepsilon^2}
\brk1{\xi_{\nS}+\tau_{\nS}}\;\;\Pas\\[3mm]
\norm1{\bunder{\boldsymbol{t}}_{\nS}^*}
\leq\dfrac{16(1+\updelta)^2}{\upvarepsilon^2}
\brk1{\xi_{\nS}+\tau_{\nS}}\;\;\Pas
\end{cases}
\end{equation}
On the other hand, for every $\jS\in\{0,\ldots,\nS\}$, 
$[\vartheta_{\iS}(\nS)=\jS]$ is
$\upsigma(\{I_{\mS}\}_{\jS\leq\mS\leq\nS})$-measurable, whereas
$\bunder{\boldsymbol{x}}_{\jS}$ is measurable with respect to
$\upsigma\brk{\bunder{\boldsymbol{x}}_0}$ if $\jS=0$ and with
respect to 
$\upsigma\brk{\{\bunder{\boldsymbol{x}}_{0},e_{\iS,\mS},
f_{\kS,\mS},\lambda_{\mS},I_{\mS},K_{\mS}\}_{0\leq\mS<\jS}}$ if
$\jS>0$. Therefore, in view of Assumption~\ref{a:2} and
\eqref{t:7as}, $\mathsf{1}_{[\vartheta_{\iS}(\nS)=\jS]}$ and
$\bunder{\boldsymbol{x}}_{\jS}$ are independent, and
\begin{equation}
(\forall\iS\in\II)\quad
\EE\norm{\bunder{\boldsymbol{x}}_{\vartheta_{\iS}(\nS)}
-\bunder{\boldsymbol{\zS}}}^2
=\sum_{\jS=0}^{\nS}\PP\brk1{[\vartheta_{\iS}(\nS)=\jS]}
\EE\norm{\bunder{\boldsymbol{x}}_{\jS}
-\bunder{\boldsymbol{\zS}}}^2
\leq\max_{0\leq\mS\leq\nS}
\EE\norm{\bunder{\boldsymbol{x}}_{\mS}
-\bunder{\boldsymbol{\zS}}}^2,
\end{equation}
and likewise for $(\upvarrho_{\kS}(\nS))_{\kkk}$.
Thus, we deduce from \eqref{e:6bdd} that there exists a constant
$\upkappa\in\RPP$ satisfying
\begin{align}
\label{e:6epsA}
\EE\varepsilon_{\nS}(\cdot,\bunder{\boldsymbol{\zS}})
&\leq\frac{1}{\upvarepsilon}\,\EE\brk2{
\norm1{\bunder{\boldsymbol{w}}_{\nS}-\bunder{\boldsymbol{\zS}}}\,
\norm1{\bunder{\boldsymbol{e}}_{\nS}^*}
+\norm1{\bunder{\boldsymbol{e}}_{\nS}}\,
\norm1{\bunder{\boldsymbol{t}}_{\nS}^*}
+\norm1{\bunder{\boldsymbol{e}}_{\nS}}\,
\norm1{\bunder{\boldsymbol{e}}_{\nS}^*}}
\nonumber\\
&\leq\upkappa\EE\brk1{(\xi_{\nS}+\tau_{\nS})\tau_{\nS}}
\nonumber\\
&\leq\upkappa\sqrt{\EE(\xi_{\nS}+\tau_{\nS})^2}
\sqrt{\EE\tau_{\nS}^2}\nonumber\\
&\leq\upkappa\brk2{\sqrt{\EE\xi_{\nS}^2}
+\sqrt{\EE\tau_{\nS}^2}}
\sqrt{\EE\tau_{\nS}^2}\nonumber\\
&\leq\upkappa\brk3{\brk1{1+\card\II+\card\KK}
\sqrt{\max_{0\leq\mS\leq\nS}
\EE\norm{\bunder{\boldsymbol{x}}_{\mS}
-\bunder{\boldsymbol{\zS}}}^2}
+\sup_{\jjj}\sqrt{\EE\tau_{\jS}^2}}\sqrt{\EE\tau_{\nS}^2}.
\end{align}
Because \eqref{t:7as} asserts that
$\sum_{\nnn}\sqrt{\EE\tau_{\nS}^2}<\pinf$ and
$\sup_{\jjj}\sqrt{\EE\tau_{\jS}^2}<\pinf$, it remains to show
that $(\max_{0\leq\mS\leq\nS}\EE\norm{\bunder{\boldsymbol{x}}_{\mS}
-\bunder{\boldsymbol{\zS}}}^2)_{\nnn}$ is bounded.
Indeed, Theorem~\ref{t:5}\ref{t:5i+} and
\eqref{e:6epsA} imply that
\begin{align}
&
(\forall\nnn)\quad\EE\norm{\bunder{\boldsymbol{x}}_{\nS+1}
-\bunder{\boldsymbol{\zS}}}^2
\nonumber\\
&\leq\EE\norm{\bunder{\boldsymbol{x}}_{\nS}
-\bunder{\boldsymbol{\zS}}}^2
+2\EE\varepsilon_{\nS}(\cdot,\bunder{\boldsymbol{\zS}})
\EE\lambda_{\nS}
\nonumber\\
&\leq\EE\norm{\bunder{\boldsymbol{x}}_{\nS}
-\bunder{\boldsymbol{\zS}}}^2
+2\uprho\EE\varepsilon_{\nS}(\cdot,\bunder{\boldsymbol{\zS}})
\nonumber\\
&\leq\max_{0\leq\mS\leq\nS}\EE\norm{\bunder{\boldsymbol{x}}_{\mS}
-\bunder{\boldsymbol{\zS}}}^2
+2\uprho\EE\varepsilon_{\nS}(\cdot,\bunder{\boldsymbol{\zS}})
\nonumber\\
&\leq\max_{0\leq\mS\leq\nS}\EE\norm{\bunder{\boldsymbol{x}}_{\mS}
-\bunder{\boldsymbol{\zS}}}^2
+2\uprho\upkappa\brk3{\brk1{1+\card\II+\card\KK}
\sqrt{\max_{0\leq\mS\leq\nS}
\EE\norm{\bunder{\boldsymbol{x}}_{\mS}
-\bunder{\boldsymbol{\zS}}}^2}
+\sup_{\jjj}\sqrt{\EE\tau_{\jS}^2}}\sqrt{\EE\tau_{\nS}^2}
\nonumber\\
&\leq\max_{0\leq\mS\leq\nS}\EE\norm{\bunder{\boldsymbol{x}}_{\mS}
-\bunder{\boldsymbol{\zS}}}^2
+\uprho\upkappa\brk1{1+\card\II+\card\KK}
\brk2{1+\max_{0\leq\mS\leq\nS}
\EE\norm{\bunder{\boldsymbol{x}}_{\mS}
-\bunder{\boldsymbol{\zS}}}^2}\sqrt{\EE\tau_{\nS}^2}
+2\uprho\upkappa
\sup_{\jjj}\sqrt{\EE\tau_{\jS}^2}\sqrt{\EE\tau_{\nS}^2}.
\nonumber
\end{align}
Hence,
\begin{multline}
\label{e:qf}
(\forall\nnn)\quad
1+\max_{0\leq\mS\leq\nS+1}\EE\norm{\bunder{\boldsymbol{x}}_{\mS}
-\bunder{\boldsymbol{\zS}}}^2\\
\leq\brk3{1
+\uprho\upkappa\brk1{1+\card\II+\card\KK}\sqrt{\EE\tau_{\nS}^2}}
\brk2{1+\max_{0\leq\mS\leq\nS}
\EE\norm{\bunder{\boldsymbol{x}}_{\mS}
-\bunder{\boldsymbol{\zS}}}^2}
+2\uprho\upkappa
\sup_{\jjj}\sqrt{\EE\tau_{\jS}^2}\sqrt{\EE\tau_{\nS}^2}.
\end{multline}
Applying \cite[Lemma~5.31]{Livre1} to \eqref{e:qf}, we obtain
$(\max_{0\leq\mS\leq\nS}\EE\norm{\bunder{\boldsymbol{x}}_{\mS}
-\bunder{\boldsymbol{\zS}}}^2)_{\nnn}$ is bounded and,
therefore,
$\sum_{\nnn}\EE\varepsilon_{\nS}(\cdot,\bunder{\boldsymbol{\zS}})<
\pinf$.
Hence, Theorem~\ref{t:5}\ref{t:5ii} asserts that
\begin{equation}
\label{e:4932+}
\sum_{\nnn}\EE\|\bunder{\boldsymbol{x}}_{\nS+1}
-\bunder{\boldsymbol{x}}_{\nS}\|^2<\pinf.
\end{equation}

\ref{t:7i}--\ref{t:7ii}: These follow from
Theorem~\ref{t:5}\ref{t:5i-}, \eqref{e:4932+}, and \eqref{e:6var}.

\ref{t:7iii-}--\ref{t:7iii}: By \eqref{e:4932+} and
Proposition~\ref{p:4} applied in $\HHS\oplus\GGS$ to
$(\bunder{\boldsymbol{x}}_{\nS})_{\nnn}$, together with
Lemma~\ref{l:1},
\begin{equation}
\label{e:pconver}
(\forall\iS\in\II)(\forall\kS\in\KK)\quad
\bunder{\boldsymbol{x}}_{\vartheta_{\iS}(\nS)}
-\bunder{\boldsymbol{x}}_{\nS}\Pto\bunder{\boldsymbol{0}}
\quad\text{and}\quad
\bunder{\boldsymbol{x}}_{\varrho_{\kS}(\nS)}
-\bunder{\boldsymbol{x}}_{\nS}\Pto\bunder{\boldsymbol{0}}.
\end{equation}
Note that the almost sure boundedness of
$(\bunder{\boldsymbol{x}}_{\nS})_{\nnn}$ provides the almost sure
boundedness of $(\norm{\bunder{\boldsymbol{t}}_{\nS}^*})_{\nnn}$,
$(\norm{y_{\iS,\nS}-a_{\iS,\nS}})_{\nnn}$, and
$(\norm{y_{\kS,\nS}^*-b_{\kS,\nS}^*})_{\nnn}$. 
Hence, we deduce from \eqref{e:pconver} and \eqref{e:long7} that
$\phi_{\nS}\Pto 0$. Further, 
\begin{align}
\psi_{\nS}
&=\Sum_{\iii}\brk2{\scal1{y_{\iS,\nS}-a_{\iS,\nS}}{t_{\iS,\nS}^*}
-\upgamma_{\iS,\nS}^{-1}\norm1{y_{\iS,\nS}-a_{\iS,\nS}}^2}
+\Sum_{\kkk}\brk2{\scal1{y_{\kS,\nS}^*-b_{\kS,\nS}^*}{t_{\kS,\nS}} 
-\upmu_{\kS,\nS}\norm1{y_{\kS,\nS}^*-b_{\kS,\nS}^*}^2}
\nonumber\\
&=-\Sum_{\iii}\scal2{y_{\iS,\nS}-a_{\iS,\nS}}
{\Sum_{\kkk}\LS_{\kS\iS}^*\brk1{v_{\kS,\vartheta_{\iS}(\nS)}^*
-b_{\kS,\nS}^*}}
+\Sum_{\kkk}\scal2{y_{\kS,\nS}^*-b_{\kS,\nS}^*}
{\Sum_{\iii}\LS_{\kS\iS}\brk1{x_{\iS,\varrho_{\kS}(\nS)}
-a_{\iS,\nS}}}
\nonumber\\
&=-\Sum_{\iii}\scal2{y_{\iS,\nS}-a_{\iS,\nS}}
{\Sum_{\kkk}\LS_{\kS\iS}^*\brk1{v_{\kS,\vartheta_{\iS}(\nS)}^*
-v_{\kS,\varrho_{\kS}(\nS)}^*}}
+\Sum_{\kkk}\scal2{y_{\kS,\nS}^*-b_{\kS,\nS}^*}
{\Sum_{\iii}\LS_{\kS\iS}\brk1{x_{\iS,\varrho_{\kS}(\nS)}
-x_{\iS,\vartheta_{\iS}(\nS)}}},
\end{align}
from which we deduce that $\psi_{\nS}\Pto 0$ and, therefore,
$\nu_{\nS}\Pto 0$.
On the other hand, let $\nnn$ and note that
$\nu_{\nS}\leq\abs{\phi_{\nS}}+\abs{\psi_{\nS}}$,
$\psi_{\nS}\geq-\abs{\psi_{\nS}}$, and
$\phi_{\nS}\geq-\abs{\phi_{\nS}}$. Then \eqref{e:Del1} gives
\begin{align}
\Delta_{\nS}
&\geq\Sum_{\iii}\upgamma_{\iS,\nS}^{-1}
\norm1{y_{\iS,\nS}-a_{\iS,\nS}}^2
+\Sum_{\kkk}\upmu_{\kS,\nS}
\norm1{y_{\kS,\nS}^*-b_{\kS,\nS}^*}^2
-2\abs{\psi_{\nS}}-2\abs{\phi_{\nS}}
\nonumber\\
&\geq\upvarepsilon\brk3{\Sum_{\iii}
\norm1{y_{\iS,\nS}-a_{\iS,\nS}}^2
+\Sum_{\kkk}
\norm1{y_{\kS,\nS}^*-b_{\kS,\nS}^*}^2}
-2\abs{\psi_{\nS}}-2\abs{\phi_{\nS}}.
\label{e:Del2}
\end{align}
Moreover, since $\lambda_{\nS}\geq\upvarepsilon\;\Pas$, 
Theorem~\ref{t:5}\ref{t:5iii} gives
$\varlimsup\Delta_{\nS}\leq0\;\Pas$ Hence, applying
Lemma~\ref{l:3} to the first inequality in \eqref{e:Del2}, with
$\chi_{\nS}=-2|\psi_{\nS}|-2|\phi_{\nS}|\Pto0$, yields
\begin{equation}
\label{e:61111}
(\forall\iS\in\II)(\forall\kS\in\KK)\quad
y_{\iS,\nS}-a_{\iS,\nS}\Pto0\quad\text{and}\quad 
y_{\kS,\nS}^*-b_{\kS,\nS}^*\Pto0.
\end{equation}
Let $(\iS,\kS)\in\II\times\KK$. Then \eqref{e:fr},
\eqref{e:pconver}, and \eqref{e:61111} imply that
$x_{\iS,\nS}-a_{\iS,\nS}\Pto0$ and
$v_{\kS,\nS}^*-b_{\kS,\nS}^*\Pto0$. Further, it follows from
\eqref{e:long7} that  
\begin{equation}
t_{\iS,\nS}^*=\upgamma_{\iS,\nS}^{-1}
\brk1{y_{\iS,\nS}-a_{\iS,\nS}}
-\Sum_{\kkk}\LS_{\kS\iS}^*\brk1{
v_{\kS,\vartheta_{\iS}(\nS)}^*-b_{\kS,\nS}^*}
\quad\text{and}\quad
t_{\kS,\nS}=\Sum_{\iii}\LS_{\kS\iS}
\brk1{x_{\iS,\varrho_{\kS}(\nS)}-a_{\iS,\nS}}
+\upmu_{\kS,\nS}\brk1{y_{\kS,\nS}^*-b_{\kS,\nS}^*}.
\end{equation}
We then derive from \eqref{e:61111} and \eqref{e:pconver},
together with the fact that $x_{\iS,\nS}-a_{\iS,\nS}\Pto 0$ and
$v_{\kS,\nS}^*-b_{\kS,\nS}^*\Pto 0$, that
$t_{\iS,\nS}^*\Pto0$ and $t_{\kS,\nS}\Pto0$. 
Finally, it is clear from \eqref{t:7as} that
$\bunder{\boldsymbol{e}}_{\nS}\to\bunder{\boldsymbol{0}}\;\Pas$ and
$\bunder{\boldsymbol{e}}_{\nS}^*\to\bunder{\boldsymbol{0}}\;\Pas$
Altogether, since $\norm{\bunder{\boldsymbol{w}}_{\nS}
-\bunder{\boldsymbol{q}}_{\nS}}=2\sqrt{\upalpha\nu_{\nS}}\Pto 0$,
we obtain
\begin{equation}
\label{e:76}
\bunder{\boldsymbol{x}}_{\nS}-\bunder{\boldsymbol{w}}_{\nS}
-\bunder{\boldsymbol{e}}_{\nS}\Pto\bunder{\boldsymbol{0}},\quad
\bunder{\boldsymbol{w}}_{\nS}+\bunder{\boldsymbol{e}}_{\nS}
-\bunder{\boldsymbol{q}}_{\nS}\Pto\bunder{\boldsymbol{0}},
\quad\text{and}\quad
\bunder{\boldsymbol{w}}_{\nS}^*+\bunder{\boldsymbol{e}}_{\nS}^*
+\boldsymbol{\mathsf{C}}\bunder{\boldsymbol{q}}_{\nS}
\Pto\bunder{\boldsymbol{0}}.
\end{equation}
We conclude by invoking Theorem~\ref{t:5}\ref{t:5iv}.
\end{proof}

\begin{remark}
\label{r:76}
We complement Theorem~\ref{t:7} with the following observations.
\begin{enumerate}
\item
The deterministic version of Theorem~\ref{t:7} is new. It
corresponds to deterministic relaxations 
$(\uplambda_{\nS})_{\nnn}$ in $\left]0,2\right[$. Further, the
stochastic block selection strategy \eqref{e:2393+} becomes 
\begin{equation}
\label{e:5341}
(\exi\mathsf{N}\in\NN\smallsetminus\{0\})
(\forall\nnn)\quad
\II=\bigcup_{\jS=\nS}^{\nS+\mathsf{N}-1}\II_{\jS}
\quad\text{and}\quad
\KK=\bigcup_{\kS=\nS}^{\nS+\mathsf{N}-1}\KK_{\kS}.
\end{equation}
Finally, condition \eqref{t:7as} on the error sequences reduces to
\begin{equation}
\label{e:je}
\brk1{\forall(\iS,\kS)\in\II\times\KK}\quad
\sum_{\nnn}\norm{\mathsf{e}_{\iS,\nS}}<\pinf\quad
\text{and}\quad\sum_{\nnn}\norm{\mathsf{f}_{\kS,\nS}}<\pinf.
\end{equation}
This summable error pattern is needed in particular in an
application of Problem~\ref{prob:4} to multistage stochastic
programming which exploits the deterministic version
of our Theorem~\ref{t:7} \cite{Ecks26}. Note that the result of
\cite[Theorem~13]{MaPr18} does not allow for the error pattern
\eqref{e:je}.
\item
\label{r:76ii}
A stochastic block-iterative algorithm for solving 
Problem~\ref{prob:4} was proposed in \cite[Corollary~5.3]{Siop15},
with almost sure convergence of its iterates. This algorithm
involves deterministic relaxations in $\left]0,2\right[$ and
necessitates inversions to handle the linear operators. In the case
when $\II$ is a singleton, further algorithms with the same
features were proposed in \cite{Siim25}. The algorithm of 
\cite[Proposition~4.6]{Pesq15} also guarantees almost sure
convergence of the iterates but it requires knowledge of the norms
of linear operators, which may be impossible in some applications
\cite{Nume16}. The same comments apply to the algorithm of
\cite[Theorem~2.1 and Algorithm~3.1]{Cham24}, which considers the
minimization case with $\II$ as a singleton. Additionally, none of
these prior works show convergence in $L^2$, nor can they benefit
from adaptive strategies as in Assumption~\ref{a:2} since their
block-selection distributions remain constant throughout the
iterations.
\item
\label{r:76iii}
The convergence results invoke Theorem~\ref{t:5}\ref{t:5iv}, which
requires Euclidean spaces. Note that we cannot use
Theorem~\ref{t:5}\ref{t:5iv-}, which would provide weak
convergence in general Hilbert spaces, because the convergences in
\eqref{e:76} are only in probability and not almost sure.
\end{enumerate} 
\end{remark}

\section{Randomized block-iterative saddle projective splitting}
\label{sec:6}

\subsection{Problem setting}

We consider a highly structured composite multivariate primal-dual
inclusion problem introduced in \cite{Sadd22} and further studied
in \cite[Section~10]{Acnu24}. This model includes a mix of
set-valued, cocoercive, and Lipschitzian monotone operators, as
well as linear operators and various monotonicity-preserving
operations among them. Its multivariate structure captures problems
in areas such as domain decomposition methods \cite{Aldu23,Nume16},
game theory \cite{Joca22,Pave25},
mean field games \cite{Bri18b},
machine learning \cite{Bach12,Bric19},
network flow problems \cite{Buim22,Rock95},
neural networks \cite{Svva20},
and stochastic programming \cite{Buim25,Ecks25}.

\begin{problem}
\label{prob:1}
Let $(\HS_{\iS})_{\iii}$ and $(\GS_{\kS})_{\kkk}$
be finite families of Euclidean spaces with respective
direct sums $\HHS=\bigoplus_{\iii}\HS_{\iS}$ and
$\GGS=\bigoplus_{\kkk}\GS_{\kS}$.
Denote by $\boldsymbol{\xS}=(\xS_{\iS})_{\iii}$
a generic element in $\HHS$.
For every $\iS\in\II$ and every $\kS\in\KK$,
let $\sS_{\iS}^*\in\HS_{\iS}$, let $\rS_{\kS}\in\GS_{\kS}$,
and suppose that the following are satisfied:
\begin{enumerate}[label={[\alph*]}]
\item
\label{prob:1a}
$\AS_{\iS}\colon\HS_{\iS}\to 2^{\HS_{\iS}}$ is maximally monotone,
$\CS_{\iS}\colon\HS_{\iS}\to\HS_{\iS}$ is cocoercive with constant
$\upalpha_{\iS}^{\CC}\in\RPP$,
$\QS_{\iS}\colon\HS_{\iS}\to\HS_{\iS}$ is monotone and Lipschitzian
with constant $\upalpha_{\iS}^{\LL}\in\RP$, and 
$\RS_{\iS}\colon\HHS\to\HS_{\iS}$.
\item
\label{prob:1b}
$\BS_{\kS}^{\MM}\colon\GS_{\kS}\to 2^{\GS_{\kS}}$ is maximally 
monotone, $\BS_{\kS}^{\CC}\colon\GS_{\kS}\to\GS_{\kS}$ is 
cocoercive with constant $\upbeta_{\kS}^{\CC}\in\RPP$, and 
$\BS_{\kS}^{\LL}\colon\GS_{\kS}\to\GS_{\kS}$
is monotone and Lipschitzian with constant 
$\upbeta_{\kS}^{\LL}\in\RP$.
\item
\label{prob:1c}
$\DS_{\kS}^{\MM}\colon\GS_{\kS}\to 2^{\GS_{\kS}}$ is maximally 
monotone, $\DS_{\kS}^{\CC}\colon\GS_{\kS}\to\GS_{\kS}$ is 
cocoercive with constant $\updelta_{\kS}^{\CC}\in\RPP$, and 
$\DS_{\kS}^{\LL}\colon\GS_{\kS}\to\GS_{\kS}$
is monotone and Lipschitzian with constant
$\updelta_{\kS}^{\LL}\in\RP$.
\item
\label{prob:1d}
$\LS_{\kS\iS}\colon\HS_{\iS}\to\GS_{\kS}$ is linear.
\end{enumerate}
In addition, it is assumed that
\begin{enumerate}[resume,label={[\alph*]}]
\item
\label{prob:1e}
$\boldsymbol{\RS}\colon\HHS\to\HHS\colon
\boldsymbol{\xS}\mapsto(\RS_{\iS}\boldsymbol{\xS})_{\iii}$
is monotone and Lipschitzian with constant $\upchi\in\RP$.
\end{enumerate}
The objective is to solve the primal problem
\begin{multline}
\label{e:1p}
\text{find}\;\:\overline{\boldsymbol{\xS}}\in\HHS
\;\:\text{such that}\;\:(\forall\iS\in\II)\;\;
\sS_{\iS}^*\in\AS_{\iS}\overline{\xS}_{\iS}+\CS_{\iS}
\overline{\xS}_{\iS}+\QS_{\iS}\overline{\xS}_{\iS}
+\RS_{\iS}\overline{\boldsymbol{\xS}}\\
+\Sum_{\kkk}\LS_{\kS\iS}^*\Bigg(\Big(
\big(\BS_{\kS}^{\MM}+\BS_{\kS}^{\CC}+\BS_{\kS}^{\LL}\big)\infconv
\big(\DS_{\kS}^{\MM}+\DS_{\kS}^{\CC}+\DS_{\kS}^{\LL}\big)\Big)
\Bigg(\Sum_{\jS\in\mbox{\scriptsize\ttfamily{I}}}
\LS_{\kS\jS}\overline{\xS}_{\jS}-\rS_{\kS}\Bigg)\Bigg)
\end{multline}
and the associated dual problem
\begin{multline}
\label{e:1d}
\text{find}\;\:\overline{\boldsymbol{\vS}}^*\in\GGS
\;\:\text{such that}\;\:
(\exi\boldsymbol{\xS}\in\HHS)(\forall\iS\in\II)(\forall\kS\in\KK)\\
\begin{cases}
\sS^*_{\iS}-\Sum_{\jS\in\mbox{\scriptsize\ttfamily{K}}}
\LS_{\jS\iS}^*\overline{\vS}_{\jS}^*\in
\AS_{\iS}\xS_{\iS}+\CS_{\iS}\xS_{\iS}+\QS_{\iS}\xS_{\iS}
+\RS_{\iS}\boldsymbol{\xS};\\
\overline{\vS}_{\kS}^*\in
\Big(\big(\BS_{\kS}^{\MM}+\BS_{\kS}^{\CC}+\BS_{\kS}^{\LL}\big)
\infconv
\big(\DS_{\kS}^{\MM}+\DS_{\kS}^{\CC}+\DS_{\kS}^{\LL}\big)\Big)
\Bigg(\Sum_{\jS\in\mbox{\scriptsize\ttfamily{I}}}\LS_{\kS\jS}
\xS_{\jS}-\rS_{\kS}\Bigg).
\end{cases}
\end{multline}
Finally, $\mathscr{P}$ denotes the set of solutions to
\eqref{e:1p}, $\mathscr{D}$ denotes the set of solutions to
\eqref{e:1d}, and we set $\XXX=\HHS\oplus\GGS\oplus\GGS\oplus\GGS$.
\end{problem}

To deal with large size problems in which $\II$ and/or $\KK$ is
sizable, the deterministic block-iterative algorithm proposed in
\cite{Sadd22} has the ability to activate only subgroups of
coordinates and operators at each iteration instead of all of them
as in classical methods. We propose a stochastic version of this
block-iterative algorithm with almost sure convergence to a
solution of Problem~\ref{prob:1}. The convergence analysis will
rely on an application of Theorem~\ref{t:5} in $\XXX$ using the
following saddle formalism. 

\begin{definition}[{\cite[Definition~1]{Sadd22}}]
\label{d:S}
The \emph{saddle operator} associated with Problem~\ref{prob:1} is
\begin{align}
\label{e:saddle}
\sad\colon\XXX\to2^{\XXX}\colon
&(\boldsymbol{\xS},\boldsymbol{\yS},\boldsymbol{\zS},
\boldsymbol{\vS}^*)\mapsto\nonumber\\
&\hskip -12mm\Bigg(\bigtimes_{\iii}
\brk3{{-}\sS_{\iS}^*+\AS_{\iS}\xS_{\iS}+\CS_{\iS}\xS_{\iS}
+\QS_{\iS}\xS_{\iS}+\RS_{\iS}\boldsymbol{\xS}
+\sum_{\kkk}\LS^*_{\kS\iS}\vS^*_{\kS}},
\bigtimes_{\kkk}\brk1{\BS_{\kS}^{\MM}\yS_{\kS}
+\BS_{\kS}^{\CC}\yS_{\kS}+\BS_{\kS}^{\LL}\yS_{\kS}
-\vS_{\kS}^*},\nonumber\\
&\hskip -12mm\;\;\bigtimes_{\kkk}\brk1{\DS_{\kS}^{\MM}\zS_{\kS}
+\DS_{\kS}^{\CC}\zS_{\kS}+\DS_{\kS}^{\LL}\zS_{\kS}
-\vS_{\kS}^*},\bigtimes_{\kkk}\bigg\{
\rS_{\kS}+\yS_{\kS}+\zS_{\kS}-\sum_{\iii}\LS_{\kS\iS}\xS_{\iS}
\bigg\}~\Bigg),
\end{align}
and the \emph{saddle form} of Problem~\ref{prob:1} is to 
\begin{equation}
\label{e:sf}
\text{find}\;\:\overline{\bunder{\boldsymbol{\mathsf{x}}}}
\in\XXX\;\:\text{such that}\;\:\bunder{\boldsymbol{0}}
\in\sad\overline{\bunder{\boldsymbol{\mathsf{x}}}}.
\end{equation}
\end{definition}

Item \ref{p:6ii} below asserts that finding a saddle point, i.e.,
solving \eqref{e:sf}, provides a solution to Problem~\ref{prob:1}.

\begin{proposition}[{\cite[Proposition~1]{Sadd22}}]
\label{p:6}
Consider the setting of Problem~\ref{prob:1} 
and Definition~\ref{d:S}. 
Then the following hold:
\begin{enumerate}
\item
\label{p:6i}
$\sad$ is maximally monotone.
\item
\label{p:6ii}
Suppose that $\overline{\bunder{\boldsymbol{\xS}}}
=(\overline{\boldsymbol{\xS}},\overline{\boldsymbol{\yS}},
\overline{\boldsymbol{\zS}},\overline{\boldsymbol{\vS}}^*)
\in\zer\sad$. Then $(\overline{\boldsymbol{\xS}},
\overline{\boldsymbol{\vS}}^*)\in\mathscr{P}\times\mathscr{D}$.
\item
\label{p:6iv}
$\mathscr{D}\neq\emp$ $\Leftrightarrow$
$\zer\sad\neq\emp$ $\Rightarrow$ $\mathscr{P}\neq\emp$.
\end{enumerate}
\end{proposition}

To use Theorem~\ref{t:5}, we decompose the saddle operator 
$\sad$ of \eqref{e:saddle} as the sum of 
\begin{align}
\label{e:1179}
\bunder{\boldsymbol{\mathsf{W}}}\colon\XXX\to2^{\XXX}\colon
(\boldsymbol{\xS},\boldsymbol{\yS},\boldsymbol{\zS},
\boldsymbol{\vS}^*)
\mapsto
&\Bigg(
\bigtimes_{\iii}\bigg({-}\sS_{\iS}^*+\AS_{\iS}\xS_{\iS}
+\QS_{\iS}\xS_{\iS}+\RS_{\iS}\boldsymbol{\xS}
+\sum_{\kkk}\LS^*_{\kS\iS}\vS^*_{\kS}\bigg),
\bigtimes_{\kkk}\big(\BS_{\kS}^{\MM}\yS_{\kS}
+\BS_{\kS}^{\LL}\yS_{\kS}-\vS_{\kS}^*\big),\nonumber\\
&\;\;\bigtimes_{\kkk}\big(\DS_{\kS}^{\MM}\zS_{\kS}
+\DS_{\kS}^{\LL}\zS_{\kS}-\vS_{\kS}^*\big),
\bigtimes_{\kkk}\bigg\{
\rS_{\kS}+\yS_{\kS}+\zS_{\kS}-\sum_{\iii}\LS_{\kS\iS}\xS_{\iS}
\bigg\}~\Bigg)
\end{align}
and
\begin{equation}
\label{e:2796}
\bunder{\boldsymbol{\mathsf{C}}}\colon\XXX\to\XXX\colon
(\boldsymbol{\xS},\boldsymbol{\yS},\boldsymbol{\zS},
\boldsymbol{\vS}^*)
\mapsto\brk2{
\big(\CS_{\iS}\xS_{\iS}\big)_{\iii},
\big(\BS_{\kS}^{\CC}\yS_{\kS}\big)_{\kkk},
\big(\DS_{\kS}^{\CC}\zS_{\kS}\big)_{\kkk},\boldsymbol{0}}.
\end{equation}
As seen in \cite[Proposition~2(ii)--(iii)]{Sadd22},
$\bunder{\boldsymbol{\mathsf{W}}}$ is maximally monotone and 
$\bunder{\boldsymbol{\mathsf{C}}}$ is $\upalpha$-cocoercive
with $\upalpha=\min\{\upalpha_{\iS}^{\CC},\upbeta_{\kS}^{\CC},
\updelta_{\kS}^{\CC}\}_{\iii,\kkk}$. This confirms that
\eqref{e:sf} fits the framework described in
Problem~\ref{prob:19}.

\subsection{Algorithm and convergence}
We now introduce our stochastic block-iterative algorithm. It
differs from that of \cite{Sadd22} in that the selection of the
blocks of variables and operators to be activated at each iteration
is random, and so is the relaxation strategy. In addition, the
relaxation parameters need not be bounded by 2. The algorithm will
operate under the following index sweeping strategy. 

\begin{assumption}
\label{a:1}
$\II$ and $\KK$ are nonempty finite sets,
$(\uppi_{\iS})_{\iii}$ and
$(\upzeta_{\kS})_{\kkk}$ are in $\left]0,1\right]$, and 
$\mathsf{N}\in\NN\smallsetminus\{0\}$. $(I_{\nS})_{\nnn}$ are 
nonempty sets composed of elements randomly taken in $\II$ and 
$(K_{\nS})_{\nnn}$ are nonempty sets composed of elements randomly 
taken in $\KK$. 
Moreover, $I_0=\II$, $K_0=\KK$, $(\upsigma(I_{\nS}))_{\nnn}$ and
$(\upsigma(\KK_{\nS}))_{\nnn}$ are independent families satisfying
\begin{equation}
\label{e:2393}
\quad(\forall\nnn)\;\;
\brk3{\displaystyle(\forall\iS\in\II)\;\;
\PP\brk3{\brk[s]3{\iS\in
\bigcup_{\jS=\nS}^{\nS+\mathsf{N}-1}I_{\jS}}}
\geq\uppi_{\iS}\quad\text{and}\quad
\displaystyle(\forall\kS\in\KK)\;\;
\PP\brk3{\brk[s]3{\kS\in
\bigcup_{\jS=\nS}^{\nS+\mathsf{N}-1}K_{\jS}}}
\geq\upzeta_{\kS}}.
\end{equation}
\end{assumption}

\begin{algorithm}
\label{algo:1}
Consider the setting of Problem~\ref{prob:1} and suppose that
Assumption~\ref{a:1} is in force. 
Set $\upalpha=\min\{\upalpha_{\iS}^{\CC},\upbeta_{\kS}^{\CC},
\updelta_{\kS}^{\CC}\}_{\iii,\kkk}$, and 
let $\upphi\in\RPP$ and $\upvarepsilon\in\zeroun$ be such that
$\upphi>1/(4\upalpha)$ and 
$1/\upvarepsilon>\max\{\upalpha_{\iS}^{\LL}+\upchi+\upphi,
\upbeta_{\kS}^{\LL}+\upphi,\updelta_{\kS}^{\LL}+\upphi
\}_{\iii,\kkk}$.
Let $\uprho\in[2,\pinf[$ and, for every $(\iS,\kS)\in\II\times\KK$,
let $x_{\iS,0}\in L^2(\upOmega,\FE,\PP;\HS_{\iS})$ and 
$\{y_{\kS,0},z_{\kS,0},v_{\kS,0}^*\}\subset 
L^2(\upOmega,\FE,\PP;\GS_{\kS})$. Iterate
\begin{equation}
\label{e:long1}
\begin{array}{l}
\text{for}\;\nS=0,1,\ldots\\
\left\lfloor
\begin{array}{l}
\text{for every}\;\iS\in I_{\nS}\\
\left\lfloor
\begin{array}{l}
l_{\iS,\nS}^*=\QS_{\iS}x_{\iS,\nS}
+\RS_{\iS}\boldsymbol{x}_{\nS}
+\sum_{\kkk}\LS_{\kS\iS}^*v_{\kS,\nS}^*;\;
\upgamma_{\iS,\nS}
\in[\upvarepsilon,1/(\upalpha_{\iS}^{\LL}+\upchi+\upphi)];\\
a_{\iS,\nS}=\mathsf{J}_{\upgamma_{\iS,\nS}\AS_{\iS}}\big(
x_{\iS,\nS}+\upgamma_{\iS,\nS}(\sS_{\iS}^*-l_{\iS,\nS}^*
-\CS_{\iS}x_{\iS,\nS})\big);\;
a_{\iS,\nS}^*=\upgamma_{\iS,\nS}^{-1}(x_{\iS,\nS}
-a_{\iS,\nS})-l_{\iS,\nS}^*+\QS_{\iS}a_{\iS,\nS};\\
\xi_{\iS,\nS}=\|a_{\iS,\nS}-x_{\iS,\nS}\|^2;
\end{array}
\right.\\
\text{for every}\;\iS\in\II\smallsetminus I_{\nS}\\
\left\lfloor
\begin{array}{l}
a_{\iS,\nS}=a_{\iS,\nS-1};\;a_{\iS,\nS}^*=a_{\iS,\nS-1}^*;\;
\xi_{\iS,\nS}=\xi_{\iS,\nS-1};\\
\end{array}
\right.\\
\text{for every}\;\kS\in K_{\nS}\\
\left\lfloor
\begin{array}{l}
\upmu_{\kS,\nS}\in[\upvarepsilon,1/(\upbeta_{\kS}^{\LL}
+\upphi)];\;
\upnu_{\kS,\nS}\in[\upvarepsilon,1/(\updelta_{\kS}^{\LL}
+\upphi)];\;
\uptau_{\kS,\nS}\in[\upvarepsilon,1/\upvarepsilon];\\
u_{\kS,\nS}^*=v_{\kS,\nS}^*-\BS_{\kS}^{\LL}y_{\kS,\nS};\;
w_{\kS,\nS}^*=v_{\kS,\nS}^*-\DS_{\kS}^{\LL}z_{\kS,\nS};\\
b_{\kS,\nS}=\mathsf{J}_{\upmu_{\kS,\nS}\BS_{\kS}^{\MM}}
\big(y_{\kS,\nS}
+\upmu_{\kS,\nS}(u_{\kS,\nS}^*-\BS_{\kS}^{\CC}y_{\kS,\nS})\big);\;
d_{\kS,\nS}=\mathsf{J}_{\upnu_{\kS,\nS}\DS_{\kS}^{\MM}}
\big(z_{\kS,\nS}
+\upnu_{\kS,\nS}(w_{\kS,\nS}^*-\DS_{\kS}^{\CC}z_{\kS,\nS})\big);\\
o_{\kS,\nS}=\rS_{\kS}+b_{\kS,\nS}+d_{\kS,\nS}
-\sum_{\iii}\LS_{\kS\iS}a_{\iS,\nS};\;
o_{\kS,\nS}^*=\uptau_{\kS,\nS}\big(
\sum_{\iii}\LS_{\kS\iS}x_{\iS,\nS}
-y_{\kS,\nS}-z_{\kS,\nS}-\rS_{\kS}\big)+v_{\kS,\nS}^*;\\
q_{\kS,\nS}^*=\upmu_{\kS,\nS}^{-1}(y_{\kS,\nS}-b_{\kS,\nS})
+u_{\kS,\nS}^*+\BS_{\kS}^{\LL}b_{\kS,\nS}-o_{\kS,\nS}^*;\;
t_{\kS,\nS}^*=\upnu_{\kS,\nS}^{-1}(z_{\kS,\nS}-d_{\kS,\nS})
+w_{\kS,\nS}^*+\DS_{\kS}^{\LL}d_{\kS,\nS}-o_{\kS,\nS}^*;\\
\eta_{\kS,\nS}=\|b_{\kS,\nS}-y_{\kS,\nS}\|^2
+\|d_{\kS,\nS}-z_{\kS,\nS}\|^2;\\
\end{array}
\right.\\
\text{for every}\;\kS\in\KK\smallsetminus K_{\nS}\\
\left\lfloor
\begin{array}{l}
b_{\kS,\nS}=b_{\kS,\nS-1};\;
d_{\kS,\nS}=d_{\kS,\nS-1};\;
o_{\kS,\nS}=\rS_{\kS}+b_{\kS,\nS}+d_{\kS,\nS}
-\sum_{\iii}\LS_{\kS\iS}a_{\iS,\nS};\;
o_{\kS,\nS}^*=o_{\kS,\nS-1}^*;\\
q_{\kS,\nS}^*=q_{\kS,\nS-1}^*;\;
t_{\kS,\nS}^*=t_{\kS,\nS-1}^*;\;
\eta_{\kS,\nS}=\eta_{\kS,\nS-1};
\end{array}
\right.\\
\text{for every}\;\iS\in\II\\
\left\lfloor
\begin{array}{l}
p_{\iS,\nS}^*=a_{\iS,\nS}^*
+\RS_{\iS}\boldsymbol{a}_{\nS}
+\sum_{\kkk}\LS_{\kS\iS}^*o_{\kS,\nS}^*;
\end{array}
\right.\\[1mm]
\begin{aligned}
\Delta_{\nS}&=\textstyle
{-}(4\upalpha)^{-1}\big(\sum_{\iii}\xi_{\iS,\nS}
+\sum_{\kkk}\eta_{\kS,\nS}\big)
+\sum_{\iii}\scal{x_{\iS,\nS}-a_{\iS,\nS}}
{p_{\iS,\nS}^*}\\
&\textstyle
\quad\;+\sum_{\kkk}
\big(\scal{y_{\kS,\nS}-b_{\kS,\nS}}{q_{\kS,\nS}^*}
+\scal{z_{\kS,\nS}-d_{\kS,\nS}}{t_{\kS,\nS}^*}
+\scal{o_{\kS,\nS}}{v_{\kS,\nS}^*-o_{\kS,\nS}^*}\big);
\end{aligned}\\[4mm]
\theta_{\nS}=
\dfrac{\mathsf{1}_{[\Delta_{\nS}>0]}\Delta_{\nS}}
{\sum_{\iii}\|p_{\iS,\nS}^*\|^2
+\sum_{\kkk}\big(
\|q_{\kS,\nS}^*\|^2+\|t_{\kS,\nS}^*\|^2+\|o_{\kS,\nS}\|^2\big)
+\mathsf{1}_{[\Delta_{\nS}\leq 0]}};\\
\lambda_{\nS}\in 
L^\infty(\upOmega,\FE,\PP;[\upvarepsilon,\uprho])\\
\text{for every}\;\iS\in\II\\
\left\lfloor
\begin{array}{l}
x_{\iS,\nS+1}=x_{\iS,\nS}-\lambda_{\nS}\theta_{\nS}p_{\iS,\nS}^*;
\end{array}
\right.\\
\text{for every}\;\kS\in\KK\\
\left\lfloor
\begin{array}{l}
y_{\kS,\nS+1}=y_{\kS,\nS}-\lambda_{\nS}\theta_{\nS}q_{\kS,\nS}^*;\;
z_{\kS,\nS+1}=z_{\kS,\nS}-\lambda_{\nS}\theta_{\nS}t_{\kS,\nS}^*;\;
v_{\kS,\nS+1}^*=v_{\kS,\nS}^*-\lambda_{\nS}\theta_{\nS}o_{\kS,\nS}.
\end{array}
\right.\\[1mm]
\end{array}
\right.
\end{array}
\end{equation}
\end{algorithm}

The convergence properties of Algorithm~\ref{algo:1} are
established in the following theorem.

\begin{theorem}
\label{t:1}
Consider the setting of Algorithm~\ref{algo:1}. Suppose that
$\mathscr{D}\neq\emp$,
$\inf_{\nnn}\EE(\lambda_{\nS}(2-\lambda_{\nS}))>0$, and, for every
$\nnn$, $\lambda_{\nS}$ is independent of
\begin{equation}
\upsigma\brk2{\brk[c]1{x_{\iS,\pS}}_{\iii,0\leq\pS\leq\nS}
\cup\brk[c]1{y_{\kS,\pS},
z_{\kS,\pS},v_{\kS,\pS}^*}_{\kkk,0\leq\pS\leq\nS}\cup
\brk[c]1{I_{\nS},K_{\nS}}}.
\end{equation}
Then the following hold: 
\begin{enumerate}
\item
\label{t:1i}
Let $\iS\in\II$. Then $(x_{\iS,\nS})_{\nnn}$ lies in
$L^2(\upOmega,\FE,\PP;\HS_{\iS})$ and
$\sum_{\nnn}\EE\|x_{\iS,\nS+1}-x_{\iS,\nS}\|^2<\pinf$.
\item
\label{t:1ii}
Let $\kS\in\KK$. Then $(y_{\kS,\nS})_{\nnn}$, 
$(z_{\kS,\nS})_{\nnn}$, and $(v_{\kS,\nS}^*)_{\nnn}$ are sequences
in $L^2(\upOmega,\FE,\PP;\GS_{\kS})$. Further,
$\sum_{\nnn}\EE\|y_{\kS,\nS+1}-y_{\kS,\nS}\|^2<\pinf$,
$\sum_{\nnn}\EE\|z_{\kS,\nS+1}-z_{\kS,\nS}\|^2<\pinf$, and
$\sum_{\nnn}\EE\|v_{\kS,\nS+1}^*-v_{\kS,\nS}^*\|^2<\pinf$.
\item
\label{t:1iii-}
Let $\iS\in\II$ and $\kS\in\KK$. Then
$x_{\iS,\nS}-a_{\iS,\nS}\Pto 0$,
$y_{\kS,\nS}-b_{\kS,\nS}\Pto 0$,
$z_{\kS,\nS}-d_{\kS,\nS}\Pto 0$, and
$v_{\kS,\nS}^*-o_{\kS,\nS}^*\Pto 0$.
\item
\label{t:1iii}
There exist a $\mathscr{P}$-valued random variable 
$\overline{\boldsymbol{x}}$ and a $\mathscr{D}$-valued random
variable $\overline{\boldsymbol{v}}^*$ such that, for every 
$\iS\in\II$ and every $\kS\in\KK$, 
$x_{\iS,\nS}\to\overline{x}_{\iS}\;\Pas$,
$a_{\iS,\nS}\to\overline{x}_{\iS}\;\Pas$,
and $v_{\kS,\nS}^*\to\overline{v}_{\kS}^*\;\Pas$
\end{enumerate}
\end{theorem}
\begin{proof}
The results will be derived from Theorem~\ref{t:5} applied to
$\bunder{\boldsymbol{\ZS}}=\zer\sad$ in $\XXX$, following the
general pattern of the deterministic proof of
\cite[Theorem~1]{Sadd22}. We use the notation of
Definition~\ref{d:S}, as well as \eqref{e:1179} and \eqref{e:2796}.
Note that, since $\mathscr{D}\neq\emp$, 
Proposition~\ref{p:6}\ref{p:6iv} asserts that $\zer\sad\neq\emp$.
Let us show that \eqref{e:long1} is a special case of
\eqref{e:a20}. We define the random indices
\begin{equation}
\label{e:9214}
(\forall\nnn)\quad
\begin{cases}
(\forall\iS\in\II)\quad
\vartheta_{\iS}(\nS)=\max\menge{\jjj}{\jS\leq\nS
\;\:\text{and}\;\:\iS\in I_{\jS}}\\
(\forall\kS\in\KK)\quad
\varrho_{\kS}(\nS)=\max\menge{\jjj}{\jS\leq\nS
\;\:\text{and}\;\:\kS\in K_{\jS}}.
\end{cases}
\end{equation}
It then follows from \eqref{e:long1} that
\begin{equation}
\label{e:4173}
(\forall\iS\in\II)(\forall\nnn)\quad
a_{\iS,\nS}=a_{\iS,\vartheta_{\iS}(\nS)}\;\Pas,\quad
a_{\iS,\nS}^*=a_{\iS,\vartheta_{\iS}(\nS)}^*\;\Pas,\quad
\xi_{\iS,\nS}=\xi_{\iS,\vartheta_{\iS}(\nS)}\;\Pas,
\end{equation}
and 
\begin{equation}
\label{e:6609}
(\forall\kS\in\KK)(\forall\nnn)\quad
\begin{cases}
b_{\kS,\nS}=b_{\kS,\varrho_{\kS}(\nS)}\;\Pas;\;\:
d_{\kS,\nS}=d_{\kS,\varrho_{\kS}(\nS)}\;\Pas;\;\:
\eta_{\kS,\nS}=\eta_{\kS,\varrho_{\kS}(\nS)}\;\Pas;\\
o_{\kS,\nS}^*=o_{\kS,\varrho_{\kS}(\nS)}^*\;\Pas;\;\:
q_{\kS,\nS}^*=q_{\kS,\varrho_{\kS}(\nS)}^*\;\Pas;\;\:
t_{\kS,\nS}^*=t_{\kS,\varrho_{\kS}(\nS)}^*\;\Pas
\end{cases}
\end{equation}
To match the notation of Theorem~\ref{t:5}, set
\begin{equation}
\label{e:8870}
(\forall\nnn)\quad
\begin{cases}
\bunder{\boldsymbol{x}}_{\nS}=(\boldsymbol{x}_{\nS},
\boldsymbol{y}_{\nS},\boldsymbol{z}_{\nS},\boldsymbol{v}_{\nS}^*);
\\
\widetilde{\bunder{\boldsymbol{q}}}_{\nS}=
(\boldsymbol{x}_{\nS},\boldsymbol{y}_{\nS},\boldsymbol{z}_{\nS},
\boldsymbol{o}_{\nS}^*);\\
\bunder{\boldsymbol{w}}_{\nS}=(\boldsymbol{a}_{\nS},
\boldsymbol{b}_{\nS},\boldsymbol{d}_{\nS},\boldsymbol{o}_{\nS}^*);
\\
\bunder{\boldsymbol{w}}_{\nS}^*=\big(
\boldsymbol{p}_{\nS}^*
-(\CS_{\iS}x_{\iS,\vartheta_{\iS}(\nS)})_{\iii},
\boldsymbol{q}_{\nS}^*
-(\BS_{\kS}^{\CC}y_{\kS,\varrho_{\kS}(\nS)})_{\kkk},
\boldsymbol{t}_{\nS}^*
-(\DS_{\kS}^{\CC}z_{\kS,\varrho_{\kS}(\nS)})_{\kkk},
\boldsymbol{o}_{\nS}\big);\\
\bunder{\boldsymbol{q}}_{\nS}=
\big((x_{\iS,\vartheta_{\iS}(\nS)})_{\iii},
(y_{\kS,\varrho_{\kS}(\nS)})_{\kkk},
(z_{\kS,\varrho_{\kS}(\nS)})_{\kkk},
(o_{\kS,\nS}^*)_{\kkk}\big);\\
\bunder{\boldsymbol{t}}_{\nS}^*=(
\boldsymbol{p}_{\nS}^*,\boldsymbol{q}_{\nS}^*,
\boldsymbol{t}_{\nS}^*,\boldsymbol{o}_{\nS});\\
\brk1{\bunder{\boldsymbol{e}}_{\nS},
\bunder{\boldsymbol{e}}_{\nS}^*,
\bunder{\boldsymbol{f}}_{\nS}^*}
=\brk1{\bunder{\boldsymbol{0}},
\bunder{\boldsymbol{0}},
\bunder{\boldsymbol{0}}}.
\end{cases}
\end{equation}
Then it follows from \eqref{e:GBeps} that,
for every $\nnn$ and every $\bunder{\boldsymbol{\zS}}\in\zer\sad$, 
$\varepsilon_{\nS}(\cdot,\bunder{\boldsymbol{\zS}})=0\;\Pas$ 
Next, we observe that, for every $\iS\in\II$ and every $\nnn$, 
\eqref{e:4173}, \eqref{e:9214}, \eqref{e:long1}, and
\cite[Proposition~23.2(ii)]{Livre1} imply that
\begin{align}
\label{e:6869}
a_{\iS,\nS}^*-\CS_{\iS}x_{\iS,\vartheta_{\iS}(\nS)}
&=a_{\iS,\vartheta_{\iS}(\nS)}^*
-\CS_{\iS}x_{\iS,\vartheta_{\iS}(\nS)}\nonumber\\
&=\upgamma_{\iS,\vartheta_{\iS}(\nS)}^{-1}
\big(x_{\iS,\vartheta_{\iS}(\nS)}
-a_{\iS,\vartheta_{\iS}(\nS)}\big)
-l_{\iS,\vartheta_{\iS}(\nS)}^*
-\CS_{\iS}x_{\iS,\vartheta_{\iS}(\nS)}
+\QS_{\iS}a_{\iS,\vartheta_{\iS}(\nS)}
\nonumber\\
&\in{-}\sS_{\iS}^*+\AS_{\iS}a_{\iS,\vartheta_{\iS}(\nS)}
+\QS_{\iS}a_{\iS,\vartheta_{\iS}(\nS)}
\nonumber\\
&={-}\sS_{\iS}^*+\AS_{\iS}a_{\iS,\nS}+\QS_{\iS}a_{\iS,\nS}\;\;\Pas
\end{align}
and, therefore, that
\begin{align}
\label{e:9073}
p_{\iS,\nS}^*-\CS_{\iS}x_{\iS,\vartheta_{\iS}(\nS)}
&=a_{\iS,\nS}^*-\CS_{\iS}x_{\iS,\vartheta_{\iS}(\nS)}
+\RS_{\iS}\boldsymbol{a}_{\nS}
+\sum_{\kkk}\LS_{\kS\iS}^*o_{\kS,\nS}^*
\nonumber\\
&\in{-}\sS_{\iS}^*+\AS_{\iS}a_{\iS,\nS}+\QS_{\iS}a_{\iS,\nS}
+\RS_{\iS}\boldsymbol{a}_{\nS}
+\sum_{\kkk}\LS_{\kS\iS}^*o_{\kS,\nS}^*\;\;\Pas
\end{align}
Likewise, we derive from 
\eqref{e:6609}, \eqref{e:9214}, \eqref{e:long1}, and 
\cite[Proposition~23.2(ii)]{Livre1} that
\begin{equation}
\label{e:3549}
(\forall\kS\in\KK)(\forall\nnn)\quad
\begin{cases}
q_{\kS,\nS}^*-\BS_{\kS}^{\CC}y_{\kS,\varrho_{\kS}(\nS)}
\in\BS_{\kS}^{\MM}b_{\kS,\nS}+\BS_{\kS}^{\LL}b_{\kS,\nS}
-o_{\kS,\nS}^*\;\;\Pas;\\
t_{\kS,\nS}^*-\DS_{\kS}^{\CC}z_{\kS,\varrho_{\kS}(\nS)}
\in\DS_{\kS}^{\MM}d_{\kS,\nS}+\DS_{\kS}^{\LL}d_{\kS,\nS}
-o_{\kS,\nS}^*\;\;\Pas;\\
o_{\kS,\nS}=\rS_{\kS}+b_{\kS,\nS}+d_{\kS,\nS}
-\sum_{\iii}\LS_{\kS\iS}a_{\iS,\nS}\;\;\Pas
\end{cases}
\end{equation}
In turn, we derive from \eqref{e:8870} and \eqref{e:1179} that
the sequence $(\bunder{\boldsymbol{w}}_{\nS},
\bunder{\boldsymbol{w}}_{\nS}^*)_{\nnn}$ lies in
$\gra\bunder{\boldsymbol{\mathsf{W}}}\;\Pas$
Next, using \eqref{e:8870} and \eqref{e:2796}, we obtain,
for every $\nnn$,
$\bunder{\boldsymbol{t}}_{\nS}^*
=\bunder{\boldsymbol{w}}_{\nS}^*
+\bunder{\boldsymbol{\CS}}\bunder{\boldsymbol{q}}_{\nS}\;\Pas$ 
Additionally, \eqref{e:long1} and \eqref{e:4173}--\eqref{e:8870}
yield
\begin{equation}
\label{e:4376}
(\forall\nnn)\quad
\Sum_{\iii}\xi_{\iS,\nS}+\Sum_{\kkk}\eta_{\kS,\nS}
=\|\bunder{\boldsymbol{w}}_{\nS}
-\bunder{\boldsymbol{q}}_{\nS}\|^2\quad\Pas
\end{equation}
Hence, in view of \eqref{e:long1},
\begin{equation}
\label{e:8350}
(\forall\nnn)\quad
\Delta_{\nS}=\scal{\bunder{\boldsymbol{x}}_{\nS}
-\bunder{\boldsymbol{w}}_{\nS}}{
\bunder{\boldsymbol{t}}_{\nS}^*}
-(4\upalpha)^{-1}\|\bunder{\boldsymbol{w}}_{\nS}
-\bunder{\boldsymbol{q}}_{\nS}\|^2\;\;\Pas
\end{equation}
On the other hand, 
\begin{equation}
(\forall\iS\in\II)(\forall\kS\in\KK)(\forall\nnn)\;\;
\begin{cases}
\RS_{\iS},\QS_{\iS}, \BS_{\kS}^{\LL},\DS_{\kS}^{\LL}\;
\text{and}\;\LS_{\kS\iS}\;\text{are Lipschitzian;}\\
\CS_{\iS}, \BS_{\kS}^{\CC},\;\text{and}\;\DS_{\kS}^{\CC}\;
\text{are cocoercive, hence Lipschitzian;}\\
\mathsf{J}_{\upgamma_{\iS,\nS}\AS_{\iS}},
\mathsf{J}_{\upmu_{\kS,\nS}\BS_{\kS}^{\MM}},\;\text{and}\;
\mathsf{J}_{\upnu_{\kS,\nS}\DS_{\kS}^{\MM}}\;
\text{are $1$-Lipschitzian.}
\end{cases}
\end{equation}
It therefore follows from 
Lemmas~\ref{l:100} and \ref{l:101}, and an inductive argument that
the variables defined in \eqref{e:8870} belong to
$L^2(\upOmega,\FE,\PP;\XXX)$. 
Altogether, taking into account the assumptions, we have shown that
\eqref{e:long1} is a realization of \eqref{e:a20}.
In turn, Theorem~\ref{t:5}\ref{t:5ii} asserts that
\begin{equation}
\label{e:4932}
\sum_{\nnn}\EE\|\bunder{\boldsymbol{x}}_{\nS+1}
-\bunder{\boldsymbol{x}}_{\nS}\|^2<\pinf.
\end{equation}

\ref{t:1i}--\ref{t:1ii}:
These follow from Theorem~\ref{t:5}\ref{t:5i-}, \eqref{e:4932}, 
and \eqref{e:8870}.

\ref{t:1iii-}--\ref{t:1iii}:
Theorem~\ref{t:5}\ref{t:5i} implies that
$(\bunder{\boldsymbol{x}}_{\nS})_{\nnn}$ is bounded $\Pas$
Therefore, arguing as in the proof of \cite[Theorem~1]{Sadd22}, 
\begin{equation}
\label{e:22}
\brk2{\widetilde{\bunder{\boldsymbol{q}}}_{\nS}}_{\nnn},\;\;
(\bunder{\boldsymbol{w}}_{\nS})_{\nnn},\;\:\text{and}\;\;
(\bunder{\boldsymbol{t}}_{\nS}^*)_{\nnn}\;\:\text{are bounded}\;\Pas
\end{equation}
As a result, \eqref{e:8350} and
Theorem~\ref{t:5}\ref{t:5iii} yield
\begin{equation}
\label{e:9758}
\varlimsup\big(
\scal{\bunder{\boldsymbol{x}}_{\nS}
-\bunder{\boldsymbol{w}}_{\nS}}
{\bunder{\boldsymbol{t}}_{\nS}^*}
-(4\upalpha)^{-1}\|\bunder{\boldsymbol{w}}_{\nS}
-\bunder{\boldsymbol{q}}_{\nS}\|^2\big)
=\varlimsup\Delta_{\nS}\leq 0\;\;\Pas
\end{equation}
Now define
\begin{equation}
\label{e:L}
\begin{cases}
\boldsymbol{\LS}\colon\HHS\to\GGS\colon\boldsymbol{\xS}\mapsto
\brk3{\sum_{\iii}\LS_{\kS\iS}\xS_{\iS}}_{\kkk},\;\;
\text{with adjoint}\;\;
\LLS^*\colon\GGS\to\HHS\colon\boldsymbol{\vS}^*
\mapsto\brk3{\sum_{\kkk}\LS_{\kS\iS}^*{\vS}^*_{\kS}}_{\iii}\\
\bunder{\boldsymbol{\mathsf{U}}}\colon\XXX\to\XXX\colon
\brk1{\boldsymbol{\xS},\boldsymbol{\yS},\boldsymbol{\zS},
\boldsymbol{\vS}^*}\mapsto\brk1{\LLS^*\boldsymbol{\vS}^*,
-\boldsymbol{\vS}^*,-\boldsymbol{\vS}^*,
-\LLS\boldsymbol{\xS}+\boldsymbol{\yS}+\boldsymbol{\zS}}.
\end{cases}
\end{equation}
Further, for every $\nnn$, set
\begin{equation}
\label{e:3958}
\begin{cases}
(\forall\iS\in\II)\;\;
\mathsf{F}_{\iS,\nS}
=\upgamma_{\iS,\vartheta_{\iS}(\nS)}^{-1}\Id-\QS_{\iS};\\
(\forall\kS\in\KK)\;\;
\mathsf{S}_{\kS,\nS}
=\upmu_{\kS,\varrho_{\kS}(\nS)}^{-1}\Id-\BS_{\kS}^{\LL};\;\:
\mathsf{T}_{\kS,\nS}
=\upnu_{\kS,\varrho_{\kS}(\nS)}^{-1}\Id-\DS_{\kS}^{\LL};\\
\bunder{\boldsymbol{\mathsf{F}}}_{\nS}\colon\XXX\to\XXX\colon
(\boldsymbol{\xS},\boldsymbol{\yS},\boldsymbol{\zS},
\boldsymbol{\vS}^*)\mapsto
\brk1{(\mathsf{F}_{\iS,\nS}\xS_{\iS})_{\iii},
(\mathsf{S}_{\kS,\nS}\yS_{\kS})_{\kkk},
(\mathsf{T}_{\kS,\nS}\zS_{\kS})_{\kkk},
(\upsigma_{\kS,\varrho_{\kS}(\nS)}^{-1}
\vS_{\kS}^*)_{\kkk}}
\end{cases}
\end{equation}
and
\begin{equation}
\label{e:3514}
\begin{cases}
\bunder{\widetilde{\boldsymbol{x}}}_{\nS}=
\brk2{(x_{\iS,\vartheta_{\iS}(\nS)})_{\iii},
(y_{\kS,\varrho_{\kS}(\nS)})_{\kkk},
(z_{\kS,\varrho_{\kS}(\nS)})_{\kkk},
(v_{\kS,\varrho_{\kS}(\nS)}^*)_{\kkk}};\\
\bunder{\boldsymbol{v}}_{\nS}^*=
\bunder{\boldsymbol{\mathsf{F}}}_{\nS}
\bunder{\boldsymbol{x}}_{\nS}
-\bunder{\boldsymbol{\mathsf{F}}}_{\nS}
\bunder{\boldsymbol{w}}_{\nS};\;\:
\bunder{\boldsymbol{u}}_{\nS}^*=
\bunder{\boldsymbol{\mathsf{U}}}
\bunder{\boldsymbol{w}}_{\nS}
-\bunder{\boldsymbol{\mathsf{U}}}
\bunder{\boldsymbol{x}}_{\nS};\\
\bunder{\boldsymbol{r}}_{\nS}^*
=\big((\RS_{\iS}\boldsymbol{a}_{\nS}
-\RS_{\iS}\boldsymbol{x}_{\nS})_{\iii},
\boldsymbol{0},\boldsymbol{0},
\boldsymbol{0}\big);\;\:
\bunder{\widetilde{\boldsymbol{r}}}_{\nS}^*
=\big((\RS_{\iS}\boldsymbol{a}_{\nS}
-\RS_{\iS}\boldsymbol{x}_{\vartheta_{\iS}(\nS)})_{\iii},
\boldsymbol{0},\boldsymbol{0},
\boldsymbol{0}\big);\\
\bunder{\boldsymbol{l}}_{\nS}^*=
\brk2{
\brk1{{-}\sum_{\kkk}\LS_{\kS\iS}^*
v_{\kS,\vartheta_{\iS}(\nS)}^*}_{\iii},
\brk1{v_{\kS,\varrho_{\kS}(\nS)}^*}_{\kkk},
\brk1{v_{\kS,\varrho_{\kS}(\nS)}^*}_{\kkk},
\brk1{\sum_{\iii}\LS_{\kS\iS}x_{\iS,\varrho_{\kS}(\nS)}
-y_{\kS,\varrho_{\kS}(\nS)}-z_{\kS,\varrho_{\kS}(\nS)}}_{\kkk}}.
\end{cases}
\end{equation}
Assumptions \ref{prob:1a}--\ref{prob:1c} in Problem~\ref{prob:1},
together with Lemma~\ref{l:A2}, imply that
\begin{equation}
\label{e:6983}
(\forall\nnn)\quad
\text{the operators}\;\:
\begin{cases}
(\mathsf{F}_{\iS,\nS})_{\iii}\;\:\text{are 
$(\upchi+\upphi)$-strongly monotone};\\
(\mathsf{S}_{\kS,\nS})_{\kkk}\:\:\text{and}\;\:
(\mathsf{T}_{\kS,\nS})_{\kkk}\;\:
\text{are $\upphi$-strongly monotone}.
\end{cases}
\end{equation}
Consequently, in view of \eqref{e:3958}, there exists 
$\upkappa\in\RPP$ such that
\begin{equation}
\label{e:5297}
\text{the operators}\;\:
(\bunder{\boldsymbol{\mathsf{F}}}_{\nS})_{\nnn}\;\:
\text{are $\upkappa$-Lipschitzian}.
\end{equation}
Next, using the same arguments as in the proof of 
\cite[Theorem~1]{Sadd22}, we obtain
\begin{equation}
\label{e:2842}
(\forall\nnn)\quad\bunder{\boldsymbol{\mathsf{t}}}_{\nS}^*=
\bunder{\boldsymbol{\mathsf{F}}}_{\nS}
\widetilde{\bunder{\boldsymbol{x}}}_{\nS}
-\bunder{\boldsymbol{\mathsf{F}}}_{\nS}
\bunder{\boldsymbol{w}}_{\nS}
+\widetilde{\bunder{\boldsymbol{r}}}_{\nS}^*
+\bunder{\boldsymbol{l}}_{\nS}^*
+\bunder{\boldsymbol{\mathsf{U}}}
\bunder{\boldsymbol{w}}_{\nS}\;\Pas
\end{equation}
We also note that, in view of \eqref{e:4932}, \eqref{e:9214},
and Assumption~\ref{a:1}, Proposition~\ref{p:4} and Lemma~\ref{l:1}
imply that
\begin{equation}
\label{e:8706}
(\forall\iS\in\II)(\forall\kS\in\KK)\quad
\begin{cases}
\boldsymbol{x}_{\vartheta_{\iS}(\nS)}-\boldsymbol{x}_{\nS}\Pto
\boldsymbol{0};\;\:
\boldsymbol{x}_{\varrho_{\kS}(\nS)}-\boldsymbol{x}_{\nS}\Pto
\boldsymbol{0};\\
\boldsymbol{v}_{\vartheta_{\iS}(\nS)}^*-\boldsymbol{v}_{\nS}^*\Pto
\boldsymbol{0};\;\:
\boldsymbol{y}_{\varrho_{\kS}(\nS)}-\boldsymbol{y}_{\nS}\Pto
\boldsymbol{0};\\
\boldsymbol{z}_{\varrho_{\kS}(\nS)}-\boldsymbol{z}_{\nS}\Pto
\boldsymbol{0};\;\:
\boldsymbol{v}_{\varrho_{\kS}(\nS)}^*-\boldsymbol{v}_{\nS}^*\Pto
\boldsymbol{0}.
\end{cases}
\end{equation}
Thus, \eqref{e:3514} and \eqref{e:L} yield
\begin{equation}
\label{e:6047}
\bunder{\boldsymbol{l}}_{\nS}^*
+\bunder{\boldsymbol{\mathsf{U}}}
\bunder{\boldsymbol{x}}_{\nS}
\Pto\bunder{\boldsymbol{0}},
\end{equation}
while assumption \ref{prob:1e} in Problem~\ref{prob:1} gives
\begin{equation}
\label{e:404}
(\forall\iS\in\II)\quad
\|\RS_{\iS}\boldsymbol{x}_{\vartheta_{\iS}(\nS)}
-\RS_{\iS}\boldsymbol{x}_{\nS}\|
\leq\upchi\|\boldsymbol{x}_{\vartheta_{\iS}(\nS)}
-\boldsymbol{x}_{\nS}\|\Pto0.
\end{equation}
On the other hand, \eqref{e:5297}, \eqref{e:3514}, and
\eqref{e:8706} yield
\begin{equation}
\label{e:8506}
\|\bunder{\boldsymbol{\mathsf{F}}}_{\nS}
\widetilde{\bunder{\boldsymbol{x}}}_{\nS}
-\bunder{\boldsymbol{\mathsf{F}}}_{\nS}
\bunder{\boldsymbol{x}}_{\nS}\|
\leq\upkappa\|\widetilde{\bunder{\boldsymbol{x}}}_{\nS}
-\bunder{\boldsymbol{x}}_{\nS}\|\Pto 0
\end{equation}
which, combined with 
\eqref{e:2842}, \eqref{e:3514}, 
\eqref{e:6047}, and \eqref{e:404} leads to
\begin{equation}
\label{e:9860}
\bunder{\boldsymbol{t}}_{\nS}^*
-\big(\bunder{\boldsymbol{v}}_{\nS}^*
+\bunder{\boldsymbol{r}}_{\nS}^*
+\bunder{\boldsymbol{u}}_{\nS}^*\big)
=\bunder{\boldsymbol{l}}_{\nS}^*
+\bunder{\boldsymbol{\mathsf{U}}}
\bunder{\boldsymbol{x}}_{\nS}
+\bunder{\boldsymbol{\mathsf{F}}}_{\nS}
\widetilde{\bunder{\boldsymbol{x}}}_{\nS}
-\bunder{\boldsymbol{\mathsf{F}}}_{\nS}
\bunder{\boldsymbol{x}}_{\nS}
+\widetilde{\bunder{\boldsymbol{r}}}_{\nS}^*
-\bunder{\boldsymbol{r}}_{\nS}^*
\Pto\bunder{\boldsymbol{0}}.
\end{equation}
Additionally, \eqref{e:8870} and \eqref{e:8706} yield
\begin{equation}
\label{e:5101}
\widetilde{\bunder{\boldsymbol{q}}}_{\nS}
-\bunder{\boldsymbol{q}}_{\nS}\Pto
\bunder{\boldsymbol{0}}.
\end{equation}
Therefore, by Cauchy--Schwarz and \eqref{e:22}, 
\begin{equation}
\label{e:5202}
\big|\scal1{\bunder{\boldsymbol{w}}_{\nS}
-\widetilde{\bunder{\boldsymbol{q}}}_{\nS}}
{\widetilde{\bunder{\boldsymbol{q}}}_{\nS}
-\bunder{\boldsymbol{q}}_{\nS}}\big|
\leq\bigg(\sup_{\mathsf{m}\in\NN}
\|\bunder{\boldsymbol{w}}_{\mathsf{m}}\|
+\sup_{\mathsf{m}\in\NN}
\|\widetilde{\bunder{\boldsymbol{q}}}_{\mathsf{m}}\|\bigg)
\|\widetilde{\bunder{\boldsymbol{q}}}_{\nS}
-\bunder{\boldsymbol{q}}_{\nS}\|\Pto 0
\end{equation}
while, by \eqref{e:9860},
\begin{equation}
\label{e:5203}
\big|\scal{\bunder{\boldsymbol{x}}_{\nS}
-\bunder{\boldsymbol{w}}_{\nS}}{
\bunder{\boldsymbol{t}}_{\nS}^*
-(\bunder{\boldsymbol{v}}_{\nS}^*
+\bunder{\boldsymbol{r}}_{\nS}^*
+\bunder{\boldsymbol{u}}_{\nS}^*)
}\big|
\leq\bigg(\sup_{\mathsf{m}\in\NN}
\|\bunder{\boldsymbol{x}}_{\mathsf{m}}\|
+\sup_{\mathsf{m}\in\NN}\|\bunder{\boldsymbol{w}}_{\mathsf{m}}\|
\bigg)\|\bunder{\boldsymbol{t}}_{\nS}^*
-(\bunder{\boldsymbol{v}}_{\nS}^*
+\bunder{\boldsymbol{r}}_{\nS}^*
+\bunder{\boldsymbol{u}}_{\nS}^*)\|\Pto 0.
\end{equation}
However, it follows from \eqref{e:L} and assumption \ref{prob:1d}
in Problem~\ref{prob:1} that
$\bunder{\boldsymbol{\mathsf{U}}}$ is linear and bounded, with 
$\bunder{\boldsymbol{\mathsf{U}}}^*
={-}\bunder{\boldsymbol{\mathsf{U}}}$. It then results from 
\eqref{e:3514} that, for every $\nnn$, 
$\scal{\bunder{\boldsymbol{x}}_{\nS}
-\bunder{\boldsymbol{w}}_{\nS}}{
\bunder{\boldsymbol{u}}_{\nS}^*}=0\;\Pas$
On the other hand, note that, for every $\nnn$,
\begin{align}
\label{e:5224}
&\scal{\bunder{\boldsymbol{x}}_{\nS}
-\bunder{\boldsymbol{w}}_{\nS}}
{\bunder{\boldsymbol{t}}_{\nS}^*}
-(4\upalpha)^{-1}\|\bunder{\boldsymbol{w}}_{\nS}
-\bunder{\boldsymbol{q}}_{\nS}\|^2\nonumber\\
&\quad=\scal{\bunder{\boldsymbol{x}}_{\nS}
-\bunder{\boldsymbol{w}}_{\nS}}{
\bunder{\boldsymbol{v}}_{\nS}^*
+\bunder{\boldsymbol{r}}_{\nS}^*
+\bunder{\boldsymbol{u}}_{\nS}^*}
+\scal{\bunder{\boldsymbol{x}}_{\nS}
-\bunder{\boldsymbol{w}}_{\nS}}{
\bunder{\boldsymbol{t}}_{\nS}^*
-(\bunder{\boldsymbol{v}}_{\nS}^*
+\bunder{\boldsymbol{r}}_{\nS}^*
+\bunder{\boldsymbol{u}}_{\nS}^*)}
-(4\upalpha)^{-1}\|\bunder{\boldsymbol{w}}_{\nS}
-\bunder{\boldsymbol{q}}_{\nS}\|^2\nonumber\\
&\quad=\scal{\bunder{\boldsymbol{x}}_{\nS}
-\bunder{\boldsymbol{w}}_{\nS}}{
\bunder{\boldsymbol{v}}_{\nS}^*
+\bunder{\boldsymbol{r}}_{\nS}^*}
+\scal{\bunder{\boldsymbol{x}}_{\nS}
-\bunder{\boldsymbol{w}}_{\nS}}{
\bunder{\boldsymbol{t}}_{\nS}^*
-(\bunder{\boldsymbol{v}}_{\nS}^*
+\bunder{\boldsymbol{r}}_{\nS}^*
+\bunder{\boldsymbol{u}}_{\nS}^*)}
\nonumber\\
&\quad\qquad-(4\upalpha)^{-1}\brk2{
\norm1{\bunder{\boldsymbol{w}}_{\nS}
-\widetilde{\bunder{\boldsymbol{q}}}_{\nS}}^2
+2\scal1{\bunder{\boldsymbol{w}}_{\nS}
-\widetilde{\bunder{\boldsymbol{q}}}_{\nS}}
{\widetilde{\bunder{\boldsymbol{q}}}_{\nS}
-\bunder{\boldsymbol{q}}_{\nS}}
+\norm1{\widetilde{\bunder{\boldsymbol{q}}}_{\nS}
-\bunder{\boldsymbol{q}}_{\nS}}^2}\;\;\Pas
\end{align}
Moreover, as in \cite[Equation~(95)]{Sadd22}, 
it follows from \eqref{e:3514}, \eqref{e:8870},
\eqref{e:3958}, \eqref{e:6983},
and assumption \ref{prob:1e} in Problem~\ref{prob:1} that, for 
every $\nnn$,
\begin{multline}
\label{e:1234}
\scal{\bunder{\boldsymbol{x}}_{\nS}
-\bunder{\boldsymbol{w}}_{\nS}}{
\bunder{\boldsymbol{v}}_{\nS}^*
+\bunder{\boldsymbol{r}}_{\nS}^*}
-(4\upalpha)^{-1}\|\bunder{\boldsymbol{w}}_{\nS}
-\widetilde{\bunder{\boldsymbol{q}}}_{\nS}\|^2\\
\geq\big(\upphi-(4\upalpha)^{-1}\big)
\big(\|\boldsymbol{x}_{\nS}-\boldsymbol{a}_{\nS}\|^2
+\|\boldsymbol{y}_{\nS}-\boldsymbol{b}_{\nS}\|^2
+\|\boldsymbol{z}_{\nS}-\boldsymbol{d}_{\nS}\|^2\big)
+\upvarepsilon\|\boldsymbol{v}_{\nS}^*-\boldsymbol{o}_{\nS}^*\|^2\;
\;\Pas
\end{multline}
For every $\nnn$, let us define
\begin{equation}
\begin{cases}
\xi_{\nS}=\big(\upphi-(4\upalpha)^{-1}\big)
\big(\|\boldsymbol{x}_{\nS}-\boldsymbol{a}_{\nS}\|^2
+\|\boldsymbol{y}_{\nS}-\boldsymbol{b}_{\nS}\|^2
+\|\boldsymbol{z}_{\nS}-\boldsymbol{d}_{\nS}\|^2\big)
+\upvarepsilon\|\boldsymbol{v}_{\nS}^*-\boldsymbol{o}_{\nS}^*\|^2;
\\
\chi_{\nS}=\scal{\bunder{\boldsymbol{x}}_{\nS}
-\bunder{\boldsymbol{w}}_{\nS}}{
\bunder{\boldsymbol{t}}_{\nS}^*
-(\bunder{\boldsymbol{v}}_{\nS}^*
+\bunder{\boldsymbol{r}}_{\nS}^*
+\bunder{\boldsymbol{u}}_{\nS}^*)}
-(4\upalpha)^{-1}\big(
2\scal{\bunder{\boldsymbol{w}}_{\nS}
-\widetilde{\bunder{\boldsymbol{q}}}_{\nS}}
{\widetilde{\bunder{\boldsymbol{q}}}_{\nS}
-\bunder{\boldsymbol{q}}_{\nS}}
+\|\widetilde{\bunder{\boldsymbol{q}}}_{\nS}
-\bunder{\boldsymbol{q}}_{\nS}\|^2\big).
\end{cases}
\end{equation}
Then $\inf_{\nnn}\xi_{\nS}\geq0\;\Pas$ Moreover, \eqref{e:5224}
and \eqref{e:1234} imply that, for every $\nnn$,
$\xi_{\nS}+\chi_{\nS}\leq\Delta_{\nS}\;\Pas$ In addition,
$\varlimsup\Delta_{\nS}\leq0\;\Pas$ by \eqref{e:9758} and 
$\chi_{\nS}\Pto 0$ by 
\eqref{e:5101}--\eqref{e:5203}. Therefore, in view of 
Lemma~\ref{l:3}, $\xi_{\nS}\Pto0$ and therefore
\begin{equation}
\label{e:3513}
\boldsymbol{x}_{\nS}-\boldsymbol{a}_{\nS}\Pto\boldsymbol{0},\;\;
\boldsymbol{y}_{\nS}-\boldsymbol{b}_{\nS}\Pto\boldsymbol{0},\;\;
\boldsymbol{z}_{\nS}-\boldsymbol{d}_{\nS}\Pto\boldsymbol{0},\;\;
\boldsymbol{v}_{\nS}^*-\boldsymbol{o}_{\nS}^*\Pto\boldsymbol{0},
\end{equation}
which establishes \ref{t:1iii-}.
In turn, \eqref{e:8870} and \eqref{e:5297} force
\begin{equation}
\bunder{\boldsymbol{x}}_{\nS}
-\bunder{\boldsymbol{w}}_{\nS}\Pto
\bunder{\boldsymbol{0}}\;
\quad\text{and}\quad(\forall\nnn)\quad
\|\bunder{\boldsymbol{\mathsf{F}}}_{\nS}
\bunder{\boldsymbol{x}}_{\nS}
-\bunder{\boldsymbol{\mathsf{F}}}_{\nS}
\bunder{\boldsymbol{w}}_{\nS}\|
\leq\upkappa\|\bunder{\boldsymbol{x}}_{\nS}
-\bunder{\boldsymbol{w}}_{\nS}\|.
\end{equation}
Hence,
$\bunder{\boldsymbol{\mathsf{F}}}_{\nS}
\bunder{\boldsymbol{x}}_{\nS}
-\bunder{\boldsymbol{\mathsf{F}}}_{\nS}
\bunder{\boldsymbol{w}}_{\nS}\Pto
\bunder{\boldsymbol{0}}$.
Likewise, \eqref{e:8706} yields
$\bunder{\boldsymbol{w}}_{\nS}-\bunder{\boldsymbol{q}}_{\nS}
\Pto\bunder{\boldsymbol{0}}$.
Further, we infer from \eqref{e:3514}, \eqref{e:3513},
and Problem~\ref{prob:1}\ref{prob:1e} that
$\|\bunder{\boldsymbol{r}}_{\nS}^*\|^2
=\|\boldsymbol{\RS}\boldsymbol{a}_{\nS}
-\boldsymbol{\RS}\boldsymbol{x}_{\nS}\|^2
\leq\upchi^2\|\boldsymbol{a}_{\nS}-\boldsymbol{x}_{\nS}\|^2\Pto 0$.
As a result, \eqref{e:3514} and \eqref{e:9860} imply
that
\begin{equation}
\bunder{\boldsymbol{t}}_{\nS}^*=
\brk2{\bunder{\boldsymbol{t}}_{\nS}^*
-\big(\bunder{\boldsymbol{v}}_{\nS}^*
+\bunder{\boldsymbol{r}}_{\nS}^*
+\bunder{\boldsymbol{u}}_{\nS}^*\big)}
+\big(\bunder{\boldsymbol{\mathsf{F}}}_{\nS}
\bunder{\boldsymbol{x}}_{\nS}
-\bunder{\boldsymbol{\mathsf{F}}}_{\nS}
\bunder{\boldsymbol{w}}_{\nS}\big)
+\bunder{\boldsymbol{\mathsf{U}}}
(\bunder{\boldsymbol{w}}_{\nS}
-\bunder{\boldsymbol{x}}_{\nS})
+\bunder{\boldsymbol{r}}_{\nS}^*
\Pto\bunder{\boldsymbol{0}}.
\end{equation}
Altogether, 
$\bunder{\boldsymbol{x}}_{\nS}-\bunder{\boldsymbol{w}}_{\nS}-
\bunder{\boldsymbol{e}}_{\nS}\Pto
\bunder{\boldsymbol{0}}$, 
$\bunder{\boldsymbol{w}}_{\nS}+\bunder{\boldsymbol{e}}_{\nS}
-\bunder{\boldsymbol{q}}_{\nS}\Pto
\bunder{\boldsymbol{0}}$ and
$\bunder{\boldsymbol{w}}_{\nS}^*+\bunder{\boldsymbol{e}}_{\nS}^*+
\bunder{\boldsymbol{\mathsf{C}}}
\bunder{\boldsymbol{q}}_{\nS}\Pto
\bunder{\boldsymbol{0}}$, 
and Theorem~\ref{t:5}\ref{t:5iv} therefore
guarantees that there exists a $\zer\sad$-valued random variable
$\overline{\bunder{\boldsymbol{x}}}
=(\overline{\boldsymbol{x}},\overline{\boldsymbol{y}},
\overline{\boldsymbol{z}},\overline{\boldsymbol{v}}^*)$
such that $\bunder{\boldsymbol{x}}_{\nS}\to
\overline{\bunder{\boldsymbol{x}}}\;\Pas$
This and \eqref{e:3513} imply that, for every $\iS\in\II$
and every $\kS\in\KK$,
$x_{\iS,\nS}\to\overline{x}_{\iS}\;\Pas$,
$a_{\iS,\nS}\to\overline{x}_{\iS}\;\Pas$,
and $v_{\kS,\nS}^*\to\overline{v}_{\kS}^*\;\Pas$
Finally, Proposition~\ref{p:6}\ref{p:6ii} asserts that
$\overline{\boldsymbol{x}}$
solves \eqref{e:1p} $\Pas$ and that
$\overline{\boldsymbol{v}}^*$ solves \eqref{e:1d} $\Pas$
\end{proof}

\begin{remark}
\label{r:77}
Here are some observations pertaining to Theorem~\ref{t:1}.
\begin{enumerate}
\item
\label{r:77i}
There does not exist any almost sure convergence result on
stochastic algorithms for solving Problem~\ref{prob:1} with random
block selection or random relaxations. In the case of deterministic
relaxations $(\uplambda_{\nS})_{\nnn}$ in $\left]0,2\right[$ and
deterministic blocks selection (see Example~\ref{ex:u}\ref{ex:ui}),
Theorem~\ref{t:1} appears in \cite[Theorem~1(iv)]{Sadd22}.
\item
\label{r:77ii}
In \cite[Section~5.2]{Siop15}, a randomly activated 
block-iterative method with guaranteed almost sure convergence of
the iterates is proposed for the special case 
\begin{equation}
\label{e:5p}
\text{find}\;\:\overline{\boldsymbol{\xS}}\in\HHS
\;\:\text{such that}\;\:(\forall\iS\in\II)\;\;
0\in\AS_{\iS}\overline{\xS}_{\iS}+
\RS_{\iS}\overline{\boldsymbol{\xS}}
\end{equation}
of Problem~\ref{prob:1}, where $\boldsymbol{\mathsf{R}}$ is
restricted to be cocoercive and no random relaxations are
used. 
\item
\label{r:77iii}
For notational simplicity, we have not considered stochastic 
errors in the evaluations of the single-valued operators and the
resolvents, as is done in the simpler formulations of 
Theorems~\ref{t:3} and \ref{t:7} and, for specific problems, in
\cite{Siim25,Siop15,Pafa16,John24,Pesq15,Rosa16}. 
For this reason, we employ Theorem~\ref{t:5} with 
$(\forall\nnn)(\forall\bunder{\boldsymbol{\zS}}\in\zer\sad)$
$\varepsilon_{\nS}(\cdot,\bunder{\boldsymbol{\zS}})=0\;\Pas$ 
Such stochastic errors can be introduced in Algorithm~\ref{algo:1}
under the same condition as in \eqref{t:7as} to
guarantee that $(\forall\bunder{\boldsymbol{\zS}}\in\zer\sad)$
$\sum_{\nnn}\EE\varepsilon_{\nS}(\cdot,\bunder{\boldsymbol{\zS}})
\EE\lambda_{\nS}<\pinf$.
\end{enumerate}
\end{remark}

\subsection{Application to variational inequalities}
\label{sec:63}

We propose a stochastic block-iterative algorithm for solving 
variational inequalities relative to a Minkowski sum of composite
convex sets; see \cite{Acnu24,Ruan22,Wang20} and their
bibliographies for instantiations of this model.

\newpage
\begin{problem}
\label{prob:8}
Let $0<\mS\in\NN$ and let $(\HS_{\iS})_{1\leq\iS\leq\mS}$ and $\GS$
be Euclidean spaces. For every $\iS\in\{1,\ldots,\mS\}$, let
$\mathsf{E}_{\iS}$ and $\mathsf{F}_{\iS}$ be closed convex subsets
of $\HS_{\iS}$ such that
$\reli\mathsf{E}_{\iS}\cap\reli\mathsf{F}_{\iS}\neq\emp$ and let 
$\LS_{\iS}\colon\HS_{\iS}\to\GS$ be linear. In addition, let
$\CS\colon\GS\to\GS$ be cocoercive with constant $\upalpha\in\RPP$.
The objective is to
\begin{equation}
\label{e:48}
\text{find}\;\:\overline{\sS}\in
\sum_{\iS=1}^\mS\LS_{\iS}(\mathsf{E}_{\iS}\cap\mathsf{F}_{\iS})
\;\:\text{such that}\;\:
\brk3{\forall\sS\in\sum_{\iS=1}^\mS
\LS_{\iS}(\mathsf{E}_{\iS}\cap \mathsf{F}_{\iS})}\;\;
\scal1{\overline{\sS}-\sS}{\CS\overline{\sS}}\leq 0,
\end{equation}
under the assumption that the set $\mathscr{V}$ of solutions to
this variational inequality is nonempty. 
\end{problem}

Our stochastic algorithm is motivated by the deterministic
algorithm of \cite[Equation~(112)]{Sadd22}.

\begin{algorithm}
\label{algo:8}
Consider the setting of Problem~\ref{prob:8},
set $\II=\{1,\ldots,\mS\}$, set $\KK=\{1,\ldots,\mS+1\}$, 
and suppose that Assumption~\ref{a:1} is in force. 
Let $\upvarepsilon\in\left]0,\min\{1,4\upalpha\}\right[$, let
$\upphi\in\left]1/(4\upalpha),1/\upvarepsilon\right[$,
let $\upgamma\in[\upvarepsilon,1/\upvarepsilon]$, 
let $\upmu\in[\upvarepsilon,1/\upphi]$, 
let $\uptau\in[\upvarepsilon,1/\upvarepsilon]$,
let $\uprho\in[2,\pinf[$. Further, let
$\{y_{\mS+1,0},z_{\mS+1,0},v_{\mS+1,0}^*\}\subset 
L^2(\upOmega,\FE,\PP;\GS)$ and, for every $\iS\in\II$, let
$\{x_{\iS,0},y_{\iS,0},z_{\iS,0},v_{\iS,0}^*\}\subset 
L^2(\upOmega,\FE,\PP;\HS_{\iS})$. Iterate
\begin{equation}
\label{e:7133}
\begin{array}{l}
\text{for}\;\nS=0,1,\ldots\\
\left\lfloor
\begin{array}{l}
s_{\nS}=\sum_{\iii}\LS_{\iS}x_{\iS,\nS};\\
\text{for every}\;\iS\in I_{\nS}\\
\left\lfloor
\begin{array}{l}
l_{\iS,\nS}^*=v_{\iS,\nS}^*+\LS_{\iS}^*v_{\mS+1,\nS}^*;\\
a_{\iS,\nS}
=\proj_{\mathsf{E}_{\iS}}
\big(x_{\iS,\nS}-\upgamma l_{\iS,\nS}^*\big);\;
a_{\iS,\nS}^*
=\upgamma^{-1}(x_{\iS,\nS}-a_{\iS,\nS})-l_{\iS,\nS}^*;\\
\xi_{\iS,\nS}=\|a_{\iS,\nS}-x_{\iS,\nS}\|^2;
\end{array}
\right.\\
\text{for every}\;\iS\in\II\smallsetminus I_{\nS}\\
\left\lfloor
\begin{array}{l}
a_{\iS,\nS}=a_{\iS,\nS-1};\;a_{\iS,\nS}^*=a_{\iS,\nS-1}^*;\;
\xi_{\iS,\nS}=\xi_{\iS,\nS-1};\\
\end{array}
\right.\\
\text{for every}\;\kS\in K_{\nS}\\
\left\lfloor
\begin{array}{l}
\text{if}\;\kS\in\II\\
\left\lfloor
\begin{array}{l}
b_{\kS,\nS}
=\proj_{\mathsf{F}_{\kS}}
\big(y_{\kS,\nS}+\upmu v_{\kS,\nS}^*\big);\\
o_{\kS,\nS}=b_{\kS,\nS}-a_{\kS,\nS};\;
o_{\kS,\nS}^*=\uptau(x_{\kS,\nS}-y_{\kS,\nS}-z_{\kS,\nS})
+v_{\kS,\nS}^*;\\
\end{array}
\right.\\
\text{if}\;\kS=\mS+1\\
\left\lfloor
\begin{array}{l}
b_{\kS,\nS}=y_{\kS,\nS}+\upmu(v_{\kS,\nS}^*-\CS y_{\kS,\nS});\\
o_{\kS,\nS}=b_{\kS,\nS}-\sum_{\iii}\LS_{\iS}a_{\iS,\nS};\;
o_{\kS,\nS}^*=\uptau\big(s_{\nS}-y_{\kS,\nS}-z_{\kS,\nS}\big)
+v_{\kS,\nS}^*;\\
\end{array}
\right.\\
q_{\kS,\nS}^*=\upmu^{-1}(y_{\kS,\nS}-b_{\kS,\nS})+v_{\kS,\nS}^*
-o_{\kS,\nS}^*;\;
t_{\kS,\nS}^*=\upmu^{-1}z_{\kS,\nS}+v_{\kS,\nS}^*-o_{\kS,\nS}^*;\\
\eta_{\kS,\nS}=\|b_{\kS,\nS}-y_{\kS,\nS}\|^2+\|z_{\kS,\nS}\|^2;
\end{array}
\right.\\
\text{for every}\;\kS\in\KK\smallsetminus K_{\nS}\\
\left\lfloor
\begin{array}{l}
b_{\kS,\nS}=b_{\kS,\nS-1};\;
o_{\kS,\nS}^*=o_{\kS,\nS-1}^*;\;
q_{\kS,\nS}^*=q_{\kS,\nS-1}^*;\;
t_{\kS,\nS}^*=t_{\kS,\nS-1}^*;\;
\eta_{\kS,\nS}=\eta_{\kS,\nS-1};\\
\text{if}\;\kS\in\II\\
\left\lfloor
\begin{array}{l}
o_{\kS,\nS}=b_{\kS,\nS}-a_{\kS,\nS};
\end{array}
\right.\\
\text{if}\;\kS=\mS+1\\
\left\lfloor
\begin{array}{l}
o_{\kS,\nS}=b_{\kS,\nS}-\sum_{\iii}\LS_{\iS}a_{\iS,\nS};
\end{array}
\right.\\[2mm]
\end{array}
\right.\\
\text{for every}\;\iS\in\II\\
\left\lfloor
\begin{array}{l}
p_{\iS,\nS}^*=a_{\iS,\nS}^*+o_{\iS,\nS}^*
+\LS_{\iS}^*o_{\mS+1,\nS}^*;
\end{array}
\right.\\
\begin{aligned}
\Delta_{\nS}&=\textstyle
{-}(4\upalpha)^{-1}\big(\sum_{\iii}\xi_{\iS,\nS}
+\sum_{\kkk}\eta_{\kS,\nS}\big)
+\sum_{\iii}\scal{x_{\iS,\nS}-a_{\iS,\nS}}{p_{\iS,\nS}^*}\\
&\textstyle
\quad\;+\sum_{\kkk}
\big(\scal{y_{\kS,\nS}-b_{\kS,\nS}}{q_{\kS,\nS}^*}
+\scal{z_{\kS,\nS}}{t_{\kS,\nS}^*}
+\scal{o_{\kS,\nS}}{v_{\kS,\nS}^*-o_{\kS,\nS}^*}\big);
\end{aligned}\\[4mm]
\theta_{\nS}=
\dfrac{\mathsf{1}_{[\Delta_{\nS}>0]}\Delta_{\nS}}
{\sum_{\iii}\|p_{\iS,\nS}^*\|^2
+\sum_{\kkk}\big(
\|q_{\kS,\nS}^*\|^2+\|t_{\kS,\nS}^*\|^2+\|o_{\kS,\nS}\|^2\big)
+\mathsf{1}_{[\Delta_{\nS}\leq 0]}};\\
\lambda_{\nS}\in 
L^\infty(\upOmega,\FE,\PP;[\upvarepsilon,\uprho])\\
\text{for every}\;\iS\in\II\\
\left\lfloor
\begin{array}{l}
x_{\iS,\nS+1}=x_{\iS,\nS}-\lambda_{\nS}\theta_{\nS}p_{\iS,\nS}^*;
\end{array}
\right.\\
\text{for every}\;\kS\in\KK\\
\left\lfloor
\begin{array}{l}
y_{\kS,\nS+1}=y_{\kS,\nS}-\lambda_{\nS}\theta_{\nS}q_{\kS,\nS}^*;\;
z_{\kS,\nS+1}=z_{\kS,\nS}-\lambda_{\nS}\theta_{\nS}t_{\kS,\nS}^*;\;
v_{\kS,\nS+1}^*=v_{\kS,\nS}^*-\lambda_{\nS}\theta_{\nS}o_{\kS,\nS}.
\end{array}
\right.\\[1mm]
\end{array}
\right.
\end{array}
\end{equation}
\end{algorithm}

Below, we prove the convergence of the sequence
$(s_{\nS})_{\nnn}$ generated by Algorithm~\ref{algo:8} to a
solution to Problem~\ref{prob:8}.

\begin{corollary}
\label{c:8}
Consider the setting of Algorithm~\ref{algo:8}. Suppose that
$\inf_{\nnn}\EE(\lambda_{\nS}(2-\lambda_{\nS}))>0$ and that, for
every $\nnn$, $\lambda_{\nS}$ is independent of
$\upsigma\brk{\brk[c]{x_{\iS,\pS}}_{\iii,0\leq\pS\leq\nS}
\cup\brk[c]{y_{\kS,\pS},
z_{\kS,\pS},v_{\kS,\pS}^*}_{\kkk,0\leq\pS\leq\nS}\cup
\brk[c]{I_{\nS},K_{\nS}}}$.
Suppose that \eqref{e:48} has a solution. Then there exists a
$\mathscr{V}$-valued random variable $\overline{s}$ such that
$s_{\nS}\to\overline{s}\;\Pas$
\end{corollary}
\begin{proof}
For every $\iS\in\II$, let $\proj_{\mathsf{E}_{\iS}}$ be the
projection operator onto $\mathsf{E}_{\iS}$ and let 
\begin{equation}
\mathsf{N}_{\mathsf{E}_{\iS}}\colon\HS_{\iS}\to 2^{\HS_{\iS}}
\colon\xS\mapsto
\begin{cases}
\menge{\mathsf{x}^*\in\HS_{\iS}}
{(\forall\mathsf{y}\in\mathsf{E}_{\iS})
\;\;\scal{\mathsf{y}-\mathsf{x}}{\mathsf{x}^*}\leq 0},
&\text{if}\;\;\mathsf{x}\in\mathsf{E}_{\iS};\\
\emp,&\text{otherwise}
\end{cases}
\end{equation}
be the normal cone operator of ${\mathsf{E}_{\iS}}$. Let us note
that \cite[Corollary~16.50(v)]{Livre1} yields
$(\forall\iS\in\{1,\ldots,\mS\})$
$\mathsf{N}_{\mathsf{E}_{\iS}
\cap \mathsf{F}_{\iS}}
=\mathsf{N}_{\mathsf{E}_{\iS}}+\mathsf{N}_{\mathsf{F}_{\iS}}$.
Further, in the Euclidean space $\HHS=\bigoplus_{\iii}\HS_{\iS}$,
let us introduce the primal problem 
\begin{equation}
\label{e:8p}
\text{find}\;\:\overline{\boldsymbol{\xS}}\in\HHS
\;\:\text{such that}\;\:
(\forall\iS\in\II)\;\;
0\in\mathsf{N}_{\mathsf{E}_{\iS}}\overline{\xS}_{\iS}
+\mathsf{N}_{\mathsf{F}_{\iS}}\overline{\xS}_{\iS}+
\LS_{\iS}^*\brk4{\CS\brk3{\sum_{\jJj}\LS_{\jS}
\overline{\xS}_{\jS}}},
\end{equation}
paired with the dual problem
\begin{equation}
\label{e:8d}
\text{find}\;\:
\big(\overline{\boldsymbol{\xS}}^*,\overline{\vS}^*\big)
\in\HHS\oplus\GS\;\:\text{such that}\;\:
\big(\exi\boldsymbol{\xS}\in\HHS\big)\;\;
\begin{cases}
(\forall\iS\in\II)\;\;
{-}\overline{\xS}_{\iS}^*-\LS_{\iS}^*\overline{\vS}^*\in
\mathsf{N}_{\mathsf{E}_{\iS}}\xS_{\iS}
\;\:\text{and}\;\:\overline{\xS}_{\iS}^*\in
\mathsf{N}_{\mathsf{F}_{\iS}}\xS_{\iS}\\
\overline{\vS}^*
=\CS\brk1{\sum_{\iii}\LS_{\iS}\xS_{\iS}}.
\end{cases}
\end{equation}
We denote by $\mathscr{P}$ and $\mathscr{D}$ the sets of
solutions to \eqref{e:8p} and \eqref{e:8d}, respectively.
We observe that the primal-dual problem
\eqref{e:8p}--\eqref{e:8d} is a special case
of Problem~\ref{prob:1} with
\begin{equation}
\label{e:33}
(\forall\iS\in\II)\quad\AS_{\iS}=\mathsf{N}_{\mathsf{E}_{\iS}},
\quad\CS_{\iS}=\QS_{\iS}=0,\quad\RS_{\iS}=0,
\quad\text{and}\quad\sS_{\iS}^*=0,
\end{equation}
and
\begin{equation}
\label{e:34}
(\forall\kS\in\KK)\quad
\begin{cases}
\GS_{\kS}=\HS_{\kS},\;\:
\BS_{\kS}^{\MM}=\mathsf{N}_{\mathsf{F}_{\kS}},\;\:
\BS_{\kS}^{\CC}=0\;\:\text{if}\;\:\kS\in\II;\\
\GS_{\mS+1}=\GS,\;\:
\BS_{\mS+1}^{\MM}=0,\;\:
\BS_{\mS+1}^{\CC}=\CS,\\
\DS_{\kS}^{\MM}=\{0\}^{-1},\;\:\BS_{\kS}^{\LL}=\DS_{\kS}^{\CC}
=\DS_{\kS}^{\LL}=0,\;\:\rS_{\kS}=0\\
(\forall\iS\in\II)\;\;\LS_{\kS\iS}=
\begin{cases}
\Id,&\text{if}\;\:\kS=\iS;\\
0,&\text{if}\;\:\kS\in\II\;\:\text{and}\;\:\kS\neq\iS;\\
\LS_{\iS},&\text{if}\;\:\kS=\mS+1.
\end{cases}
\end{cases}
\end{equation}
On the other hand, as in \cite[Equation~(119)]{Sadd22},
\begin{equation}
\label{e:6368}
(\forall\overline{\sS}\in\GS)\quad
\overline{\sS}\in\mathscr{V}\quad\Leftrightarrow\quad
(\exi\overline{\boldsymbol{\mathsf{x}}}\in\mathscr{P})\quad 
\overline{\sS}=\sum_{\iii}\LS_{\iS}\overline{\xS}_{\iS}.
\end{equation}
This implies that $\mathscr{P}\neq\emp$ and, in turn,
appealing to \cite[\protect{Proposition~1(v)[d]}]{Sadd22}, that
$\mathscr{D}\neq\emp$. Next, going back to 
\eqref{e:33}--\eqref{e:34}, we observe that, for every
$\upnu\in\RPP$,
\begin{equation}
(\forall\iS\in\II)\;\:
\mathsf{J}_{\upnu\AS_{\iS}}=\proj_{\mathsf{E}_{\iS}}
\quad\text{and}\quad
(\forall\kS\in\KK)\;\:
\mathsf{J}_{\upnu\DS_{\kS}^{\MM}}=0
\;\:\text{and}\;\:\mathsf{J}_{\upnu\BS_{\kS}^{\MM}}=
\begin{cases}
\proj_{\mathsf{F}_{\kS}},&\text{if}\;\:\kS\in\II;\\
\Id,&\text{if}\;\:\kS=\mS+1.
\end{cases}
\end{equation}
Therefore, after some elementary algebra, we arrive at the
conclusion that Algorithm~\ref{algo:8} is a realization of 
Algorithm~\ref{algo:1} implemented with the parameters
\begin{equation}
(\forall\nnn)(\forall\iS\in\II)(\forall\kS\in\KK)\quad
\upgamma_{\iS,\nS}=\upgamma,\;
\upmu_{\kS,\nS}=\upnu_{\kS,\nS}=\upmu,\;
\uptau_{\kS,\nS}=\uptau
\end{equation}
and applied to the primal-dual pair \eqref{e:8p}--\eqref{e:8d}.
In turn, Theorem~\ref{t:1}\ref{t:1iii} asserts that there exists a 
$\mathscr{P}$-valued random variable $\overline{\boldsymbol{x}}$
such that, for every $\iS\in\II$, 
$x_{\iS,\nS}\to\overline{x}_{\iS}\;\Pas$ 
We conclude that $s_{\nS}=\sum_{\iii}\LS_{\iS}{x}_{\iS,\nS}\to
\sum_{\iii}\LS_{\iS}\overline{x}_{\iS}$, where
$\overline{s}=\sum_{\iii}\LS_{\iS}\overline{x}_{\iS}$ is
$\mathscr{V}$-valued thanks to \eqref{e:6368}.
\end{proof}

\begin{example}
\label{ex:8}
In Problem~\ref{prob:8}, let $\rS\in\HS$ and 
$\CS\colon\sS\mapsto\sS-\rS$. Then, using the standard projection
theorem \cite[Theorem~3.16]{Livre1}, Problem~\ref{prob:8} amounts
to finding the projection of $\rS$ onto the composite Minkowski sum
$\sum_{\iS=1}^\mS\LS_{\iS}(\mathsf{E}_{\iS}\cap\mathsf{F}_{\iS})$.
Corollary~\ref{c:8} provides an almost surely convergent algorithm
to find such a best approximation which operates by random
activation of the individual projections onto the sets 
$(\mathsf{E}_{\iS})_{1\leq\iS\leq\mS}$ and
$(\mathsf{F}_{\iS})_{1\leq\iS\leq\mS}$. This algorithm is new, even
when applied to projecting onto a basic Minkowski
sum $\sum_{\iS=1}^\mS\mathsf{E}_{\iS}$.
\end{example}

\subsection{Application to multivariate minimization}

We consider a multivariate composite minimization problem.

\begin{problem}
\label{prob:2}
Let $(\HS_{\iS})_{\iii}$ and $(\GS_{\kS})_{\kkk}$
be finite families of Euclidean spaces, and let 
$\HHS=\bigoplus_{\iii}\HS_{\iS}$ and
$\GGS=\bigoplus_{\kkk}\GS_{\kS}$.
Denote by $\boldsymbol{\xS}=(\xS_{\iS})_{\iii}$
a generic element in $\HHS$. For every $\iS\in\II$ and every
$\kS\in\KK$, let $\mathsf{f}_{\iS}\in\upGamma_0(\HS_{\iS})$, let
$\upalpha_{\iS}\in\RPP$, let $\upvarphi_{\iS}\colon\HS_{\iS}\to\RR$
be convex and
differentiable with a $(1/\upalpha_{\iS})$-Lipschitzian gradient,
let $\mathsf{g}_{\kS}\in\upGamma_0(\GS_{\kS})$, let 
$\mathsf{h}_{\kS}\in\upGamma_0(\GS_{\kS})$,
and suppose that 
$\mathsf{L}_{\kS\iS}\colon\HS_{\iS}\to\GS_{\kS}$ is linear. 
In addition, let $\upchi\in\RP$ and let
$\upTheta\colon\HHS\to\RR$ be convex and differentiable with a
$\upchi$-Lipschitzian gradient. The objective is to
\begin{equation}
\label{e:2p}
\minimize{\boldsymbol{\xS}\in\HHS}{
\upTheta(\boldsymbol{\xS})
+\sum_{\iii}\brk1{\mathsf{f}_{\iS}(\xS_{\iS})
+\upvarphi_{\iS}(\xS_{\iS})}
+\sum_{\kkk}\brk1{\mathsf{g}_{\kS}
\infconv\mathsf{h}_{\kS}}
\brk3{\sum_{\iii}\LS_{\kS\iS}\xS_{\iS}}}.
\end{equation}
We denote by $\mathscr{P}$ the set of solutions to \eqref{e:2p}
and by $\nabla_{\!\iS}\,\upTheta$ the partial derivative of 
$\upTheta$ relative to $\HS_{\iS}$.
\end{problem}

\begin{algorithm}
\label{algo:6}
Consider the setting of Problem~\ref{prob:2} and suppose that
Assumption~\ref{a:1} is in force. Set 
$\upalpha=\min\{\upalpha_{\iS}\}_{\iii}$, let 
$\upphi\in\left]1/(4\upalpha),\pinf\right[$, let
$\upvarepsilon\in\left]0,\min\{1,1/(\upchi+\upphi)\}\right[$, and 
let $\uprho\in[2,\pinf[$. For every $(\iS,\kS)\in\II\times\KK$,
let $x_{\iS,0}\in L^2(\upOmega,\FE,\PP;\HS_{\iS})$ and 
$\{y_{\kS,0},z_{\kS,0},v_{\kS,0}^*\}\subset 
L^2(\upOmega,\FE,\PP;\GS_{\kS})$. Iterate
\begin{equation}
\label{e:long6}
\begin{array}{l}
\text{for}\;\nS=0,1,\ldots\\
\left|
\begin{array}{l}
\text{for every}\;\iS\in I_{\nS}\\
\left\lfloor
\begin{array}{l}
l_{\iS,\nS}^*=
\nabla_{\!\iS}\,\upTheta(\boldsymbol{x}_{\nS})
+\sum_{\kkk}\LS_{\kS\iS}^*v_{\kS,\nS}^*;\;
\upgamma_{\iS,\nS}\in[\upvarepsilon,1/(\upchi+\upphi)];\\
a_{\iS,\nS}=\prox_{\upgamma_{\iS,\nS}\mathsf{f}_{\iS}}\brk1{
x_{\iS,\nS}-\upgamma_{\iS,\nS}\brk1{l_{\iS,\nS}^*
+\nabla\upvarphi_{\iS}(x_{\iS,\nS})}};\;
a_{\iS,\nS}^*=\upgamma_{\iS,\nS}^{-1}(x_{\iS,\nS}
-a_{\iS,\nS})-l_{\iS,\nS}^*;\\
\xi_{\iS,\nS}=\|a_{\iS,\nS}-x_{\iS,\nS}\|^2;
\end{array}
\right.\\
\text{for every}\;\iS\in\II\smallsetminus I_{\nS}\\
\left\lfloor
\begin{array}{l}
a_{\iS,\nS}=a_{\iS,\nS-1};\;a_{\iS,\nS}^*=a_{\iS,\nS-1}^*;\;
\xi_{\iS,\nS}=\xi_{\iS,\nS-1};\\
\end{array}
\right.\\
\text{for every}\;\kS\in K_{\nS}\\
\left\lfloor
\begin{array}{l}
\upmu_{\kS,\nS}\in[\upvarepsilon,1/\upphi];\;
\upnu_{\kS,\nS}\in[\upvarepsilon,1/\upphi];\;
\uptau_{\kS,\nS}\in[\upvarepsilon,1/\upvarepsilon];\\
b_{\kS,\nS}=\prox_{\upmu_{\kS,\nS}\mathsf{g}_{\kS}}
\big(y_{\kS,\nS}+\upmu_{\kS,\nS}v_{\kS,\nS}^*\big);\;
d_{\kS,\nS}=\prox_{\upnu_{\kS,\nS}\mathsf{h}_{\kS}}\big(z_{\kS,\nS}
+\upnu_{\kS,\nS}v_{\kS,\nS}^*\big);\\
o_{\kS,\nS}=b_{\kS,\nS}+d_{\kS,\nS}
-\sum_{\iii}\LS_{\kS\iS}a_{\iS,\nS};\;
o_{\kS,\nS}^*=\uptau_{\kS,\nS}\big(
\sum_{\iii}\LS_{\kS\iS}x_{\iS,\nS}
-y_{\kS,\nS}-z_{\kS,\nS}\big)+v_{\kS,\nS}^*;\\
q_{\kS,\nS}^*=\upmu_{\kS,\nS}^{-1}(y_{\kS,\nS}-b_{\kS,\nS})
+v_{\kS,\nS}^*-o_{\kS,\nS}^*;\;
t_{\kS,\nS}^*=\upnu_{\kS,\nS}^{-1}(z_{\kS,\nS}-d_{\kS,\nS})
+v_{\kS,\nS}^*-o_{\kS,\nS}^*;\\
\eta_{\kS,\nS}=\|b_{\kS,\nS}-y_{\kS,\nS}\|^2
+\|d_{\kS,\nS}-z_{\kS,\nS}\|^2;\\
\end{array}
\right.\\
\text{for every}\;\kS\in\KK\smallsetminus K_{\nS}\\
\left\lfloor
\begin{array}{l}
b_{\kS,\nS}=b_{\kS,\nS-1};\;
d_{\kS,\nS}=d_{\kS,\nS-1};\;
o_{\kS,\nS}=b_{\kS,\nS}+d_{\kS,\nS}
-\sum_{\iii}\LS_{\kS\iS}a_{\iS,\nS};\\
o_{\kS,\nS}^*=o_{\kS,\nS-1}^*;\;
q_{\kS,\nS}^*=q_{\kS,\nS-1}^*;\;
t_{\kS,\nS}^*=t_{\kS,\nS-1}^*;\;
\eta_{\kS,\nS}=\eta_{\kS,\nS-1};\;
\end{array}
\right.\\
\text{for every}\;\iS\in\II\\
\left\lfloor
\begin{array}{l}
p_{\iS,\nS}^*=a_{\iS,\nS}^*
+\nabla_{\!\iS}\,\upTheta(\boldsymbol{a}_{\nS})
+\sum_{\kkk}\LS_{\kS\iS}^*o_{\kS,\nS}^*;
\end{array}
\right.\\[1mm]
\begin{aligned}
\Delta_{\nS}&=\textstyle
{-}(4\upalpha)^{-1}\big(\sum_{\iii}\xi_{\iS,\nS}
+\sum_{\kkk}\eta_{\kS,\nS}\big)
+\sum_{\iii}\scal{x_{\iS,\nS}-a_{\iS,\nS}}{p_{\iS,\nS}^*}\\
&\textstyle
\quad\;+\sum_{\kkk}
\big(\scal{y_{\kS,\nS}-b_{\kS,\nS}}{q_{\kS,\nS}^*}
+\scal{z_{\kS,\nS}-d_{\kS,\nS}}{t_{\kS,\nS}^*}
+\scal{o_{\kS,\nS}}{v_{\kS,\nS}^*-o_{\kS,\nS}^*}\big);
\end{aligned}\\[4mm]
\theta_{\nS}=
\dfrac{\mathsf{1}_{[\Delta_{\nS}>0]}\Delta_{\nS}}
{\sum_{\iii}\|p_{\iS,\nS}^*\|^2
+\sum_{\kkk}\big(
\|q_{\kS,\nS}^*\|^2+\|t_{\kS,\nS}^*\|^2+\|o_{\kS,\nS}\|^2\big)
+\mathsf{1}_{[\Delta_{\nS}\leq 0]}};\\
\lambda_{\nS}\in 
L^\infty(\upOmega,\FE,\PP;[\upvarepsilon,\uprho])\\
\end{array}
\right.\\[1mm]
\end{array}
\end{equation}
\[
\left\lfloor
\begin{array}{l}
\text{for every}\;\iS\in\II\\
\left\lfloor
\begin{array}{l}
x_{\iS,\nS+1}=x_{\iS,\nS}-\lambda_{\nS}\theta_{\nS}p_{\iS,\nS}^*;
\end{array}
\right.\\
\text{for every}\;\kS\in\KK\\
\left\lfloor
\begin{array}{l}
y_{\kS,\nS+1}=y_{\kS,\nS}-\lambda_{\nS}\theta_{\nS}q_{\kS,\nS}^*;\;
z_{\kS,\nS+1}=z_{\kS,\nS}-\lambda_{\nS}\theta_{\nS}t_{\kS,\nS}^*;\;
v_{\kS,\nS+1}^*=v_{\kS,\nS}^*-\lambda_{\nS}\theta_{\nS}o_{\kS,\nS}.
\end{array}
\right.\\[1mm]
\end{array}
\right.
\]
\end{algorithm}

\begin{corollary}
\label{c:2}
Consider the setting of Algorithm~\ref{algo:6}. Suppose that
$\inf_{\nnn}\EE(\lambda_{\nS}(2-\lambda_{\nS}))>0$ and that a 
Kuhn--Tucker point 
$(\widetilde{\boldsymbol{\xS}},\widetilde{\boldsymbol{\vS}}^*)
\in\HHS\times\GGS$ exists, that is, 
\begin{equation}
\label{e:18}
(\forall\iS\in\II)(\forall\kS\in\KK)\quad
\begin{cases}
{-}\displaystyle\sum_{\jjJ}\LS_{\jS\iS}^*
\widetilde{\mathsf{v}}_{\jS}^*\in
\uppartial\mathsf{f}_{\iS}(\widetilde{\xS}_{\iS})+
\nabla\upvarphi_{\iS}(\widetilde{\xS}_{\iS})
+\nabla_{\!{\iS}}\,\upTheta(\widetilde{\boldsymbol{\xS}});\\
\phantom{-}
\displaystyle\sum_{\jJj}\LS_{{\kS}{\jS}}\widetilde{\xS}_{\jS}\in
\uppartial\mathsf{g}_{\kS}^*(\widetilde{\mathsf{v}}_{\kS}^*)
+\uppartial\mathsf{h}_{\kS}^*(\widetilde{\mathsf{v}}_{\kS}^*).
\end{cases}
\end{equation}
Then there exists a $\mathscr{P}$-valued random variable 
$\overline{\boldsymbol{x}}$ such that, for every $\iS\in\II$,
$x_{\iS,\nS}\to\overline{x}_{\iS}\;\Pas$
\end{corollary}
\begin{proof}
For every $\iS\in\II$ and every $\kS\in\KK$, set
\begin{equation}
\label{e:2020var}
\begin{cases}
\upalpha_{\iS}^{\CC}=\upalpha_{\iS};\;
\upbeta_{\kS}^{\CC}=\updelta_{\kS}^{\CC}=\pinf;\; 
\upalpha_{\iS}^{\LL}=\upbeta_{\kS}^{\LL}=\updelta_{\kS}^{\LL}=0;\\
\AS_{\iS}=\partial\mathsf{f}_{\iS};\;
\CS_{\iS}=\nabla\upvarphi_{\iS};\;
\QS_{\iS}={0};\;\RS_{\iS}=\nabla_{\!\iS}\,\upTheta;
\;\sS_{\iS}^*={0};\\
\BS_{\kS}^{\MM}=\partial\mathsf{g}_{\kS};\;
\DS_{\kS}^{\MM}=\partial\mathsf{h}_{\kS};\;
\BS_{\kS}^{\LL}=\BS_{\kS}^{\CC}=
\DS_{\kS}^{\CC}=\DS_{\kS}^{\LL}={0};\;
\rS_{\kS}={0}.
\end{cases}
\end{equation}
Thus, it follows from \cite[Theorem~20.25, Corollary~18.17, and
Proposition~17.7]{Livre1} that the problem
\begin{equation}
\label{e:2020p}
\text{find}\;\:\overline{\boldsymbol{{\xS}}}\in\HHS\;\:
\text{such that}\;\;
(\forall\iS\in\II)\;\; 0\in
\AS_{\iS}\overline{\xS}_{\iS}+\CS_{\iS}\overline{\xS}_{\iS}
+\RS_{\iS}\overline{\boldsymbol{\xS}}
+\Sum_{\kkk}\LS_{\kS\iS}^*
\brk3{\brk2{\BS_{\kS}^{\MM}\infconv
\DS_{\kS}^{\MM}}
\brk3{\Sum_{\jJj}\LS_{\kS\jS}\overline{\xS}_{\jS}}}
\end{equation}
together with its dual is an instance of Problem~\ref{prob:1}.
Moreover, by \cite[Example~23.3]{Livre1},
Algorithm~\ref{algo:8} is an application of Algorithm~\ref{algo:1}
to \eqref{e:2020var}--\eqref{e:2020p}. Further, thanks to
\cite[Proposition~1(iv)]{Sadd22}, the existence of a Kuhn--Tucker
point in \eqref{e:18} guarantees the existence of dual solutions.
Thus, we invoke Theorem~\ref{t:1}\ref{t:1iii} to obtain the
existence of a random variable $\overline{\boldsymbol{x}}$ that
solves \eqref{e:2020p} $\Pas$ and such that, for every $\iS\in\II$,
$x_{\iS,\nS}\to\overline{x}_{\iS}\;\Pas$ Finally, if follows
form Fermat's rule \cite[Theorem~16.3]{Livre1}
and \cite[Equation~(135)]{Sadd22} that every solution to
\eqref{e:2020p} is a solution to Problem~\ref{prob:2}. We conclude
that $\overline{\boldsymbol{x}}$ is a $\mathscr{P}$-valued random
variable.
\end{proof}

\end{document}